\documentclass[11pt]{amsart}
\usepackage{amsmath}
\usepackage{amssymb}
\usepackage{amsthm}
\usepackage{latexsym}
\usepackage{hyperref}
\usepackage{enumerate}
\usepackage[all]{xy}

\newtheorem{thm}{Theorem}[section]
\newtheorem{cor}[thm]{Corollary}
\newtheorem{lem}[thm]{Lemma}
\newtheorem{prop}[thm]{Proposition}

\newtheorem{thmintro}{Theorem}

\theoremstyle{definition}

\newtheorem{ex}[thm]{Example}

\providecommand{\norm}[1]{\left\| #1 \right\|}
\newcommand{\enuma}[1]{\begin{enumerate}[\textup{(}a\textup{)}] {#1} \end{enumerate}}

\newcommand{\mb}{\mathbf}

\newcommand{\mh}{\mathbb}
\newcommand{\mr}{\mathrm}
\newcommand{\mc}{\mathcal}
\newcommand{\mf}{\mathfrak}

\newcommand{\ts}{\textstyle}

\newcommand{\isom}{\xrightarrow{\;\sim\;}}
\newcommand{\mosi}{\xleftarrow{\;\sim\;}}

\newcommand{\Z}{\mathbb Z}
\newcommand{\Q}{\mathbb Q}
\newcommand{\R}{\mathbb R}
\newcommand{\C}{\mathbb C}
\newcommand{\ep}{\epsilon}

\newcommand{\af}{\mr{aff}}

\newcommand{\un}{\mr{un}}
\newcommand{\rs}{\mr{rs}}

\newcommand{\inp}[2]{\langle #1 \,,\, #2 \rangle}

\newcommand{\matje}[4]{\left(\begin{smallmatrix} #1 & #2 \\ 
#3 & #4 \end{smallmatrix}\right)}

\def\Hom{{\rm Hom}}
\def\End{{\rm End}}

\def\Irr{{\rm Irr}}

\def\Mod{{\rm Mod}}

\def\ind{{\rm ind}}
\def\Aut{{\rm Aut}}

\begin{document}

\title{Topological K-theory of affine Hecke algebras}
\author[M. Solleveld]{Maarten Solleveld}
\address{IMAPP, Radboud Universiteit Nijmegen, Heyendaalseweg 135, 
6525AJ Nijmegen, the Netherlands}
\email{m.solleveld@science.ru.nl}
\date{\today}
\thanks{The author is supported by a NWO Vidi grant "A Hecke algebra approach to the 
local Langlands correspondence" (nr. 639.032.528).}
\maketitle
\begin{abstract}
Let $\mc H (\mc R,q)$ be an affine Hecke algebra with a positive parameter function $q$.
We are interested in the topological K-theory of its $C^*$-completion $C_r^* (\mc R,q)$. 
We will prove that $K_* (C_r^* (\mc R,q))$ does not depend on the parameter $q$, solving a 
long-standing conjecture of Higson and Plymen. For this we use representation theoretic methods, 
in particular elliptic representations of Weyl groups and Hecke algebras.

Thus, for the computation of these K-groups it suffices to work out the case $q=1$. These 
algebras are considerably simpler than for $q \neq 1$, just crossed products of commutative algebras 
with finite Weyl groups. We explicitly determine $K_* (C_r^* (\mc R,q))$ for all classical 
root data $\mc R$. This will be useful to analyse the K-theory
of the reduced $C^*$-algebra of any classical $p$-adic group. 

For the computations in the case $q=1$ we study the more general situation of a finite group 
$\Gamma$ acting on a smooth manifold $M$. We develop a method to calculate the K-theory of the
crossed product $C(M) \rtimes \Gamma$. In contrast to the equivariant Chern character of
Baum and Connes, our method can also detect torsion elements in these K-groups.
\end{abstract}

\tableofcontents

\section*{Introduction}

Affine Hecke algebras can be realized in two completely different ways. On the one
hand, they are deformations of group algebras of affine Weyl groups, and on the 
other hand they appear as subalgebras of group algebras of reductive $p$-adic groups.
Via the second interpretation affine Hecke algebras (AHAs) have proven very useful
in the representation theory of such groups. This use is in no small
part due to their explicit construction in terms of root data, which makes them
amenable to concrete calculations. 

This paper is motivated by our desire to understand and compute the (topological)
K-theory of the reduced $C^*$-algebra $C_r^* (G)$ of a reductive $p$-adic group $G$.
This is clearly related to the representation theory of $G$. For instance, when $G$
is semisimple, every discrete series $G$-representation gives rise to a one-dimensional
direct summand in the K-theory of $C_r^* (G)$. 

The problem can be transferred to AHAs in the following way. By the Bernstein
decomposition, the Hecke algebra of $G$ can be written as a countable direct sum of
two-sided ideals:
\[
\mc H (G) = \bigoplus\nolimits_{\mf s \in \mf B (G)} \mc H (G)^{\mf s} .
\]
Borel \cite{Bor} and Iwahori--Matsumoto \cite{IwMa} have shown that
one particular summand, say $\mc H (G)^{IM}$, is Morita equivalent to an AHA, say
$\mc H (\mc R,q)^{IM}$. It is expected that all other summands $\mc H (G)^{\mf s}$ 
are also Morita equivalent to AHAs, or to closely related algebras. Indeed, this has 
been proven in many cases, see \cite[\S 2.4]{ABPS3} for an overview. 

The reduced $C^*$-algebra of $G$ is a completion of $\mc H (G)$, and it admits an
analogous Bernstein decomposition
\[
C_r^* (G) = \bigoplus\nolimits_{\mf s \in \mf B (G)} C_r^* (G)^{\mf s} , 
\]
where $C_r^* (G)^{\mf s}$ is the closure of $\mc H (G)^{\mf s}$ in $C_r^* (G)$.
By \cite{BHK} the Morita equivalence $\mc H (G)^{IM} \sim_M \mc H (\mc R,q)$
extends to a Morita equivalence between $C_r^* (G)^{IM}$ and the natural $C^*$-completion
of $\mc H (\mc R,q)^{IM}$. Again it can be expected that every summand $C_r^* (G)^{\mf s}$ 
is Morita equivalent to the $C^*$-completion $C_r^* (\mc R,q)^{\mf s}$ of some AHA 
$\mc H (\mc R,q)^{\mf s}$. However, this is currently not yet proven in several cases where 
the Morita equivalence is known on the algebraic level. We will return to this issue in a 
subsequent paper. Assuming it for the moment, we get 
\[
K_* (C_r^* (G)) \cong \bigoplus\nolimits_{\mf s \in \mf B (G)} K_* (C_r^* (\mc R,q)^{\mf s}) .
\]
The left hand side figures in the Baum--Connes conjecture for reductive $p$-adic groups 
\cite{BCH}. For applications to the Baum--Connes conjecture for algebraic groups over local fields, 
it would be useful to understand $K_* (C_r^* (G))$ better, in particular its torsion subgroup.
Namely, from the work of Kasparov \cite{Kas} it is known that for many groups $G$ the Baum--Connes
assembly map is injective, and that its image is a direct summand of $K_* (C_r^* (G))$. There exist
methods \cite[\S 3.4]{Sol2} which enable one to prove that the assembly map becomes an isomorphism
after tensoring its domain and range by $\Q$, but which say little about the torsion elements in
the K-groups. If one knew in advance that $K_* (C_r^* (G))$ is torsion-free, then one 
could prove instances of the Baum--Connes conjecture with such methods.

To construct an affine Hecke algebra, we use a root datum $\mc R$ in a lattice $X$. These 
give a Weyl group $W = W(\mc R)$ and an extended affine Weyl group $W^e = X \rtimes W$. 
As parameters we take a tuple of nonzero complex numbers $q = (q_i )_i$. The AHA 
$\mc H (\mc R,q)$ is a deformation of the group algebra $\C [W^e]$, in the following sense: 
as a vector space it is $\C [W^e]$, with a multiplication rule depending algebraically on $q$,
such that $\mc H (\mc R,1) = \C [W^e]$. See Paragraph \ref{par:AHA} for the precise definition.
To get a nice $C^*$-completion $C_r^* (\mc R,q)$, we must assume that $q$ is positive, that is, 
$q_i \in \R_{>0}$ for all $i$. For $q=1$ the $C^*$-completion can be described easily:
\[
C_r^* (\mc R,1) = C_r^* (W^e) = C(T_\un) \rtimes W , 
\]
where $T_\un = \Hom_\Z (X,S^1)$ is a compact torus.

All AHAs obtained from reductive $p$-adic groups $G$ have rather special parameters:
there are $n_i$ such that $q_i = p^{n_i}$, where $p$ is the characteristic of the local
nonarchimedean field underlying $G$. Thus the realization of AHAs via root data admits
more parameters than the realization as subalgebras of $\mc H (G)$. In particular the 
algebras $\mc H (\mc R,q)$ admit continuous parameter deformations, whereas the AHAs from
reductive $p$-adic groups do not, since the prime powers $p^{n/2}$ are discrete in $\R_{>0}$.

In fact, for fixed $\mc R$ the family $C_r^* (\mc R,q)$, with varying positive $q$, form
a continuous field of $C^*$-algebras. For a given $q \neq 1$ we have the half-line of
parameters $q^\epsilon = (q_i^\epsilon )_i$ with $\epsilon \in \R_{\geq 0}$. It is known from
\cite[Theorem 4.4.2]{SolAHA} that there exists a family of $C^*$-homomorphisms 
\[
\zeta_\epsilon : C_r^* (\mc R,q^\epsilon) \to C_r^* (\mc R,q) \qquad \epsilon \geq 0 , 
\]
such that $\zeta_\epsilon$ is an isomorphism for all $\epsilon > 0$ and depends continuously 
on $\epsilon \in \R_{\geq 0}$.
Via a general deformation principle, this yields a canonical homomorphism
\begin{equation}\label{eq:K}
K_* (C_r^* (W^e)) = K_* (C_r^* (\mc R,q^0)) \to K_* (C_r^* (\mc R,q)) . 
\end{equation}
Loosely speaking, the construction goes as follows. Take a projection $p_0$ (or a unitary
$u_0$) in a matrix algebra $M_n (C_r^* (W^e)) = M_n (C_r^* (\mc R,q^0))$. For $\epsilon > 0$
small, we can apply holomorphic functional calculus to $p_0$ to produce a new projection
$p_\epsilon \in M_n (C_r^* (\mc R,q^\epsilon))$ (or a unitary $u_\epsilon$). Then \eqref{eq:K}
sends $[p_0] \in K_0 (C_r^* (\mc R,q^0))$ (respectively $u_0 \in K_1 (C_r^* (\mc R,q^0))$) 
to the image of $p_\epsilon$ (respectively $u_\epsilon$) under the isomorphism 
$M_n (C_r^* (\mc R,q^\epsilon)) \cong M_n (C_r^* (\mc R,q))$. 

Actually, more is true, by \cite[Lemma 5.1.2]{SolAHA} the map $K_* (\zeta_0)$ equals \eqref{eq:K}.
Furthermore, by \cite[Theorem 5.1.4]{SolAHA} $\zeta_0$ induces an isomorphism
\[
K_* (C_r^* (\mc R,q^0)) \otimes_\Z \C \to K_* (C_r^* (\mc R,q)) \otimes_\Z \C . 
\]
In view of the aforementioned relation with the Baum--Connes conjecture for $p$-adic groups,
we also want to understand the torsion parts of these K-groups. We will prove:

\begin{thmintro}\label{thm:1} \textup{[see Theorem \ref{thm:2.2}]} \\
The map \eqref{eq:K} is a canonical isomorphism
\[
K_* (C_r^* (\mc R,1)) \to K_* (C_r^* (\mc R,q)) .
\]
\end{thmintro}

This theorem was conjectured first by Higson and Plymen (see \cite[6.4]{Ply2} and \cite[6.21]{BCH}), 
at least when all parameters $q_i$ are equal. It is similar to the Connes--Kasparov 
conjecture for Lie groups, see \cite[Sections 4--6]{BCH} for more background. Independently Opdam 
\cite[Section 1.0.1]{Opd-Sp} conjectured Theorem \ref{thm:1} for unequal parameters. 

Unfortunately it is unclear how Theorem \ref{thm:1} could be proven by purely noncommutative
geometric means. The search for an appropriate technique was a major drive behind the 
author's PhD project (2002--2006), and partial results appeared already in his
PhD thesis \cite{SolThesis}. At that time, we hoped to derive representation consequences
from a K-theoretic proof of Theorem \ref{thm:1}. But so far, such a proof remains elusive.

In the meantime, substantial progress has been made in the representation theory of Hecke
algebras, see in particular \cite{OpSo1,CiOp1,CiOpTr}. This enables us to turn things around
(compared to 2004), now we can use representation theory to study the K-theory of 
$C_r^* (\mc R,q)$. 

Given an algebra or group $A$, let $\mr{Mod}_f (A)$ be the category of 
finite length $A$-modules, and let $R_\Z (A)$ be the Grothendieck group thereof. 
We deduce Theorem \ref{thm:1} from:

\begin{thmintro}\label{thm:2} \textup{[see Theorem \ref{thm:1.5}]} \\
The map $\mr{Mod}_f (C_r^* (\mc R,q)) \to \Mod (C_r^* (W^e)) : \pi \mapsto \pi \circ \zeta_0$
induces $\Z$-linear bijections
\begin{align*}
& R_\Z (C_r^* (\mc R,q)) \to R_\Z (C_r^* (W^e)) ,\\ 
& R_\Z (\mc H (\mc R,q)) \to R_\Z (W^e) .
\end{align*}
\end{thmintro}

A substantial part of the proof of Theorem \ref{thm:2} boils down to representations of
the finite Weyl group $W$. Following Reeder \cite{Ree}, we study the group $\overline{R_\Z}(W)$
of elliptic representations, that is, $R_\Z (W)$ modulo the subgroup spanned by all representations
induced from proper parabolic subgroups of $W$. First we show that $\overline{R_\Z}(W)$ is always
torsion-free (Theorem \ref{thm:1.6}). Then we compare it with the analogous group of elliptic
representations of $\mc H (\mc R,q)$, which leads to Theorem \ref{thm:2}.\\

Having established the general framework, we set out to compute $K_* (C_r^* (\mc R,q))$
explicitly, for some root data $\mc R$ associated to well-known groups. By Theorem \ref{thm:1}, 
we only have to consider one $q$ for each $\mc R$. In most examples, the easiest is to take $q=1$.
Then we must determine 
\[
K_r (C_r^* (\mc R,1)) = K_* (C (T_\un) \rtimes W) \cong K^*_W (T_\un) ,
\]
where the right hand side denotes the $W$-equivariant K-theory of the compact Hausdorff 
space $T_\un$. Let $T_\un /\!/ W$ be the extended quotient. Of course, the equivariant
Chern character from \cite{BaCo} gives a natural isomorphism
\[
K^*_W (T_\un) \otimes_\Z \C \to H^* (T_\un /\!/ W ;\C) .
\]
But this does not suffice for our purposes, because we are particularly interested in the 
torsion subgroup of $K_W^* (T_\un)$.  Remarkably, that appears to be quite difficult to 
determine, already for cyclic groups acting on tori \cite{LaLu}.
Using equivariant cohomology, we develop a technique to facilitate the computation of 
$K_* (C(\Sigma) \rtimes W)$ for any finite group $W$ acting smoothly on a manifold $\Sigma$. 
With extra conditions it can be made more explicit:

\begin{thmintro}\label{thm:3} \textup{[see Theorem \ref{thm:2.3}]} \\
Suppose that every isotropy group $W_t \; (t \in \Sigma)$ is a Weyl group, and that
$H^* (\Sigma /\!/ W ; \Z)$ is torsion-free. Then 
\[
K_* (C(\Sigma) \rtimes W) \cong H^* (\Sigma /\!/ W ;\Z) . 
\]
\end{thmintro}

We note that $H^* (\Sigma /\!/ W ; \Z)$ can be computed relatively easily. Theorem \ref{thm:3}
can be applied to all classical root data, and to some others as well. Let us summarise 
the outcome of our computations.

\begin{thmintro}\label{thm:4}
Let $\mc R$ be a root datum of type $GL_n, SL_n, PGL_n, SO_n, Sp_{2n}$ or $G_2$. Let 
$q$ be any positive parameter function for $\mc R$. Then $K_* (C_r^* (\mc R,q))$ is a 
free abelian group, whose rank is given explicitly in Section \ref{sec:exa}.
\end{thmintro}

Whether or not torsion elements can pop up in $K_* (C_r^* (\mc R,q))$ for other root data
remains to be seen. In view of our results it does not seem very likely, but we do not have
a general principle to rule it out.

\section{Representation theory}

\subsection{Weyl groups} \
\label{par:Weyl}

In this first paragraph will show that the representation ring $R_\Z (W)$ of any finite Weyl
group $W$ is the direct sum of two parts: the subgroup spanned by representations induced
from proper parabolic subgroups, and an elliptic part $\overline{R_\Z} (W)$. We exhibit a 
$\Z$-basis of $\overline{R_\Z} (W)$ in terms of the Springer correspondence. These results 
rely mainly on case-by-case considerations in complex simple groups.

Let $\mf a$ be a finite dimensional real vector space and let $\mf a^*$ be its dual. 
Let $Y \subset \mf a$ be a lattice and 
$X = \mr{Hom}_\Z (Y,\Z) \subset \mf a^*$ the dual lattice. Let
\begin{equation}\label{eq:1.37}
\mc R = (X, R, Y ,R^\vee ,\Delta) .
\end{equation}
be a based root datum. Thus $R$ is a reduced root system in $X ,\, R^\vee \subset Y$ 
is the dual root system, $\Delta$ is a basis of $R$ and the set of positive roots is denoted $R^+$.
Furthermore we are given a bijection $R \to R^\vee ,\: \alpha \mapsto \alpha^\vee$ such 
that $\inp{\alpha}{\alpha^\vee} = 2$ and such that the corresponding reflections
$s_\alpha : X \to X$ (resp. $s^\vee_\alpha : Y \to Y$) stabilize $R$ (resp. $R^\vee$).
We do not assume that $R$ spans $\mf a^*$. The reflections $s_\alpha$ generate the Weyl group 
$W = W (R)$ of $R$, and 
$S_\Delta := \{ s_\alpha \mid \alpha \in \Delta \}$ is the collection of simple reflections. 

For a set of simple roots $P \subset \Delta$ we let $R_P$ be the root system they generate,
and we let $W_P = W(R_P)$ be the corresponding parabolic subgroup of $W$.

Let $R_\Z (W)$ be the Grothendieck group of the category of finite dimensional complex
$W$-representations, and write $R_\C (W) = \C \otimes_\Z R_\Z (W)$. For any $P \subset \Delta$
the induction functor $\ind_{W_P}^W$ gives linear maps $R_\Z (W_P) \to R_\Z (W)$ and
$R_\C (W_P) \to R_\C (W)$. In this subsection we are mainly interested in the abelian group of 
``elliptic $W$-representations"
\begin{equation} \label{eq:1.24}
\overline{R_\Z} (W) = R_\Z (W) \big/ 
\sum\nolimits_{P \subsetneq \Delta} \ind_{W_P}^W ( R_\Z (W_P) ) .
\end{equation}
In the literature \cite{Ree,CiOpTr} one more often encounters the vector space
\[
\overline{R_\C} (W) = R_\C (W) \big/ 
\sum\nolimits_{P \subsetneq \Delta} \ind_{W_P}^W ( R_\C (W_P) ) .
\]
Recall that an element $w \in W$ is called elliptic if it fixes only the zero element
of $\mr{Span}_\R (R)$, or equivalently if it does not belong to any proper parabolic
subgroup of $W$. It was shown in \cite[Proposition 2.2.2]{Ree} that $\overline{R_\C}(W)$
is naturally isomorphic to the space of all class functions on $W$ supported on elliptic
elements. In particular $\dim_\C \overline{R_\C} (W)$ is the number of elliptic
conjugacy classes in $W$.

In \cite{CiOpTr} $\overline{R_\Z}(W)$ is defined as the subgroup of $\overline{R_\C} (W)$
generated by the $W$-representations. So in that work it is by definition a lattice.
If $\overline{R_\Z} (W)$ (in our sense) is torsion-free, then it can be identified with 
the subgroup of $\overline{R_\C} (W)$ to which it is naturally mapped.  
For our purposes it will be essential to stick to the definition \eqref{eq:1.24} and to
use some results from \cite{CiOpTr}. Therefore we want to prove that \eqref{eq:1.24}
is always a torsion-free group.

In the analysis we will make ample use of Springer's construction of representations of
Weyl groups, and of Reeder's results \cite{Ree}. Let $G$ be a connected
reductive complex group with a maximal torus $T$, such that $R \cong R (G,T)$
and $W \cong W(G,T)$. For $u \in G$ let $\mc B^u = \mc B_G^u$ be the complex variety of Borel 
subgroups of $G$ containing $u$. The group $Z_G (u)$ acts on $\mc B^u$ by conjugation, and that
induces an action of $A_G (u) := \pi_0 (Z_G (u) / Z(G))$ on the cohomology of $\mc B^u$.
For a pair $(u,\rho)$ with $u \in G$ unipotent and $\rho \in \Irr (A_G (u))$ we define
\begin{equation}\label{eq:1.47}
\begin{aligned}
& H (u,\rho) = \Hom_{A_G (u)} \big( \rho, H^* (\mc B^u ;\C) \big) , \\
& \pi (u,\rho) = \Hom_{A_G (u)} \big( \rho, H^{\mr{top}} (\mc B^u ;\C) \big) , 
\end{aligned}
\end{equation}
where top indicates the highest dimension in which the cohomology is nonzero, namely
the dimension of $\mc B^u$ as a real variety. Let us call $\rho$ geometric if 
$\pi (u,\rho) \neq 0$. Springer \cite{Spr} proved that 
\begin{itemize}
\item $W \times A_G (u)$ acts naturally on $H^i (\mc B^u ;\C)$, for each $i \in \Z_{\geq 0}$,
\item $\pi (u,\rho)$ is an irreducible $W$-representation whenever it is nonzero, 
\item this gives a bijection between $\Irr (W)$ and the $G$-conjugacy classes of pairs 
$(u,\rho)$ with $u \in G$ unipotent and $\rho \in \Irr (A_G (u))$ geometric.
\end{itemize}
It follows from a result of Borho and MacPherson \cite{BoMa} that the $W$-representations 
$H (u,\rho)$, parametrized by the same data $(u,\rho)$, also form a basis of $R_\Z (W)$, see 
\cite[Lemma 3.3.1]{Ree}. Moreover $\pi(u,\rho)$ appears with multiplicity one in $H(u,\rho)$.

\begin{ex}\label{ex:1.9}
\begin{itemize}
\item Type $A$. Only the $n$-cycles in $W = S_n$ are elliptic, and they form one
conjugacy class. The only quasidistinguished unipotent class in $GL_n (\C)$ is
the regular unipotent class. Then $A_{GL_n (\C)}(u_{\mr{reg}}) = 1$ for every
regular unipotent element $u_{\mr{reg}}$ and $H(u_{\mr{reg}},\mr{triv}) = 
H^0 (\mc B^{u_{\mr{reg}}};\C)$ is the sign representation of $S_n$ (with our
convention for the Springer correspondence).
\item Types $B$ and $C$. The elliptic classes in $W(B_n) = W(C_n) \cong S_n \rtimes (\Z / 2\Z)^n$
are parametrized by partitions of $n$. We will write them down explicitly
as $\sigma (\emptyset,\lambda)$ with $\lambda \vdash n$ in \eqref{eq:6.52}.
\item Type $D$. The elliptic classes in $W(D_n) = S_n \rtimes (\Z / 2\Z)^n_{ev}$
are precisely the elliptic classes of $W(B_n)$ that are contained in $W(D_n)$.
They can be parametrized by partitions $\lambda \vdash n$ such that $\lambda$ has
an even number of terms.
\item Type $G_2$. There are three elliptic classes in $W(G_2) = D_6$: the rotations of order 
two, of order three and of order six. The quasidistinguished unipotent classes in
$G_2 (\C)$ are the regular and the subregular class. 

We have $A_G (u_{\mr{reg}}) = 1$ and $H(u_{\mr{reg}},\mr{triv}) = \pi(u_{\mr{reg}},\mr{triv})$ 
is the sign re\-pre\-sentation of $D_6$.
For $u$ subregular $A_G (u) \cong S_3$, and the sign representation of $A_G (u)$ is not geometric.
For $\rho$ the two-dimensional irreducible representation of $A_G (u)$, 
$\pi (u,\rho) = H(u,\rho)$ is the character of $W(G_2)$ which is 1 on the reflections for 
long roots and $-1$ on the reflections for short roots.
Furthermore $\pi (u,\mr{triv})$ is the standard two-dimensional representation of $D_6$ and
$H(u,\mr{triv})$ is the direct sum of $\pi (u,\mr{triv})$ and the sign representation.
\end{itemize}
\end{ex}

For a subset $P \subset \Delta$ let $G_P$ be the standard Levi subgroup of $G$ generated 
by $T$ and the root subgroups for roots $\alpha \in R_P$. The irreducible representations
of $W_P = W(G_P,T)$ are parametrized by $G_P$-conjugacy classes of pairs $(u_P,\rho_P)$
with $u_P \in G_P$ unipotent and $\rho_P \in \Irr (A_{G_P}(u_P))$ geometric, and the 
$W_P$-representations $H_P (u_P,\rho_P)$ form another basis of $R_\Z (W_P$.

Recall from \cite[\S 3.2]{Ree} that $A_{G_P}(u_P)$ can be regarded as a subgroup 
of $A_G (u_P)$. By \cite[Proposition 2.5 and 6.2]{Kat} 
\begin{equation}\label{eq:1.35}
\ind_{W_P}^W \big( H^i (\mc B_{G_P}^{u_P} ;\C) \big) \cong
H^i (\mc B^{u_P};\C) \qquad \text{as } W \times A_G (u_P)\text{-representations.}
\end{equation}
It follows that for any $(u_P,\rho_P)$ as above there are natural isomorphisms
\begin{multline}\label{eq:1.25}
\ind_{W_P}^W \big( H_P (u_P,\rho_P) \big) \cong \Hom_{A_{G_P}(u_P)} \big( \rho_P,  
H^* (\mc B^{u_P};\C)  \big) \\
\cong \bigoplus\nolimits_{\rho \in \Irr (A_G (u_P))}
\Hom_{A_{G_P}(u_P)} (\rho_P,\rho) \otimes H (u_P,\rho) .
\end{multline}
For a unipotent conjugacy class $\mc C \subset G$ and $P \subset \Delta$, let 
$R_\Z (W_P,\mc C)$ be the subgroup of $R_\Z (W_P)$ generated by the $H_P (u_P,\rho_P)$
with $u_P \in G_P \cap \mc C$ and $\rho_P \in \Irr (A_{G_P}(u_P))$. (Notice that
$G_P \cap \mc C$ can consist of zero, one or more conjugacy classes.) In view of
\eqref{eq:1.25} we can define
\[
\overline{R_\Z}(W,\mc C) = R_\Z (W,\mc C) \big/ 
\sum\nolimits_{P \subsetneq \Delta} \ind_{W_P}^W ( R_\Z (W_P ,\mc C) ) . 
\]
We obtain a decomposition as in \cite[\S 3.3]{Ree}:
\begin{equation}\label{eq:1.26}
\overline{R_\Z}(W) = \bigoplus\nolimits_{\mc C} \overline{R_\Z}(W,\mc C) .
\end{equation}
Following Reeder \cite{Ree}, we also define elliptic representation theories for
the component groups $A_G (u)$. For $u,u_P \in \mc C$ the groups $A_G (u)$ and
$A_G (u_P)$ are isomorphic. In general the isomorphism is not natural, but it
is canonical up to inner automorphisms. This gives a natural isomorphism 
$R_\Z (A_G (u)) \cong R_\Z (A_G (u_P))$, which enables us to write
\begin{equation}\label{eq:1.27}
\overline{R_\Z}(A_G (u)) = R_\Z (A_G (u)) \big/ \sum_{P \subsetneq \Delta , u_P \in \mc C \cap G_P}
\ind_{A_{G_P}(u_P)}^{A_G (u_P)} \big( R_\Z (A_{G_P}(u_P)) \big) .
\end{equation}
For any $u_P, u'_P \in \mc C \cap G_P$ there is a natural isomorphism
\[
\ind_{A_{G_P}(u_P)}^{A_G (u_P)} \big( R_\Z (A_{G_P}(u_P)) \big) \cong
\ind_{A_{G_P}(u'_P)}^{A_G (u'_P)} \big( R_\Z (A_{G_P}(u'_P)) \big) ,
\]
so on the right hand side of \eqref{eq:1.27} it actually suffices to use only one $u_P$
whenever $\mc C \cap G_P$ is nonempty.

Let $R^\circ_\Z (A_G (u))$ be the subgroup of $R_\Z (A_G (u))$ generated by the geometric
irreducible $A_G (u)$-representations. By \cite[\S 10]{Ree} 
\[
\ind_{A_G (u_P)}^{A_G (u)} \big( R^\circ_\Z (A_{G_P} (u_P)) \big) \subset R^\circ_\Z (A_G (u)).
\]
Using this we can define
\[
\overline{R^\circ_\Z}(A_G (u)) = R^\circ_\Z (A_G (u)) \big/ \sum_{P \subsetneq \Delta , 
u_P \in \mc C \cap G_P} \ind_{A_{G_P}(u_P)}^{A_G (u_P)} \big( R^\circ_\Z (A_{G_P}(u_P)) \big) . 
\]
It follows from \eqref{eq:1.25} that every $\rho_P \in \Irr (A_{G_P} (u_P))$ which appears in 
$\rho$ is itself geometric. Hence the inclusions $R^\circ_\Z (A_{G_P}(u_P)) \to
R_\Z (A_{G_P}(u_P))$ induce an injection
\begin{equation}\label{eq:1.28}
\overline{R^\circ_\Z}(A_G (u)) \to \overline{R_\Z}(A_G (u)) . 
\end{equation}
By \cite[Proposition 3.4.1]{Ree} the maps $\rho_P \mapsto \Hom_{A_{G_P}(u_P)} \big( \rho_P,  
H^* (\mc B^{u_P};\C)$ for $P \subset \Delta$ induce a $\Z$-linear bijection
\begin{equation}\label{eq:1.29}
\overline{R^\circ_\Z}(A_G (u)) \to \overline{R_\Z}(W,\mc C) . 
\end{equation}
(In \cite{Ree} these groups are by definition subsets of complex vector spaces. But with
the above definitions Reeder's proof still applies.)  From \eqref{eq:1.26}, \eqref{eq:1.29}
and \eqref{eq:1.28} we obtain an injection
\begin{equation}\label{eq:1.30}
\overline{R_\Z}(W) \to \bigoplus\nolimits_u \overline{R_\Z}(A_G (u)) ,
\end{equation}
where $u$ runs over a set of representatives for the unipotent classes of $G$.

\begin{thm}\label{thm:1.6}
The group of elliptic representations $\overline{R_\Z}(W)$ is torsion-free. 
\end{thm}
\begin{proof}
If $W$ is a product of irreducible Weyl groups $W_i$, 
then it follows readily from \eqref{eq:1.24} that
\[
\overline{R_\Z}(W) = \bigotimes\nolimits_i \overline{R_\Z}(W_i) .
\]
Hence we may assume that $W = W (R)$ is irreducible. By \eqref{eq:1.30} it suffices to
show that each $\overline{R_\Z}(A_G (u))$ is torsion-free. If $u$ is distinguished, then
$\mc C \cap G_P = \emptyset$ for all $P \subsetneq \Delta$, and $\overline{R_\Z}(A_G (u)) =
R_\Z (A_G (u))$. That is certainly torsion-free, so we do not have to consider distinguished
unipotent $u$ anymore.

For root systems of type $A$ and of exceptional type, the tables of component groups in
\cite[\S 13.1]{Car2} show that $A_G (u)$ is isomorphic to $S_n$ with $n \leq 5$. Moreover
$S_4$ and $S_5$ only occur when $u$ is distinguished. For $A_G (u) \cong S_2$ and for
$A_G (u) \cong S_3$ one checks directly that $\overline{R_\Z}(A_G (u))$ is torsion-free,
by listing all subgroups of $A_G (u)$ and all irreducible representations thereof.

That leaves the root systems of type $B,C$ and $D$. As group of type $B_n$ we take
$G = SO_{2n+1}(\C)$. By the Bala--Carter classification, the unipotent classes $\mc C$
in $G$ are parametrized by pairs of partitions $(\alpha,\beta)$ such that 
$2 |\alpha| + |\beta| = 2n + 1$ and $\beta$ has only odd parts, all distinct. 
A typical $u \in \mc C$ is distinguished in the standard Levi subgroup
\[
G_{\alpha} := GL_{\alpha_1}(\C) \times \cdots \times GL_{\alpha_d}(\C) \times SO_{|\beta|}(\C). 
\]
The part of $u$ in $SO_{|\beta|}$ depends only on $\beta$, it has Jordan blocks of
sizes $\beta_1, \beta_2, \ldots$
Let $\alpha'$ be a partition consisting of a subset of the terms of $\alpha$, say
\begin{equation}\label{eq:1.45}
\alpha' = (n)^{m'_n} (n-1)^{m'_{n-1}} \cdots (1)^{m'_1} .
\end{equation}
Let $\alpha''$ be a partition of $|\alpha| - |\alpha'|$ obtained from the remaining
terms of $\alpha$ by repeatedly replacing some $\alpha_i, \alpha_j$ by $\alpha_i + \alpha_j$.
All the standard Levi subgroups of $G$ containing this $u$ are of the form $G_{\alpha''}$.
The $GL$-factors of $G_{\alpha''}$ do not contribute to $A_{G_{\alpha''}}(u)$.
The part $u'$ of $u$ in $SO_{2(n - |\alpha''|) + 1}(\C)$ is parametrized by $(\alpha',\beta)$
and the quotient of $Z_{SO_{2(n - |\alpha''|)+1}(\C)}(u')$ 
by its unipotent radical is isomorphic to
\begin{equation}\label{eq:1.40}
\prod_{i \text{ even}} Sp_{2 m'_i}(\C) \times \prod_{i \text{ odd, not in } \beta} O_{2m'_i}(\C)
\times S \Big( \prod_{i \text{ odd, in } \beta} O_{2m'_i + 1}(\C) \Big) ,
\end{equation}
where the $S$ indicates that we take the subgroup of elements of determinant 1.
From this one can deduce the component group:
\begin{equation}\label{eq:1.31}
A_{G_{\alpha''}}(u) \cong A_{Sp_{2(n - |\alpha''|)}(\C)}(u') \cong 
\prod_{i \text{ odd, not in } \beta, m'_i > 0} \Z / 2 \Z \times 
S \Big( \prod_{i \text{ in } \beta} \Z / 2 \Z \Big) .
\end{equation}
We see that if
\begin{itemize}
\item $\alpha$ has an even term, 
\item or $\alpha$ has an odd term with multiplicity $>1$, 
\item or $\alpha$ has an odd term which also appears in $\beta$,
\end{itemize}
then there is a standard Levi subgroup $G_{\alpha''} \subsetneq G$ with 
$A_{G_{\alpha''}}(u) \cong A_G (u)$, namely, with $\alpha''$ just that one term of $\alpha$.
In these cases $\overline{R_\Z}(A_G (u)) = 0$.

Suppose now that $\alpha$ has only distinct odd terms, and that none of those
appears in $\beta$. Then \eqref{eq:1.31} becomes
\[
A_{G}(u) \cong \prod_{i \text{ in } \alpha} \Z / 2\Z \times A
\qquad \text{where } A = S \Big( \prod_{i \text{ in } \beta} \Z / 2 \Z \Big) . 
\]
We get
\begin{multline} \label{eq:1.36}
\sum_{P \subsetneq \Delta , u_P \in \mc C \cap G_P} 
\ind_{A_{G_P}(u_P)}^{A_G (u_P)} \big( R_\Z (A_{G_P}(u_P)) \big)  \cong \\
\sum_{j \text{ in } \alpha} \ind^{A_G (u)}_{A_{G_{\alpha - (j)}}(u)}
R_\Z \Big( \prod_{i \text{ in } \alpha - (j)} \Z / 2 \Z \times A \Big) \cong \\
\sum_{j \text{ in } \alpha} \ind_{\{1\}}^{\Z / 2\Z} R_\Z (\{1\}) \otimes_\Z
R_\Z \Big( \prod_{i \text{ in } \alpha - (j)} \Z / 2 \Z \Big) \otimes_\Z R_\Z (A) .
\end{multline}
We conclude that $\overline{R_\Z}(A_G (u)) = R_\Z (A)$.

So $\overline{R_\Z}(A_G (u))$ is torsion free for all unipotent $u \in SO_{2n+1}(\C)$,
which settles the case $B_n$. The root systems of types $C_n$ and $D_n$ can be
handled in a completely analogous way, using the explicit descriptions in 
\cite[\S 13.1]{Car2}.
\end{proof}

For every $w \in W$ there exists (more or less by definition) a unique parabolic
subgroup $\tilde W \subset W$ such that $w$ is an elliptic element of $\tilde W$.
Let $\mc C (W)$ be the set of conjugacy classes of $W$.
For $P \subset \Delta$ let $\mc C_{P} (W)$ be the subset consisting
of those conjugacy classes that contain an elliptic element of $W_P$.
Let $\mc P (\Delta) / W$ be a set of representatives for the $W$-association
classes of subsets of $\Delta$. Since every parabolic subgroup is conjugate to
a standard one, for every conjugacy class $\mc C$ in $W$ there exists a unique 
$P \in \mc P (\Delta) / W$ such that $\mc C \in \mc C_{P}(W)$.

Recall from \cite[\S 3.3]{Ree} that a unipotent element $u \in G$ is called 
quasidistinguished if there exists a semisimple $t \in Z_G (u)$ such that $t u$
is not contained in any proper Levi subgroup of $G$.

\begin{prop}\label{prop:1.8}
For every $P \in \mc P (\Delta) / W$ there exists an injection from $\mc C_{P}(W)$ 
to the set of $G_P$-conjugacy classes of pairs $(u_P,\rho_P)$ with $u_P \in G_P$ 
quasidistinguished unipotent and $\rho_P \in \Irr (A_{G_P} (u_P))$ geometric, denoted 
$w \mapsto (u_{P,w},\rho_{P,w})$, such that:
\enuma{
\item $\{ H(u_w,\rho_w) : w \in \mc C_{\Delta}(W)\}$
is a $\Z$-basis of $\overline{R_\Z}(W)$.
\item The set
\[
\big\{ \ind_{W_P}^W \big( H_P (u_{P,w},\rho_{P,w}) \big) : 
P \in \mc P (\Delta ) / W, w \in \mc C_{P}(W) \big\}
\]
is a $\Z$-basis of $R_\Z (W)$.
}
\end{prop}
\begin{proof}
(a) By \cite[Proposition 2.2.2]{Ree} the rank of $\overline{R_\Z}(W)$ is the 
number of elliptic conjugacy classes of $W$.
With Theorem \ref{thm:1.6} we find $\overline{R_\Z}(W) \cong \Z^{|\mc C_{\Delta}(W)|}$.
By \eqref{eq:1.30} and \eqref{eq:1.29} $\overline{R_\Z}(W)$ has a basis consisting
of representations of the form $H(u,\rho)$ with $\rho \in \Irr (A_G (u))$ geometric.
By \cite[Proposition 3.4.1]{Ree} we need only quasidistinguished unipotent $u$.
We choose such a set of pairs $(u,\rho)$, and we parametrize it in an arbitrary
way by $\mc C_{\Delta}(W)$. \\
(b) We prove this by induction on $|\Delta|$. For $|\Delta| = 0$ the statement is trivial.

Suppose now that $|\Delta| \geq 1$ and $\alpha \in \Delta$. By the induction hypothesis
we can find maps $w \mapsto (u_P,\rho_P)$ such that the set
\[
\big\{ \ind_{W_P}^{W_{\Delta \setminus \{\alpha\}}} \big( H_P (u_{P,w},\rho_{P,w}) \big) : 
P \in \mc P (\Delta \setminus \{\alpha\}) / W_{\Delta \setminus \{\alpha\}}, 
w \in \mc C_{P}(W_{\Delta \setminus \{\alpha\}}) \big\} 
\]
is a $\Z$-basis of $R_\Z (W_{\Delta \setminus \{\alpha\}})$. By means of any setwise
splitting of $N_G (T) \to W$ we can arrange that
$(u_{P,w},\rho_{P,w})$ and $(u_{P',w'},\rho_{P',w'})$ are $G$-conjugate whenever
$(P,w)$ and $(P',w')$ are $W$-associate. Then $(P,w)$ and $(P',w')$ give rise to the
same $W$-representation. Consequently
\[
\big\{ \ind_{W_P}^W \big( H_P (u_{P,w},\rho_{P,w}) \big) : 
P \in \mc P (\Delta ) / W, P \neq \Delta, w \in \mc C_{P}(W) \big\}
\]
is well-defined and has $|\mc C (W) \setminus \mc C_{\Delta}(W)|$ elements.
By the induction hypothesis it spans 
$\sum_{P \subsetneq \Delta} \ind_{W_P}^W \big( R_\Z (W_P) \big)$,
so it forms a $\Z$-basis thereof. Combine this with \eqref{eq:1.24} and part (a).
\end{proof}

\subsection{Graded Hecke algebras} \

We consider the Grothendieck group $R_\Z (\mh H)$ of finite length modules of a graded 
Hecke algebra $\mh H$ with parameters $k$. We show that it is the direct sum of the 
subgroup spanned by modules induced from proper parabolic subalgebras and an elliptic part
$\overline{R_\Z} (\mh H)$. We prove that $\overline{R_\Z} (\mh H)$ is isomorphic to the
elliptic part of the representation ring of the Weyl group associated to $\mh H$. By
Paragraph \ref{par:Weyl}, $\overline{R_\Z} (\mh H)$ is free abelian and does not depend 
on the parameters $k$. The main ingredients are the author's work \cite{SolHomGHA} on the 
periodic cyclic homology of graded Hecke algebras, and the study of discrete series 
representations by Ciubotaru, Opdam and Trapa \cite{CiOp2,CiOpTr}.

Graded Hecke algebras are also known as degenerate (affine) Hecke algebras. They were 
introduced by Lusztig in \cite{Lus-Gr}. In the notation from \eqref{eq:1.37} we call 
\begin{equation}
\tilde{\mc R} = (\mf a^* ,R, \mf a, R^\vee, \Delta )
\end{equation}
a degenerate root datum. We pick complex numbers $k_\alpha$ for $\alpha \in \Delta$,
such that $k_\alpha = k_\beta$ if $\alpha$ and $\beta$ are in the same $W$-orbit.
We put $\mf t = \C \otimes_\R \mf a$. 

The graded Hecke algebra associated to these data is the complex vector space
\[
\mh H = \mh H (\tilde{\mc R},k) = \mc O ( \mf t) \otimes \C [W] ,
\]
with multiplication defined by the following rules:
\begin{itemize}
\item $\mh C[W]$ and $\mc O (\mf t )$ are canonically embedded as subalgebras;
\item for $\xi \in \mf t^*$ and $s_\alpha \in S_\Delta$ we have the cross relation
\begin{equation}\label{eq:1.1}
\xi \cdot s_\alpha - s_\alpha \cdot s_\alpha (\xi) = 
k_\alpha \inp{\xi}{\alpha^\vee} .
\end{equation}
\end{itemize}
Notice that $\mh H (\tilde{\mc R},0) = \mc O (\mf t) \rtimes W$.

Multiplication with any $\ep \in \mh C^\times$ defines a bijection $\mf t^* \to \mf t^*$,
which clearly extends to an algebra automorphism of $\mc O (\mf t) = S(\mf t^* )$. 
From the cross relation
\eqref{eq:1.1} we see that it extends even further, to an algebra isomorphism
\begin{equation}\label{eq:1.3}
\mh H (\tilde{\mc R},\ep k) \to \mh H (\tilde{\mc R}, k)
\end{equation}
which is the identity on $\mh C[W]$. For $\ep = 0$ this map is well-defined, but
obviously not bijective.

For a set of simple roots $P \subset \Delta$ we write
\begin{equation}
\begin{array}{l@{\qquad}l}
R_P = \mh Q P \cap R & R_P^\vee = \mh Q R_P^\vee \cap R^\vee , \\
\mf a_P = \R P^\vee & \mf a^P = (\mf a^*_P )^\perp ,\\
\mf a^*_P = \R P & \mf a^{P*} = (\mf a_P )^\perp  ,\\
\tilde{\mc R}_P = ( \mf a_P^* ,R_P ,\mf a_P ,R_P^\vee ,P) & 
\tilde{\mc R}^P = (\mf a^*,R_P ,\mf a,R_P^\vee ,P) .
\end{array}
\end{equation}
Let $k_P$ be the restriction of $k$ to $R_P$. We call
\[
\mh H^P = \mh H (\tilde{\mc R}^P, k_P)
\]
a parabolic subalgebra of $\mh H$. It contains $\mh H_P = \mh H (\tilde{\mc R}_P,k_P)$
as a direct summand.

The centre of $\mh H (\tilde{\mc R},k)$ is $\mc O (\mf t)^W = \mc O (\mf t / W)$ 
\cite[Proposition 4.5]{Lus-Gr}. Hence the central character of an irreducible 
$\mh H (\tilde{\mc R},k)$-representation is an element of $\mf t / W$.

Let $(\pi,V)$ be an $\mh H (\tilde{\mc R},k)$-representation. We say that 
$\lambda \in \mf t$ is an $\mc O (\mf t)$-weight of $V$ (or of $\pi$) if
\[
\{ v \in V : \pi (\xi) v = \lambda (\xi) v \text{ for all } \xi \in \mf t^* \}
\]
is nonzero. Let Wt$(V) \subset \mf t$ be the set of $\mc O (t)$-weights of $V$.

Temperedness of a representation is defined via its $\mc O (\mf t)$-weights. We write
\begin{align*}
& \mf a^+ = \{ \mu \in \mf a : \inp{\alpha}{\mu} \geq 0 \: \forall \alpha \in \Delta \} , \\
& \mf a^{*+} :=  \{ x \in \mf a^* :  \inp{x}{\alpha^\vee} \geq 0 
\: \forall \alpha \in \Delta \} , \\
& \mf a^- = \{ \lambda \in \mf a : \inp{x}{\lambda} \leq 0 \: \forall x \in \mf a^{*+} \} = 
\big\{ \sum\nolimits_{\alpha \in \Delta} \lambda_\alpha \alpha^\vee : \lambda_\alpha \leq 0 \big\} .
\end{align*}
The interior $\mf a^{--}$ of $\mf a^-$ equals
$\big\{ {\ts \sum_{\alpha \in \Delta}} \lambda_\alpha \alpha^\vee : \lambda_\alpha < 0 \big\}$
if $\Delta$ spans $\mf a^*$, and is empty otherwise.

We regard $\mf t = \mf a \oplus i \mf a$ as the polar
decomposition of $\mf t$, with associated real part map $\Re : \mf t \to \mf a$.
By definition, a finite dimensional $\mh H (\tilde{\mc R},k)$-module $(\pi,V)$ is tempered
$\Re (\mr{Wt}(V)) \subset \mf a^{-}$. More restrictively, we say that $(\pi,V)$ belongs to
the discrete series if $\Re (\mr{Wt}(V)) \subset \mf a^{--}$.

We are interested in the restriction map
\[
\begin{array}{cccc}
\mr{r} : & \Mod (\mh H (\tilde{\mc R},k)) & \to & \Mod (\C [W]) , \\
& V & \mapsto & V |_W .
\end{array}
\]
We can also regard it as the composition of representations with the algebra homomorphism
\eqref{eq:1.3} for $\ep = 0$, then its image consists of 
$\mc O (\mf t) \rtimes W$-representations on which $\mc O (\mf t)$ acts via $0 \in \mf t$.

Let $\Irr_0 (\mh H)$ be the set of irreducible tempered $\mh H (\tilde{\mc R},k)$-modules
with central character in $\mf a / W$. It is known from 
\cite[Theorem 6.5]{SolHomGHA} that, for real-valued $k$, r induces a bijection
\begin{equation}\label{eq:1.2}
\mr{r}_\C : \C \, \Irr_0 (\mh H (\tilde{\mc R},k)) \to R_\C (W) .
\end{equation}
Using work of Lusztig, Ciubotaru \cite[Corollary 3.6]{Ciu} showed that, for parameters
of ``geometric" type,
\begin{equation}\label{eq:1.5}
\mr{r}_\Z : \Z \, \Irr_0 (\mh H (\tilde{\mc R},k)) \to R_\Z (W) \text{ is bijective.}
\end{equation}
We will generalize this to arbitrary real parameters. (Parameters $k$ of geometric type 
need not be real-valued, but via \eqref{eq:1.3} they can be reduced to that.)

We recall some notions from \cite{CiOp1}. Let $R_\Z (\mh H (\tilde{\mc R},k))$ be the
Grothendieck group of (the category of) finite dimensional $\mh H (\tilde{\mc R},k)$-modules.
For any parabolic subalgebra $\mh H^P = \mh H (\tilde{\mc R}^P,k_P)$ the induction functor
$\ind_{\mh H^P}^{\mh H}$ induces a map $R_\Z (\mh H^P) \to R_\Z (\mh H)$.
If the $\mc O (\mf t)$-weights of $V \in \Mod (\mh H^P)$ are contained in some $U \subset \mf t$, 
then by \cite[Theorem 6.4]{BaMo} the $\mc O (\mf t)$-weights of $\ind_{\mh H^P}^{\mh H} V$
are contained in $W^P U$, where $W^P$ is the set of shortest length representatives of
$W / W_P$. This implies that $\ind_{\mh H^P}^{\mh H}$ preserves temperedness 
\cite[Corollary 6.5]{BaMo} and central characters. In particular it induces a map
\begin{equation}\label{eq:1.9}
\ind_{\mh H^P}^{\mh H} : \Z \, \Irr_0 (\mh H^P) \to \Z \, \Irr_0 (\mh H). 
\end{equation}
Many arguments in this section make use of the group of ``elliptic $\mh H$-representations"
\begin{equation}\label{eq:1.4}
\overline{R_\Z} (\mh H) = R_\Z (\mh H (\tilde{\mc R},k)) \big/
\sum\nolimits_{P \subsetneq \Delta} \ind_{\mh H^P}^{\mh H} \big( R_\Z (\mh H^P) \big) .
\end{equation}
Since $\mh H (\tilde{\mc R},k) = \mc O (\mf t) \otimes \C [W]$ as vector spaces, 
\begin{equation}\label{eq:1.34}
\mr{r} \circ \ind_{\mh H^P}^{\mh H} = \ind_{W_P}^W \circ \mr{r}^P,
\end{equation}
where $\mr{r}^P$ denotes the analogue of r for $\mh H^P$. Hence r induces a $\Z$-linear map
\begin{equation}\label{eq:1.6}
\bar{\mr{r}} : \overline{R_\Z} (\mh H (\tilde{\mc R},k)) \to \overline{R_\Z} (W) . 
\end{equation}

\begin{prop}\label{prop:1.7}
The map \eqref{eq:1.6} is surjective, and its kernel is the torsion subgroup of
$\overline{R_\Z} (\mh H (\tilde{\mc R},k))$.
\end{prop}
\begin{proof}
By Theorem \ref{thm:1.6} $\overline{R_\Z} (W)$ is torsion-free, so it can be identified
with its image in $\overline{R_\C} (W)$. This means that our definition of $\overline{R_\Z} (W)$
agrees with that in \cite{CiOpTr}. Likewise, in \cite{CiOpTr} the subgroup 
$\overline{R'_\Z} (\mh H (\tilde{\mc R},k))$ of $\overline{R_\C} (\mh H (\tilde{\mc R},k))$ 
generated by the actual representations is considered. In other words, 
$\overline{R'_\Z} (\mh H (\tilde{\mc R},k))$ is defined as the quotient
of $\overline{R_\Z} (\mh H (\tilde{\mc R},k))$ by its torsion subgroup.

By \cite[Proposition 5.6]{CiOpTr} the map 
\begin{equation}\label{eq:1.32}
\overline{\mr r} : \overline{R'_\Z} (\mh H (\tilde{\mc R},k)) \to \overline{R_\Z} (W) 
\end{equation}
is bijective, except possibly when $R$ has type $F_4$ and $k$ is not a generic parameter. 
However, in view of the more recent work \cite[\S 3.2]{CiOp2}, the limit argument given 
(for types $B_n$ and $G_2$) in \cite[\S 5.1]{CiOpTr} also applies to $F_4$. Thus \eqref{eq:1.32} 
is bijective for all $\tilde{\mc R}$ and all real-valued parameters $k$.
\end{proof}

\begin{lem}\label{lem:1.1}
Let $k$ be real-valued. The canonical map
\[
\Z \, \Irr_0 (\mh H (\tilde{\mc R},k)) \to \overline{R_\Z} (\mh H (\tilde{\mc R},k)) 
\]
is surjective.
\end{lem}
\begin{proof}
It was noted in \cite[Lemma 6.3]{OpSo2} (in the context of affine Hecke algebras)
that every element of $\overline{R_\Z} (\mh H (\tilde{\mc R},k))$ can be represented
by a tempered virtual representation. Consider any irreducible tempered $\mh H$-representation
$\pi$. By \cite[Proposition 8.2]{SolGHA} there exists a $P \subset \Delta$, a discrete
series representation $\delta$ of $\mh H_P$ and an element $\nu \in i \mf a^P$, such
that $\pi$ is a direct summand of 
\[
\pi (P,\delta,\nu) = \ind_{\mh H_P \otimes \mc O (\mf t^P)}^{\mh H} (\delta \otimes \C_\nu) .
\]
By \cite[Proposition 8.3]{SolGHA} the reducibility of $\pi (P,\delta,\nu)$ is determined by 
intertwining operators $\pi (w,P,\delta,\nu)$ for elements $w \in W$ that stabilize $(P,\delta,\nu)$. 
Suppose that $\nu \neq 0$. Then $W_\nu$ is a proper parabolic subgroup of $W$, so the stabilizer
of $(P,\delta,\nu)$ is contained in $W_Q$ for some $P \subset Q \subsetneq \Delta$.
In that case $\pi = \ind_{\mh H^Q}^{\mh H}(\pi^Q)$ for some irreducible representation
$\pi^Q$ of $\mh H^Q$, so $\pi$ becomes zero in $\overline{R_\Z} (\mh H (\tilde{\mc R},k))$.

Therefore we need only $\Z$-linear combinations of summands of $\pi (\delta,P,0)$ (with 
varying $P,\delta$) to surject to $\overline{R_\Z} (\mh H (\tilde{\mc R},k))$.
Since $k$ is real, discrete series representations of $\mh H_P$ have central characters in 
$\mf a_P / W_P$ \cite[Lemma 2.13]{Slo3}. It follows that $\pi (P,\delta,0)$ and all its 
constituents (among which is $\pi$) admit a central character in $\mf a / W$. 
\end{proof}

\begin{thm}\label{thm:1.2}
Let $k$ be real-valued. The restriction-to-$W$ maps
\[
\begin{array}{ccccc}
\mr{r}_\Z & : & \Z \, \Irr_0 (\mh H (\tilde{\mc R},k)) & \to & R_\Z (W) , \\
\overline{\mr{r}} & : & \overline{R_\Z} (\mh H (\tilde{\mc R},k)) & \to & \overline{R_\Z} (W) 
\end{array}
\]
are bijective.
\end{thm}
\begin{proof}
We will show this by induction on the semisimple rank of $\tilde{\mc R}$ 
(i.e. the rank of $R$).
Suppose first that the semisimple rank is zero. Then $W = 1$ and $\mh H = \mc O (\mf t)$. 
For $\lambda \in \mf t$ the character
\[
\mr{ev}_\lambda : f \mapsto f(\lambda)
\]
is a tempered $\mc O (\mf t)$-representation if and only if $\Re (\lambda) = 0$. If $\lambda$
is at the same time a real central character (i.e. $\lambda \in \mf a$), then $\lambda = 0$.
Hence $\Irr_0 (\mh H )$ consists just of $\mr{ev}_0$. It is mapped to the trivial 
$W$-representation by r, so the theorem holds in this case.

Now let $\tilde{\mc R}$ be of positive semisimple rank. It is a direct sum of degenerate
root data with $R$ irreducible or $R$ empty, and $\mh H (\tilde{\mc R},k)$ decomposes 
accordingly. As we already know the result when $R$ is empty, it remains to establish the
case where $R$ is irreducible. 

Any proper parabolic subalgebra $\mh H^P \subset \mh H$ has smaller semisimple rank, so
by the induction hypothesis 
\begin{equation}\label{eq:1.7}
\mr{r}^P : \Z \, \Irr_0 (\mh H^P) \to \Z \, \Irr_0 (W_P) \text{ is bijective.} 
\end{equation}
Consider the commutative diagram
\begin{equation}\label{eq:1.33}
\begin{array}{ccccccccc}
0 & \to & \sum_{P \subsetneq \Delta} \ind_{\mh H^P}^{\mh H} \big( \Z \,\Irr_0 (\mh H^P) \big) &
\to & \Z \, \Irr_0 (\mh H) & \to & \overline{R_\Z}(\mh H) & \to & 0 \\
& & \downarrow & & \downarrow & & \downarrow \\
0 & \to & \sum_{P \subsetneq \Delta} \ind_{W_P}^W \big( R_\Z (W_P) \big) &
\to & R_\Z (W) & \to & \overline{R_\Z}(W) & \to & 0 
\end{array} 
\end{equation}
The second row is exact by definition. By \eqref{eq:1.7} and \eqref{eq:1.34} the left vertical 
arrow is bijective and by Proposition \ref{prop:1.7} the right vertical arrow is surjective. 
Together with Lemma \ref{lem:1.1} these imply that the middle vertical arrow is surjective.
By \eqref{eq:1.2} both $\Z \, \Irr_0 (\mh H)$ and $R_\Z (W)$ are free abelian groups of 
the same rank $|\Irr (W)| = |\Irr_0 (\mh H)|$, so the middle vertical arrow is in fact bijective.

The results so far imply that the kernel of $\Z \, \Irr_0 (\mh H) \to \overline{R_\Z}(W)$
is precisely $\sum_{P \subsetneq \Delta} \ind_{\mh H^P}^{\mh H} \big( \Z \, \Irr_0 (\mh H^P) \big)$.
The latter group is already killed in $\overline{R_\Z}(\mh H)$, so the map
$\overline{R_\Z}(\mh H) \to  \overline{R_\Z}(W)$ is injective as well.
We conclude that \eqref{eq:1.33} is a bijection between two short exact sequences.
\end{proof}

We will need Theorem \ref{thm:1.2} for somewhat more general algebras. Let $\Gamma$ be a
finite group acting on $\tilde{\mc R}$: it acts $\R$-linearly on $\mf a$, and the dual
action on $\mf a^*$ stabilizes $R$ and $\Delta$. We assume that $k_{\gamma (\alpha)} =
k_\alpha$ for all $\alpha \in R, \gamma \in \Gamma$. Then $\Gamma$ acts on 
$\mh H (\tilde{\mc R},k)$ by the algebra automorphisms satisfying
\[
\gamma (\xi N_w) = \gamma (\xi) N_{\gamma w \gamma^{-1}} \qquad 
\gamma \in \Gamma, \xi \in \mf a^* , w \in W .
\]
Let $\natural : \Gamma^2 \to \C^\times$ be a 2-cocycle and let $\C[\Gamma,\natural]$ be the 
twisted group algebra. We recall that it has a standard basis 
$\{ N_\gamma : \gamma \in \Gamma \}$ and multiplication rules
\[
N_\gamma N_{\gamma'} = \natural (\gamma,\gamma') N_{\gamma \gamma'} 
\qquad \gamma, \gamma' \in \Gamma .
\]
We can endow the vector space $\mh H (\tilde{\mc R},k) \otimes \C [\Gamma,\natural]$
with the algebra structure such that 
\begin{itemize}
\item $\mh H (\tilde{\mc R},k)$ and $\C [\Gamma,\natural]$ are embedded as subalgebras,
\item $N_\gamma h N_\gamma^{-1} = \gamma (h)$ for 
$\gamma \in \Gamma, h \in \mh H (\tilde{\mc R},k)$.
\end{itemize}
We denote this algebra by $\mh H (\tilde{\mc R},k) \rtimes \C [\Gamma,\natural]$ and call
it a twisted graded Hecke algebra. If $\natural$ is trivial, then it reduces to the crossed 
product $\mh H (\tilde{\mc R},k) \rtimes \Gamma$. All our previous notions for graded 
Hecke algebras admit natural generalizations to this setting.

Notice that $W \Gamma$ is a group with $W$ as normal subgroup and $\Gamma$ as quotient. 
The 2-cocycle $\natural$ can be lifted to $(W \Gamma)^2 \to \Gamma^2 \to (\C^\times)^2$,
and that yields a twisted group algebra $\C [W\Gamma,\natural]$ in 
$\mh H (\tilde{\mc R},k) \rtimes \C [\Gamma,\natural]$. It is worthwhile to note the case $k=0$:
\begin{equation}\label{eq:1.18}
\mh H (\tilde{\mc R},0) \rtimes \C [\Gamma,\natural] = 
\mc O (\mf t) \rtimes \C [W \Gamma,\natural]) .
\end{equation}
We consider the restriction map
\begin{equation}\label{eq:1.10}
\mr{r} : \Mod \big( \mh H (\tilde{\mc R},k) \rtimes \C [\Gamma,\natural] \big) \to
\Mod ( \C [W\Gamma,\natural] ) .
\end{equation}
Every $\C [W\Gamma,\natural]$-module can be extended in a unique way to an 
$\mc O (\mf t) \rtimes \C [W \Gamma,\natural])$-module on which $\mc O (\mf t)$ acts
via evaluation at $0 \in \mf t$, so the right hand side of \eqref{eq:1.10} 
can be considered as a subcategory
of $\Mod \big( \mh H (\tilde{\mc R},0) \rtimes \C [\Gamma,\natural] \big)$.

\begin{prop}\label{prop:1.3}
Let $k : R / W\Gamma \to \R$ be a parameter function and let $\natural : 
\Gamma^2 \to \C^\times$ be a 2-cocycle. The map \eqref{eq:1.10} induces a bijection
\[
\mr{r}_\Z : \Z \, \Irr_0 \big( \mh H (\tilde{\mc R},k) \rtimes \C [\Gamma,\natural] \big) 
\to R_\Z ( \C [W\Gamma,\natural] ) .
\]
\end{prop}
\begin{proof}
Let $\tilde{\Gamma} \to \Gamma$ be a finite central extension such that $\natural$ becomes
trivial in $H^2 (\tilde{\Gamma},\C^\times)$. Such a group always exists: one can take the
Schur extension from \cite[Theorem 53.7]{CuRe}. Then there exists a central idempotent
$p_\natural \in \C [\ker (\tilde{\Gamma} \to \Gamma)]$ such that
\begin{equation}\label{eq:1.11}
\C [\Gamma,\natural] \cong p_\natural \C [\tilde{\Gamma}] . 
\end{equation}
The map $\mr{r}_\Z$ becomes
\begin{equation}\label{eq:1.12}
\Z \, \Irr_0 \big( \mh H (\tilde{\mc R},k) \rtimes p_\natural \C [\tilde \Gamma] \big) \to
R_\Z ( p_\natural \C [W \tilde \Gamma] ) .
\end{equation}
Since $p_\natural \C [\tilde{\Gamma}]$ is a direct summand of $\C [\tilde{\Gamma}]$,
\eqref{eq:1.12} is just a part of
\[
\mr{r}_\Z : \Z \, \Irr_0 \big( \mh H (\tilde{\mc R},k) \rtimes \tilde \Gamma \big) \to
R_\Z ( W \rtimes \tilde \Gamma ) .
\]
Hence it suffices to prove the proposition when $\natural$ is trivial, which we assume
from now on. By \cite[Theorem 6.5.c]{SolHomGHA}
\begin{equation}\label{eq:1.14}
\mr{r}_\C : \C \, \Irr_0 \big( \mh H (\tilde{\mc R},k) \rtimes \Gamma \big) 
\to R_\C ( W\Gamma ) .
\end{equation}
is a $\C$-linear bijection. So at least
\begin{equation}\label{eq:1.15}
\mr{r}_\Z : \Z \, \Irr_0 \big( \mh H (\tilde{\mc R},k) \rtimes \Gamma \big) 
\to R_\Z ( W\Gamma ) 
\end{equation}
is injective and has image of finite index in $R_\Z ( W\Gamma )$.

Given $(\pi,V) \in \Irr (\mh H (\tilde{\mc R},k))$, let $\Gamma_\pi$ be
the stabilizer in $\Gamma$ of the isomorphism class of $\pi$. For every $\gamma \in
\Gamma_\pi$ we can find $I^\gamma \in \Aut_\C (V)$ such that 
\[
I^\gamma \circ \pi (N_\gamma h N_\gamma^{-1}) = \pi (h) \circ I^\gamma \qquad
\text{for all } h \in \mh H (\tilde{\mc R},k) .
\]
By Schur's Lemma there exists a 2-cocycle $\natural_\pi : \Gamma_\pi^2 \to \C^\times$
such that
\[
I^{\gamma \gamma'} = \natural_\pi (\gamma,\gamma') I^\gamma I^{\gamma'} \quad
\text{for all } \gamma, \gamma' \in \Gamma .
\]
Let $(\tau,M) \in \Irr (\C [\Gamma_\pi,\natural_\pi])$, then $M \otimes V$ becomes an 
irreducible $\mh H \rtimes \Gamma_\pi$-module. Clifford theory (see e.g.
\cite[Appendix]{RaRa}, \cite[\S 51]{CuRe} or \cite[Appendix]{SolGHA}) tells us that
$\ind_{\mh H \rtimes \Gamma_\pi}^{\mh H \rtimes \Gamma} (M \otimes V)$ is an
irreducible $\mh H \rtimes \Gamma$-module. Moreover this construction provides a bijection
\[
\Irr (\mh H \rtimes \Gamma) \to \{ (\pi,M) : \pi \in \Irr (\mh H ) / \Gamma, 
M \in \Irr (\C [\Gamma_\pi, \natural_\pi]) \} .
\]
We note that
\begin{equation}\label{eq:1.13}
\mr{r} \big( \ind_{\mh H \rtimes \Gamma_\pi}^{\mh H \rtimes \Gamma} (M \otimes V) \big) =
\ind_{W \rtimes \Gamma_\pi}^{W \rtimes \Gamma} (M \otimes \mr{r}(V) ) .
\end{equation}
Similarly, Clifford theory provides a bijection between $\Irr (W \rtimes \Gamma)$ and
\[
\{ (\tau,N) : \tau \in \Irr (W) / \Gamma , N \in \Irr (\C [\Gamma_\tau, \natural_\tau]) \} .
\]
Since $W$ is a Weyl group, the 2-cocycle $\natural_\tau$ is always trivial 
\cite[Proposition 4.3]{ABPS1}. With \eqref{eq:1.13} it follows that $\natural_\pi$ is
also trivial, for every $\pi \in \Irr (\mh H (\tilde{\mc R},k))$.

Consider any $\ind_{W \rtimes \Gamma_\tau}^{W \rtimes \Gamma} (N \otimes V_\tau) \in
\Irr (W \rtimes \Gamma)$. Theorem \ref{thm:1.2} guarantees the existence of unique
$m_\pi \in \Z$ such that $V_\tau = \sum_{(\pi,V) \in \Irr_0 (\mh H)} m_\pi \mr{r}(V)$.
By the uniqueness, $\Gamma_\pi \supset \Gamma_\tau$ whenever $m_\pi \neq 0$.
Hence $N \otimes V$ is a well-defined $\mh H \rtimes \Gamma_\pi$-module (it may be
reducible though), and
\begin{multline*}
\ind_{W \rtimes \Gamma_\tau}^{W \rtimes \Gamma} (N \otimes V_\tau) = \\
\ind_{W \rtimes \Gamma_\tau}^{W \rtimes \Gamma} \big( N \otimes 
\sum_{(\pi,V) \in \Irr_0 (\mh H)} m_\pi \mr{r}(V) \big) 
= \mr{r} \big( \sum_{(\pi,V) \in \Irr_0 (\mh H)} m_\pi \ind_{\mh H \rtimes 
\Gamma_\pi}^{\mh H \rtimes \Gamma} (N \otimes V) \big) .
\end{multline*}
This proves that \eqref{eq:1.15} is also surjective.
\end{proof}

\subsection{Affine Hecke algebras} \
\label{par:AHA}

Let $\mc H$ be an affine Hecke algebra with positive parameters $q$. We compare its Grothendieck
group of finite length modules $R_\Z (\mc H)$ with the analogous group for the parameters $q=1$.
By some of the main results of \cite{SolAHA}, the $\Q$-vector space $\Q \otimes_\Z R_\Z (\mc H)$
is canonically isomorphic to its analogue for $q=1$. We show that this is already an isomorphism
for $R_\Z (\mh H)$, without tensoring by $\Q$. This follows from the results of the previous 
paragraph, in combination with the standard reduction from affine Hecke algebras to graded
Hecke algebras \cite{Lus-Gr}.

As before, let $\mc R = (X,R,Y,R^\vee,\Delta)$ be a based root datum.
We have the affine Weyl group $W^\af = \mh Z R \rtimes W$ 
and the extended (affine) Weyl group $W^e = X \rtimes W$. Both can be considered as groups
of affine transformations of $\mf a^*$. We denote the translation corresponding to $x \in X$ by 
$t_x$. As is well-known, $W^\af$ is a Coxeter group, and the basis $\Delta$ of $R$ gives rise to 
a set $S^\af$ of simple (affine) reflections. More explicitly, let $\Delta_M^\vee$ be the set of 
maximal elements of $R^\vee$, with respect to the dominance ordering coming from $\Delta$. Then
\[
S^\af = S_\Delta \cup \{ t_\alpha s_\alpha : \alpha^\vee \in \Delta_M^\vee \} .
\]
The length function $\ell$ of the Coxeter system $(W^\af ,S^\af )$ extends naturally to $W^e$.
The elements of length zero form a subgroup $\Omega \subset W^e$ and $W^e = W^\af \rtimes \Omega$.

A complex parameter function for $\mc R$ is a map $q : S^\af \to \mh C^\times$ such that 
$q(s) = q(s')$ if $s$ and $s'$ are conjugate in $W^e$. This extends naturally to a map 
$q : W^e \to \C^\times$ which is 1 on $\Omega$ and satisfies 
\[
q(w w') = q(w) q(w') \quad \text{if} \quad \ell (w w') = \ell (w) + \ell (w').
\]
Equivalently (see \cite[\S 3.1]{Lus-Gr}) one can define $q$ as a $W$-invariant function
\begin{equation}\label{eq:1.21}
q : R \cup \{ 2 \alpha : \alpha^\vee \in 2 Y \} \to \C^\times . 
\end{equation}
We speak of equal parameters if $q(s) = q(s')$ for all $s,s' \in S^\af$ and of positive 
parameters if $q(s) \in \R_{>0}$ for all $s \in S^\af$. 
We fix a square root $q^{1/2} : S^\af \to \mh C^\times$.

The affine Hecke algebra $\mc H = \mc H (\mc R ,q)$ is the unique associative
complex algebra with basis $\{ N_w \mid w \in W^e \}$ and multiplication rules
\begin{equation}\label{eq:multrules}
\begin{array}{lll}
N_w \, N_{w'} = N_{w w'} & \mr{if} & \ell (w w') = \ell (w) + \ell (w') \,, \\
\big( N_s - q(s)^{1/2} \big) \big( N_s + q(s)^{-1/2} \big) = 0 & \mr{if} & s \in S^\af .
\end{array}
\end{equation}
In the literature one also finds this algebra defined in terms of the
elements $q(s)^{1/2} N_s$, in which case the multiplication can be described without
square roots. This explains why $q^{1/2}$ does not appear in the notation $\mc H (\mc R ,q)$.
For $q = 1$ \eqref{eq:multrules} just reflects the defining relations of $W^e$, so
$\mc H (\mc R,1) = \C [W^e]$.

The set of dominant elements in $X$ is
\[
X^+ = \{ x \in X : \inp{x}{\alpha^\vee} \geq 0 \; \forall \alpha \in \Delta \} .
\]
The subset $\{ N_{t_x}  : x \in X^+ \} \subset \mc H (\mc R,q)$ is closed under 
multiplication, and isomorphic to $X^+$ as a semigroup. For any $x \in X$ we put
\[
\theta_x = N_{t_{x_1}} N_{t_{x_2}}^{-1} \text{, where } 
x_1 ,x_2 \in X^+ \text{ and } x = x_1 - x_2 . 
\]
This does not depend on the choice of $x_1$ and $x_2$, so $\theta_x \in \mc H (\mc R,q)^\times$
is well-defined. The Bernstein presentation of $\mc H (\mc R,q)$ \cite[\S 3]{Lus-Gr} 
says that:
\begin{itemize}
\item $\{ \theta_x : x \in X \}$ forms a $\C$-basis of a subalgebra of $\mc H (\mc R,q)$
isomorphic to $\C[X] \cong \mc O (T)$, which we identify with $\mc O (T)$.
\item $\mc H (W,q) := \C \{ N_w : w \in W \}$ is a finite dimensional subalgebra of 
$\mc H (\mc R,q)$ (known as the Iwahori--Hecke algebra of $W$).
\item The multiplication map $\mc O (T) \otimes \mc H (W,q) \to \mc H (\mc R,q)$ 
is a $\C$-linear bijection.
\item There are explicit cross relations between $\mc H (W,q)$ and $\mc O (T)$, deformations
of the standard action of $W$ on $\mc O (T)$.
\end{itemize}
To define parabolic subalgebras of affine Hecke algebras, we associate some objects
to any $P \subset \Delta$:
\[
\begin{array}{l@{\qquad}l}
X_P = X \big/ \big( X \cap (P^\vee )^\perp \big) &
X^P = X / (X \cap \mh Q P ) , \\
Y_P = Y \cap \mh Q P^\vee & Y^P = Y \cap P^\perp , \\
T_P = \mr{Hom}_{\mh Z} (X_P, \mh C^\times ) &
T^P = \mr{Hom}_{\mh Z} (X^P, \mh C^\times ) , \\
\mc R_P = ( X_P ,R_P ,Y_P ,R_P^\vee ,P) & \mc R^P = (X,R_P ,Y,R_P^\vee ,P) , \\
\mc H_P = \mc H (\mc R_P,q_P) & \mc H^P = \mc H (\mc R^P ,q^P) .
\end{array}
\]
Here $q_P$ and $q^P$ are derived from $q$ via \eqref{eq:1.21}. Both $\mc H_P$ and $\mc H^P$
are called parabolic subalgebras of $\mc H$. One can regard $\mc H_P$ as a ``semisimple"
quotient of $\mc H^P$.

Any $t \in T^P$ and any $u \in T^P \cap T_P$ give rise to algebra automorphisms
\begin{equation}\label{eq:1.23}
\begin{array}{llcl}
\psi_u : \mc H_P \to \mc H_P , & \theta_{x_P} N_w & \mapsto & u (x_P) \theta_{x_P} N_w , \\
\psi_t : \mc H^P \to \mc H^P , & \theta_x N_w & \mapsto & t(x) \theta_x N_w  .
\end{array}
\end{equation}
Let $\Gamma$ be a finite group acting on $\mc R$, i.e. it acts $\Z$-linearly on $X$ and
preserves $R$ and $\Delta$. We also assume that $\Gamma$ acts on $T$ by affine 
transformations, whose linear part comes from the action on $X$. Thus $\Gamma$ acts on
$\mc O (T) \cong \C [X]$ by
\begin{equation}\label{eq:1.38}
\gamma (\theta_x) = z_\gamma (x) \theta_{\gamma x} ,  
\end{equation}
for some $z_\gamma \in T$. Since this is a group action, we must have $z_\gamma \in T^W$.

We suppose throughout that $q^{1/2}$ is $\Gamma$-invariant, so
that $\gamma \in \Gamma$ acts on $\mc H (\mc R,q)$ by the algebra automorphism
\begin{equation}\label{eq:1.39}
\sum_{w \in W, x \in X} c_{x,w} \theta_x N_w \; \mapsto \; 
\sum_{w \in W, x \in X} c_{x,w} z_\gamma (x) \theta_{\gamma (x)} N_{\gamma w \gamma^{-1}} . 
\end{equation}
We can build the crossed product algebra 
\begin{equation}\label{eq:1.22}
\mc H (\mc R,q) \rtimes \Gamma. 
\end{equation}
In \cite{SolAHA} we considered a slightly less general action of $\Gamma$ on $\mc H (\mc R,q)$,
where the elements $z_\gamma \in T^W$ from \eqref{eq:1.38} were all equal to 1. But 
the relevant results from \cite{SolAHA}  do not rely on $\Gamma$ fixing the unit element of $T$,
so they are also valid for the actions as in \eqref{eq:1.39}. In this paper we will tacitly
use some results from \cite{SolAHA} in the generality of \eqref{eq:1.39}. We note that
nontrivial $z_\gamma \in T^W$ are sometimes needed to describe Hecke algebras coming from
$p$-adic groups, for example \cite[\S 4]{Roc}.

We can also endow the group $\Gamma$ with a 2-cocycle $\natural : \Gamma^2 \to \C^\times$. 
Then the vector space $\mc H (\mc R,q) \otimes \C[\Gamma,\natural]$ obtains a multiplication 
such that $\mc H(\mc R,q)$ and $\C[\Gamma,\natural]$ are subalgebras and
\[
N_\gamma h N_\gamma^{-1} = \gamma (h) \quad 
\text{for all } \gamma \in \Gamma, h \in \mc H (\mc R,q) . 
\]
We denote this by $\mc H (\mc R,q) \rtimes \C[\Gamma,\natural]$ and call it a twisted 
affine Hecke algebra. Such twists seem necessary to describe algebras appearing in the
representation theory of non-split $p$-adic groups, see e.g. \cite[Example 5.5]{ABPS2}. 
For reference we record the case $q = 1$:
\begin{equation}\label{eq:1.20}
\mc H (\mc R, 1) \rtimes \C[\Gamma,\natural] = \mc O (T) \rtimes \C[W\Gamma,\natural] .
\end{equation}
The representation theory of (twisted) affine Hecke algebras is closely related to that 
of (twisted) graded Hecke algebras, as first shown by Lusztig \cite{Lus-Gr}. 
Since $\mc H (\mc R,q)$ is of finite rank as a module over its commutative subalgebra
$\mc O (T)$, all irreducible $\mc H (\mc R,q)$-modules have finite dimension. 
The set of $\mc O (T)$-weights of an $\mc H (\mc R,q)$-module $V$ will be denoted by
Wt$(V)$. 

The vector space $\mf t = \mf a \oplus i \mf a$ can now be interpreted as the Lie algebra
of the complex torus $T = \Hom_\Z (X,\C^\times)$. The latter has a polar decomposition 
$T = T_\rs \times T_\un$ where $T_\rs = \Hom_\Z (X,\R_{>0})$ and $T_\un$ is the unique 
maximal compact subgroup of $T$. The polar decomposition of an element $t \in T$ is
written as $t = |t| \, (t \, |t|^{-1})$.

We write $T^- = \exp (\mf a^-) \subset T_\rs$ and $T^{--} = \exp (\mf a^{--}) \subset T_\rs$. 
We say that a module $V$ for $\mc H (\mc R,q)$ (or for $\mc H (\mc R,q) \rtimes \C[\Gamma,\natural]$) 
is tempered if $|\mr{Wt}(V)| \subset T^-$, and that it is discrete series if
$|\mr{Wt}(V)| \subset T^{--}$. (The latter is only possible if $R$ spans $\mf a$, for
otherwise $\mf a^{--}$ and $T^{--}$ are empty.)

By the Bernstein presentation, the centre of $\mc H (\mc R,q) \rtimes \C[\Gamma,\natural]$ 
contains $\mc O (T)^{W \Gamma}$. For any $W \Gamma$-invariant subset $U \subset T$, let 
$\Mod_{f,U} \big(\mc H (\mc R,q) \rtimes \C[\Gamma,\natural]\big)$ be the category of finite 
dimensional $\mc H (\mc R,q) \rtimes \C[\Gamma,\natural]$-modules awhose 
$\mc O (T)^{W \Gamma}$-weights all lie in $U / W \Gamma$. We denote the Grothendieck group
of this category by $R_{\Z,U}\big( \mc H (\mc R,q) \rtimes \C[\Gamma,\natural]\big)$.

The centre of  $\mh H (\tilde{\mc R},k) \rtimes \C[\Gamma,\natural]$ contains 
$\mc O (\mf t)^{W \Gamma}$. For any $W\Gamma$-invariant subset $V \subset \mf t$ we define
$\Mod_{f,V} \big(\mh H (\tilde{\mc R},k) \rtimes \C[\Gamma,\natural] \big)$ analogously.\\

Fix $u \in T_\un$. To $\mc R$ and $u$ we can associate some new objects. 
First we define the root system
\[
R_u = \{ \alpha \in R : s_\alpha (u) = u \} ,
\]
and we let $\Delta_u$ be the unique basis of $R_u$ contained in $R^+$. Then 
\begin{align*}
& (W \Gamma)_u = W(R_u) \rtimes \Gamma'_u , \\
& \Gamma'_u = \{ w \in W \Gamma : w (u) = u, w (\Delta_u) = \Delta_u \} . 
\end{align*}
Now we can define the based root data
\[
\mc R_u = (X,R_u,Y,R_u^\vee,\Delta_u) \quad \text{and} \quad
\tilde{\mc R_u} = (\mf a^* ,R_u,\mf a ,R_u^\vee,\Delta_u) .
\]
We define a parameter function $k_u : R_u \to \R$ for $\tilde{\mc R_u}$ by
\[
2 k_{u,\alpha} = \log (q(s_\alpha)) + \alpha (u) \log (q(t_\alpha s_\alpha)) . 
\]
Let $\natural_u : (\Gamma'_u )^2 \to \C^\times$ be the restriction to $\natural$.
With a slight variation on Lusztig's reduction theorems \cite[\S 8--9]{Lus-Gr}
one can prove:

\begin{thm}\label{thm:1.4}
Let $q : W^e \to \R_{>0}$ be a positive parameter function. The categories
\[
\Mod_{f,W\Gamma u T_\rs} \big( \mc H (\mc R,q) \rtimes \C[\Gamma,\natural] \big) 
\quad \text{and} \quad \Mod_{f,\mf a} 
\big( \mh H (\tilde{\mc R_u},k_u) \rtimes \C[\Gamma'_u,\natural_u ] \big) 
\]
are equivalent. The equivalence respects parabolic induction, temperedness and
discrete series.
\end{thm}
\begin{proof}
Let $\tilde \Gamma$ and the central idempotent $p_\natural$ be as in \eqref{eq:1.12}. Then
\begin{equation}\label{eq:1.16}
\begin{aligned} 
& \mc H (\mc R,q) \rtimes \C [\Gamma, \natural] = 
p_\natural (\mc H (\mc R,q) \rtimes \tilde \Gamma ) , \\
& \mh H (\tilde{\mc R_u},k_u) \rtimes \C [\Gamma'_u, \natural_u] = 
p_\natural (\mh H (\tilde{\mc R_u},k_u) \rtimes \tilde{\Gamma}'_u ) .
\end{aligned}
\end{equation}
By \cite[Corollary 2.15]{SolAHA} the theorem holds for $\mc H (\mc R,q) 
\rtimes \tilde \Gamma$ and $\mh H (\tilde{\mc R_u},k_u) \rtimes \tilde{\Gamma}'_u$.
The claimed properties of this equivalence were checked in detail in \cite[\S 2.1]{AMS3}.

This is based on a comparison of localizations of these algebras, as in \cite{Lus-Gr}.
The comparison maps \cite[Theorems 2.1.2 and 2.1.4]{SolAHA} are the identity on
$\C [\tilde{\Gamma}'_u \cap \tilde{\Gamma}]$, so they preserve $p_\natural$.
Hence we can restrict the result from \cite{SolAHA} to the direct summands
\eqref{eq:1.16}.
\end{proof}

From Theorem \ref{thm:1.4} and \eqref{eq:1.10} (and \eqref{eq:1.18} and \eqref{eq:1.20}
for the bottom line) we construct a diagram
\[
\begin{array}{ccc} 
\Mod_{f,W\Gamma u T_\rs} \big(\mc H (\mc R,q) \rtimes \C[\Gamma,\natural] \big) & \isom &
\Mod_{f,\mf a} \big(\mh H (\tilde{\mc R_u},k_u) \rtimes \C[\Gamma'_u,\natural_u ]\big) \\
\downarrow \mr{r}_u & & \downarrow \mr{r} \\
\Mod_{f,W\Gamma u} \big(\mc H (\mc R,1) \rtimes \C[\Gamma,\natural] \big) & \mosi &
\Mod_{f,0} \big(\mh H (\tilde{\mc R_u},0) \rtimes \C[\Gamma'_u,\natural_u ]\big) \\
\parallel & & \parallel \\
\Mod_{f,W\Gamma u} \big( \mc O (T) \rtimes \C[W\Gamma,\natural] \big) &
\mosi & \Mod_{f,0} \big( \mc O (\mf t) \rtimes \C[(W\Gamma)_u,\natural_u] \big) 
\end{array} 
\]
where $\mr{r}_u$ is the unique map that makes the diagram commutative. Using the
technique in the proof of Theorem \ref{thm:1.4}, we can immediately extend all relevant 
results in \cite{SolAHA} from $\mc H (\mc R,q) \rtimes \tilde \Gamma$ to twisted affine 
Hecke algebras. In view of this, we will freely use results from \cite{SolAHA} in that
generality.

As shown in \cite[\S 2.3]{SolAHA},
there exists a unique system of $\Z$-linear maps (for all possible $\mc R,q,\Gamma$)
\begin{equation}\label{eq:1.17}
\zeta^\vee : R_\Z \big(\mc H (\mc R,q) \rtimes \C[\Gamma,\natural] \big) \longrightarrow
R_\Z \big(\mc H (\mc R,1) \rtimes \C[\Gamma,\natural] \big)  
\end{equation}
such that:
\begin{itemize}
\item $\zeta^\vee (\pi) = \mr{r}_u (\pi)$ for tempered representations in
$\Mod_{f,W\Gamma u T_\rs} \big(\mc H (\mc R,q) \rtimes \C[\Gamma,\natural] \big)$,
\item $\zeta^\vee$ commutes with parabolic induction,
\item $\zeta^\vee$ respects the formation of standard modules for the Langlands 
classification, in the sense of \cite[Corollary 2.2.5]{SolAHA}.
\end{itemize}

\begin{thm}\label{thm:1.5}
The map \eqref{eq:1.17} is bijective for every positive parameter function $q$. 
\end{thm}
\begin{proof}
Proposition \ref{prop:1.3} and Theorem \ref{thm:1.4} imply that \eqref{eq:1.17}
gives a bijection
\begin{equation}\label{eq:1.19}
R_{\Z,\mr{temp},W \Gamma u T_\rs} \big(\mc H (\mc R,q) \rtimes \C[\Gamma,\natural] 
\big) \to R_{\Z,\mr{temp},W\Gamma u} \big(\mc H (\mc R,1) \rtimes \C[\Gamma,\natural] \big) ,
\end{equation}
where the subscripts ``temp'' indicate that we formed these Grothendieck groups by 
starting with tempered modules only.
Any tempered $\mc O (T) \rtimes \C[W\Gamma,\natural]$-module only has
$\mc O (T)$-weights in $T_\un$, so on the right hand side of \eqref{eq:1.19} we may
just as well replace $W\Gamma u$ by $W\Gamma u T_\rs$. Thus \eqref{eq:1.17}
restricts to a bijection between subgroups generated by tempered modules on both
sides.

In \cite[Corollary 2.3.2]{SolAHA} it was shown that \eqref{eq:1.17} becomes a 
$\Q$-linear bijection upon tensoring both sides with $\Q$. The second half of the proof
of that result \cite[\S 3.4]{SolAHA} extends the statement from the tempered to the
general case. It says essentially that whatever happens in the space
$\Irr (\mc H (\mc R,q) \rtimes \C [\Gamma,\natural])$ can be detected and understood
already by looking at tempered representations. From that, the bijectivity in the
tempered case and the multiplicity one property of the Langlands classification
(every standard module has a unique irreducible quotient, appearing with multiplicity
one, see \cite[Theorem 2.2.4]{SolAHA}), we obtain the bijectivity of \eqref{eq:1.17}
in general.
\end{proof}

\section{Topological K-theory}

\subsection{The $C^*$-completion of an affine Hecke algebra} \

In this paragraph we recall the structure of $C^*$-algebras associated to 
affine Hecke algebras. These deep results mainly stem from the work of Delorme--Opdam
\cite{Opd-Sp,DeOp1,DeOp2}.

Recall that $q$ is a positive parameter function for $\mc R$. We define
a *-operation and a trace on $\mc H (\mc R,q)$ by
\begin{align*}
& \big( \sum\nolimits_{w \in W^e} c_w N_w \big)^* = 
\sum\nolimits_{w \in W^e} \overline{c_w} N_{w^{-1}} , \\
& \tau \big( \sum\nolimits_{w \in W^e} c_w N_w \big) = c_e .
\end{align*}
Since $q(s_\alpha) > 0$, * preserves the relations \eqref{eq:multrules}
and defines an anti-involution of $\mc H (\mc R,q)$. The set $\{ N_w : w \in W^e\}$
is an orthonormal basis of $\mc H (\mc R,q)$ for the inner product
\[
\inp{h_1}{h_2} = \tau (h_1^* h_2) . 
\]
This gives $\mc H (\mc R,q)$ the structure of a Hilbert algebra.
The Hilbert space completion $L^2 (\mc R)$ of $\mc H (\mc R,q)$ is a module 
over $\mc H (\mc R,q)$, via left multiplication. Moreover every 
$h \in \mc H (\mc R,q)$ acts as a bounded linear operator \cite[Lemma 2.3]{Opd-Sp}. 
The reduced $C^*$-algebra of $\mc H (\mc R,q)$ \cite[\S 2.4]{Opd-Sp}, denoted
$C_r^* (\mc R,q)$, is defined as the closure of $\mc H (\mc R,q)$ in the algebra of 
bounded linear operators on $L^2 (\mc R)$. 

As in \eqref{eq:1.22}, we can extend this to a $C^*$-algebra 
$C_r^* (\mc R,q) \rtimes \Gamma$, provided that $q$ is $\Gamma$-invariant.
We will not bother about twisted group algebras $\C [\Gamma,\natural]$ in this
chapter, for with the technique from \eqref{eq:1.16} is easy to generalize our
results to that setting and the context of $C^*$-algebras crossed products with
groups look much more natural.

Let us recall some background about $C_r^* (\mc R,q) \rtimes \Gamma$, mainly
from \cite{Opd-Sp,SolAHA}. It follows from \cite[Corollary 5.7]{DeOp1} that it is a
finite type I $C^*$-algebra and that $\Irr (C_r^* (\mc R,q))$ is precisely the tempered 
part of $\Irr (\mc H (\mc R,q))$. The structure of $C_r^* (\mc R,q) \rtimes \Gamma$
is described in terms of parabolically induced representations. As induction data
we use triples $(P,\delta,t)$ where:
\begin{itemize}
\item $P \subset \Delta$;
\item $\delta$ is an irreducible discrete series representation of $\mc H_P$;
\item $t \in T^P$.
\end{itemize}
We regard two triples $(P,\delta,t)$ and $(P',\delta',t')$ as equivalent if
$P = P', t = t'$ and $\delta \cong \delta'$.
Notice that $\mc H_P$ comes from a semisimple root datum, so it can have discrete
series representations. For every $t \in T^P$ there exists a surjection
$\phi_t : \mc H^P \to \mc H_P$, which combines the projection $X \to X_P$ with
evaluation at $t$. To such a triple $(P,\delta,t)$ we associate the 
$\mc H \rtimes \Gamma$-representation
\[
\pi^\Gamma (P,\delta,t) = \ind_{\mc H^P}^{\mc H \rtimes \Gamma} (\delta \circ \phi_t). 
\]
(When $\Gamma = 1$, we often suppress it from these and similar notations.)
For $t \in T_\un^P = T^P \cap T_\un$ these representations extend continuously
to the respective $C^*$-completions of the involved algebras. Let $\Xi_\un$ be
the set of triples $(P,\delta,t)$ as above, such that moreover $t \in T_\un$.
Considering $P$ and $\delta$ as discrete variables, we regard $\Xi_\un$ as a
disjoint union of finitely many compact real tori (of different dimensions).

Let $\mc V_\Xi^\Gamma$ be the vector bundle over $\Xi_\un$, whose fibre at
$(P,\delta,t)$ is the vector space underlying $\pi^\Gamma (P,\delta,t)$. That
vector space is independent of $t$, so the vector bundle is trivial. Let
$\End (\mc V_\Xi^\Gamma)$ be the algebra bundle with fibres $\End_\C \big( \pi^\Gamma 
(P,\delta,t) \big)$. Every element of $C_r^* (\mc R,q) \rtimes \Gamma$ naturally
defines a continuous section of $\End (\mc V_\Xi^\Gamma)$.

There exists a finite groupoid $\mc G$ which acts on $\End (\mc V_\Xi^\Gamma)$.
It is made from elements of $W\rtimes \Gamma$ and of $K_P := T_P \cap T^P$.
More precisely, its base space is the power set of $\Delta$, and for
$P,Q \subseteq \Delta$ the collection of arrows from $P$ to $Q$ is
\begin{equation}\label{eq:GPQ}
\mc G_{PQ} = \{ (g,u) : g \in \Gamma \ltimes W , u \in K_P , g (P) = Q \} .
\end{equation}
Whenever it is defined, the multiplication in $\mc G$ is 
\[
(g',u') \cdot (g,u) = (g' g, g^{-1} (u') u) .
\]
In particular, writing $W\Gamma (P,P) = \{ w \in W\Gamma : w (P) = P \}$, 
we have the group
\begin{equation}
\mc G_{PP} = W\Gamma (P,P) \rtimes K_P . 
\end{equation}
Usually we will write elements of $\mc G$ simply as $gu$. For $\gamma \in \Gamma W$ 
with $\gamma (P) = Q \subset \Delta$ there are algebra isomorphisms
\begin{equation}\label{eq:psigamma}
\begin{array}{llcl}
\psi_\gamma : \mc H_P \to \mc H_Q , &
\theta_{x_P} N_w & \mapsto & \theta_{\gamma (x_P)} N_{\gamma w \gamma^{-1}} , \\
\psi_\gamma : \mc H^P \to \mc H^Q , &
\theta_x N_w & \mapsto & \theta_{\gamma x} N_{\gamma w \gamma^{-1}} .
\end{array}
\end{equation}
The groupoid $\mc G$ acts from the left on $\Xi_\un$ by
\begin{equation}\label{eq:2.3}
(g,u) \cdot (P,\delta,t) := (g (P),\delta \circ \psi_u^{-1} \circ \psi_g^{-1},g (ut)) ,
\end{equation}
the action being defined if and only if $g (P) \subset \Delta$.
Suppose that $g(P) = Q \subset \Delta$ and 
$\delta' \cong \delta \circ \psi_u^{-1} \circ \psi_g^{-1}$. 
By \cite[Theorem 4.33]{Opd-Sp} and \cite[Theorem 3.1.5]{SolAHA} there exists an
intertwining operator
\begin{equation}\label{eq:2.9}
\pi^\Gamma (gu,P,\delta,t) \in \Hom_{\mc H (\mc R,q) \rtimes \Gamma} 
\big( \pi^\Gamma (P,\delta,t) , \pi^\Gamma (Q,\delta',g (ut)) \big) ,
\end{equation}
which depends algebraically on $t \in T^P_\un$.
Then the action of $\mc G$ on the continuous sections 
$C (\Xi_\un ;\End (\mc V_\Xi^\Gamma))$ is given by
\begin{equation}\label{eq:2.1}
(g \cdot f) (\xi) = \pi^\Gamma (g,g^{-1} \xi) f (g^{-1} \xi) \pi^\Gamma (g,g^{-1} \xi )^{-1}
\qquad g \in \mc G_{PQ}, \xi = (Q,\delta',t').
\end{equation}

\begin{thm}\label{thm:2.1}
\textup{(\cite[Corollary 5.7]{DeOp1} and \cite[Theorem 3.2.2]{SolAHA})} \\
There exists a canonical isomorphism of $C^*$-algebras
\[
C_r^* (\mc R,q) \rtimes \Gamma \isom C \big( \Xi_\un ; \End (\mc V_\Xi^\Gamma) \big)^{\mc G}.
\]
\end{thm}
For $q=1$ this simplifies to the well-known isomorphism
\begin{equation}\label{eq:2.2}
C^*_r (\mc R,1) \rtimes \Gamma = C (T_\un) \rtimes W \Gamma \isom 
C \big( T_\un ; \End_\C (\C [W \Gamma]) \big)^{W \Gamma} . 
\end{equation}
Let $\mc G_{P,\delta}$ be the setwise stabilizer of $(P,\delta,T^P_\un)$ in the group 
$\mc G_{PP}$. Let $(P,\delta) / \mc G$ be a set of representatives for the action of $\mc G$
on pairs $(P,\delta)$ obtained from \eqref{eq:2.3}. Theorem \ref{thm:2.1} can be rephrased
as an isomorphism
\begin{equation}\label{eq:2.4}
C_r^* (\mc R,q) \rtimes \Gamma \isom \bigoplus\nolimits_{(P,\delta) / \mc G} 
C \big( T^P_\un ; \End_\C (\pi^\Gamma (P,\delta,t)) \big)^{\mc G_{P,\delta}} .
\end{equation}
Let us discuss the representation theory of $C_r^* (\mc R,q) \rtimes \Gamma$ 
(i.e. the tempered unitary representations of $\mc H (\mc R,q) \rtimes \Gamma)$) in more 
detail. Our approach, after Harish-Chandra and Opdam, starts with the discrete series of a 
parabolic subalgebra $\mc H (\mc R_P,q_P) = \mc H_P$. It is known from 
\cite[Lemma 3.31]{Opd-Sp} that the central character of any (irreducible) discrete series 
representation $\delta$  of $\mc H_P$ (a $W_P$-orbit in $T_P$) has a very specific property, 
it must consist of \emph{residual points} in $T_P$, with respect to $(\mc R_P,q_P)$. 

For $t \in T_P$ we write
\begin{align*}
& R_P^z (t) = \big\{ \alpha \in R_P : \alpha (t) \in \{ 1,-1\} \big\} , \\ 
& R_P^p (t) = \big\{ \alpha \in R_P : \alpha (t) \in \{ q(s_\alpha)^{1/2} 
q(s_\alpha t_\alpha)^{1/2}, -q(s_\alpha)^{1/2} q(s_\alpha t_\alpha)^{-1/2} \} \big\} .
\end{align*}
(We remark that there is only one irreducible root datum for which $q(s_\alpha t_\alpha)$
need not be equal to $q(s_\alpha)$, namely with $R = B_n$.) By definition $t \in T_P$
is residual if
\[
|R_P^p (t) | - |R_P^z (t)| = \dim_\C (T_P) = |P|. 
\]
Residuality depends in a subtle way on the parameters $q$. For instance, when $q=1$ and
$X_P \neq 0$, there are no residual points. Residual points have been classified in
\cite{HeOp}. It turns out that all the coordinates of a residual point $t$ are monomials
in the parameters $q(s)^{\pm 1/2} ,\; s \in S^\af$. Thus we can write $t = t (q^{1/2})$.

Let $\mc Q (\mc R)$ be the space of all maps $q : S^\af \to \R_{>0}$ such that 
$q(s) = q(s')$ if $s$ and $s'$ are conjugate in $X \rtimes W\Gamma$. Given $t = t (q^{1/2})$,
there is a Zariski-open subset of the real variety $\mc Q (\mc R)$ on which $t(q^{1/2})$
defines a residual point. For this reason we call the map 
\[
\mc Q (\mc R) \to T : q \mapsto t(q^{1/2}) 
\]
a generic residual point. We say that a parameter function $q \in \mc Q (\mc R)$ is generic
if all generic residual points for parabolic subalgebras $\mc H_P$ of $\mc H$ are
actually residual points for that $q$.

When there is only one free parameter in $q$, for instance when $R$ is of type 
$A, D$ or $E$, then every positive parameter function $q \neq 1$ is generic. On the other 
hand, when $R$ contains root systems of type $B,C,F$ or $G$, then usually no equal 
parameter function ($q(s) = q(s')$ for all $s,s' \in S^\af$) is generic.

The discrete series representations of $\mc H (\mc R_P,q_P)$ were classified in \cite{OpSo1}, 
at least when $R$ is irreducible and $q_P$ generic. Later the classification
was extended to the non-generic cases, along with an actual construction of the 
representations, in \cite{CiOp2}. Using these papers, it is in principle always possible
to find a set of representatives for the action of $\mc G$ on the pairs $(P,\delta)$
as in \eqref{eq:2.4}.

Now we describe a single direct summand 
$C \big( T^P_\un ; \End_\C (\pi^\Gamma (P,\delta,t)) \big)^{\mc G_{P,\delta}}$ 
of \eqref{eq:2.4} more explicitly. Fix $t \in T^P_\un$ and let $\mc G_\xi$ be the 
isotropy group of $\xi = (P,\delta,t)$ in $\mc G$. The intertwining operators
$\pi^\Gamma (g,\xi) ,\; g \in \mc G_\xi$ make $\pi^\Gamma (\xi)$ into a projective
$\mc G_\xi$-representation. Decompose it as
\[
\pi^\Gamma (\xi) = \bigoplus\nolimits_\rho \C^{m_\rho} \otimes V_\rho , 
\]
where $(\rho,V_\rho)$ runs through the set of (equivalence classes of) irreducible
projective $\mc G_\xi$-representations. From \eqref{eq:2.1} we see that the evaluation
at $t$ of any element of 
$C \big( T^P_\un ; \End_\C (\pi^\Gamma (P,\delta,t)) \big)^{\mc G_{P,\delta}}$
lies in 
\[
\End_{\mc G_\xi} (\pi^\Gamma (\xi)) \cong \bigoplus\nolimits_\rho \End_\C (\C^{m_\rho}) . 
\]
The action of $\mc G_\xi$ on $\pi^\Gamma (P,\delta,t)$ can be analysed further with the
theory of R-groups from \cite{DeOp2}. In that paper there is no group $\Gamma$, but
with the intertwining operators as in \cite[Theorem 3.1.5]{SolAHA} the extension to
the case with $\Gamma$ is straightforward. By \cite[Propositions 4.5 and 4.7]{DeOp2}
there exists a root system $R_\xi$ on which $\mc G_\xi$ acts, and an R-group
$\mf R_\xi = \mr{Stab}_{\mc G_\xi}(R_\xi \cap R_P^+)$, such that
\begin{equation}\label{eq:2.8}
\mc G_\xi = W(R_\xi) \rtimes \mf R_\xi .
\end{equation}
By \cite[Theorem 4.3.iv]{DeOp2} the intertwining operator $\pi^\Gamma (g,\xi)$ is a
scalar multiple of the identity if $g \in W(R_\xi)$. Hence 
\[
\End_{\mc G_\xi} (\pi^\Gamma (\xi)) = \End_{\mf R_\xi} (\pi^\Gamma (\xi)) .
\]
Moreover, by \cite[Theorem 5.4]{DeOp2} the operators 
\[
\pi^\Gamma (r,\xi) \in \End_\C (\pi^\Gamma (\xi)) ,\; r \in \mf R_\xi
\]
are linearly independent. To classify all irreducible representations of \\
$C \big( T^P_\un ; \End_\C (\pi^\Gamma (P,\delta,t)) \big)^{\mc G_{P,\delta}}$,
it remains to determine \eqref{eq:2.8} and to study $\pi^\Gamma (\xi)$ as a projective
$\mf R_\xi$-representation, for all $\xi = (P,\delta,t)$. In all cases that we will
encounter in this paper, $\mf R_\xi$ is abelian and $\pi^\Gamma (\xi)$ is actually
a linear $\mf R_\xi$-representation. Together with Theorem \ref{thm:1.5} this enables
us to determine $\Irr (C_r^* (\mc R,q) \rtimes \Gamma)$ in those cases.

\subsection{K-theory and equivariant cohomology} \

The computation of the topological K-theory of $C_r^* (\mc R,q) \rtimes \Gamma$ is the main
goal of this paper. It follows from \eqref{eq:2.4}, especially the compactness of $T^P_\un$, 
that the abelian group 
\[
K_* ( C_r^* (\mc R,q) \rtimes \Gamma ) = 
K_0 (C_r^* (\mc R,q) \rtimes \Gamma ) \oplus K_1 ( C_r^* (\mc R,q) \rtimes \Gamma )
\]
is finitely generated, see \cite[Lemma 5.1.3]{SolAHA} and its proof.
By \cite[Theorem 5.1.4]{SolAHA}, which relies on the study of the representation theory and
of parameter deformations of affine Hecke algebras in \cite{SolAHA}, 
$\Q \otimes_\Z K_* (C_r^* (\mc R,q) \rtimes \Gamma)$ does not depend on the parameters $q$. 
Combining this with the conclusions from Paragraph \ref{par:AHA}, we will deduce that also 
$K_* (C_r^* (\mc R,q) \rtimes \Gamma)$ itself is independent of $q$. 

Next we use equivariant cohomology and the equivariant Chern character to express 
$K_* (C_r^* (\mc R,q) \rtimes \Gamma)$ in terms of the cohomology of a sheaf on a CW-complex.
This is inspired by the equivariant Chern characters with values in Bredon cohomology
developed in \cite{Slo1,LuOl}. Our version also applies to certain non-commutative algebras,
and provides more information about the torsion elements than \cite{Slo1,LuOl}.\\

In \cite[Theorem 4.4.2]{SolAHA} an injective homomorphism of $C^*$-algebras
\[
\zeta_0 : C_r^* (\mc R,1) \rtimes \Gamma \longrightarrow
C_r^* (\mc R,q) \rtimes \Gamma
\]
was constructed, with the property
\[
\pi \circ \zeta_0 \cong \zeta^\vee (\pi) \qquad 
\text{for all } \pi \in \Mod_f (C_r^* (\mc R,q) \rtimes \Gamma) .
\]
\begin{thm}\label{thm:2.2}
The map $K_* (\zeta_0) : K_* ( C_r^* (\mc R,1) \rtimes \Gamma ) \longrightarrow
K_* (C_r^* (\mc R,q) \rtimes \Gamma)$ is an isomorphism.
\end{thm}
\begin{proof}
Let $u \in T_\un$. Then \eqref{eq:1.19} says that 
$\zeta^\vee$ provides a bijection between the Grothendieck group of finite length 
$C_r^* (\mc R,q) \rtimes \Gamma$-modules with $Z(\mc H (\mc R,q) \rtimes \Gamma)$-character 
in $W \Gamma u T_{rs}$ and the analogous group for $C_r^* (X \rtimes W) \rtimes \Gamma$.
For tempered modules $\zeta^\vee$ agrees with the map $\zeta^*$ from \cite[\S 2.3]{SolAHA}.

These $C^*$-completions have the same irreducible representations as the respective
Schwartz completions of these algebras (see \cite[\S 6]{Opd-Sp} or \cite[\S 3.2]{SolAHA}),
namely the irreducible tempered representations of the underlying affine Hecke algebras.
That follows from the comparison of Theorem \ref{thm:2.1} with its analogue for Schwartz
completions \cite[Theorem 3.2.2]{SolAHA}. With these translation steps we see that part (c) 
of \cite[Lemma 5.1.5]{SolAHA} holds. Then \cite[Lemma 5.1.5]{SolAHA} tells us
that also its part (a) holds, which is the statement of the theorem. 
\end{proof}

When we want to compute $K_* (C_r^* (\mc R,q) \rtimes \Gamma)$, we can use Theorem
\ref{thm:2.2} to replace $q$ by 1, then apply it another time to replace 1 by any
positive parameter function $q'$ we like. We will do the actual computation either when
$q=1$ or when $q$ is generic among all possible parameter functions.

In Section 3 we will encounter many root data $\mc R$ which are a product of root data
$\mc R_1$ and $\mc R_2$. If $\Gamma_i$ is a group acting on $\mc R_i$ in the usual
way, then $\Gamma := \Gamma_1 \times \Gamma_2$ acts on $\mc R$. In this case 
$C_r^* (\mc R,q) \rtimes \Gamma$ is defined as an algebra of bounded linear operators on 
\[
L^2 (\mc R) \otimes \C [\Gamma] = L^2 (\mc R_1) \otimes \C[\Gamma_1] \otimes
L^2 (\mc R_2) \otimes \C[\Gamma_2] .
\]
It is the closure of the algebraic tensor product of $C_r^* (\mc R_1,q_1) \rtimes \Gamma_1$
and $C_r^* (\mc R_2,q_2) \rtimes \Gamma_2$ in $B \big( L^2 (\mc R) \otimes \C [\Gamma] \big)$,
which means that
\begin{equation}
C_r^* (\mc R,q) \rtimes \Gamma = C_r^* (\mc R_1,q_1) \rtimes \Gamma_1 \otimes_{\mr{min}}
C_r^* (\mc R_2,q_2) \rtimes \Gamma_2 ,
\end{equation}
the minimal tensor product of $C^*$-algebras. These $C^*$-algebras are separable and of type I,
so the paper \cite{Scho} applies to them. The K\"unneth Theorem \cite{Scho} says that there
exists a natural $\Z / 2 \Z$-graded short exact sequence
\begin{multline}\label{eq:2.10}
0 \to K_* ( C_r^* (\mc R_1,q_1) \rtimes \Gamma_1 ) \otimes_\Z K_* ( C_r^* (\mc R_2,q_2) \rtimes 
\Gamma_2 ) \to K_* (C_r^* (\mc R,q) \rtimes \Gamma) \to \\
\mr{Tor}_\Z \big( K_* ( C_r^* (\mc R_1,q_1) \rtimes \Gamma_1 ), 
K_* ( C_r^* (\mc R_2,q_2) \rtimes \Gamma_2 ) \big) \to 0 .
\end{multline}
In particular, this becomes an isomorphism
\[
K_* ( C_r^* (\mc R_1,q_1) \rtimes \Gamma_1 ) \otimes_\Z K_* ( C_r^* (\mc R_2,q_2) \rtimes 
\Gamma_2 ) \isom K_* (C_r^* (\mc R,q) \rtimes \Gamma)  
\]
if $K_* ( C_r^* (\mc R_i,q_i) \rtimes \Gamma_i )$ has no torsion for $i =1,2$.
With \eqref{eq:2.10} we can often reduce the computation of K-groups to the case where $R$ 
is irreducible.

By \eqref{eq:2.2} and the Green--Julg Theorem \cite{Jul}
\[
K_* (C_r^* (\mc R,1) \rtimes \Gamma) = K_* (C(T_\un) \rtimes W\Gamma) \cong
K_*^{W\Gamma} (C(T_\un)) .
\]
Furthermore, by the equivariant Serre--Swan Theorem \cite[Theorem 2.3.1]{Phi}
\begin{equation}\label{eq:2.5}
K_*^{W\Gamma} (C(T_\un)) \cong K^*_{W \Gamma} (T_\un) . 
\end{equation}
Together with Theorem \ref{thm:2.1} we get
\begin{equation}\label{eq:2.6}
K_* (C_r^* (\mc R,q) \rtimes \Gamma) \cong K^*_{W\Gamma}(T_\un) . 
\end{equation}
The right hand side in \eqref{eq:2.5} and \eqref{eq:2.6} is just Atiyah's 
$W\Gamma$-equivariant K-theory of the compact Hausdorff space $T_\un$. 
Let $T_\un /\!/ W\Gamma$ be the extended quotient (see also Paragraph \ref{par:crossed}).
We recall from \cite[Theorem 1.19]{BaCo} that the equivariant 
Chern character gives a natural isomorphism
\begin{equation}\label{eq:2.7}
K^*_{W\Gamma} (T_\un) \otimes_\Z \C \isom H^* \big( T_\un /\!/ W\Gamma; \C \big) . 
\end{equation}
(Here $H^*$ could be many cohomology theories, in this paper we stick to \u{C}ech
cohomology.) With \eqref{eq:2.5} we find a canonical isomorphism
\begin{equation}\label{eq:2.12}
K_* (C_r^* (\mc R,q) \rtimes \Gamma) \otimes_\Z \C \cong
H^* ( T_\un /\!/ W\Gamma; \C) .
\end{equation}
In \eqref{eq:2.7} it is essential to use complex coefficients, so this does not 
tell us much about the torsion in $K_* (C_r^* (\mc R,q) \rtimes \Gamma)$. To study
the torsion elements better, we will compare the topological K-theory of relevant
$C^*$-algebras with a suitable version of equivariant cohomology from
\cite{Bre}. Let $\Sigma$ be a countable, locally finite and finite 
dimensional $G$-CW complex, where $G$ is a finite group.
Assume that all cells are oriented and that the action of $G$ 
preserves these orientations.

We define a category $\mc K$ whose objects are the finite 
subcomplexes of $\Sigma$. The morphisms from $K$ to $K'$ are the 
maps $K \to K' : x \to gx$ for $g \in G$ such that 
$gK \subset K'$. Now a local coefficient system on $\Sigma$ is a covariant functor 
from $\mc K$ to the category of abelian groups, and the group
$C^q (\Sigma;\mf L)$ of $q$-cochains is the set of all functions
$f$ on the $q$-cells of $\Sigma$ with the property that 
$f(\tau) \in \mf L(\tau)$ for all $\tau$.
Furthermore we define a coboundary map $\mr{d} : C^q (\Sigma 
;\mf L ) \to C^{q+1} (\Sigma ;\mf L )$ by
\begin{equation}\label{eq:2.53}
(\mr{d} f) (\sigma) = \sum\nolimits_{\tau \in \Sigma^{(q)}} [\tau : \sigma ] 
\mf L(\tau \to \sigma) f(\tau)
\end{equation}
where the sum runs over the set $\Sigma^{(q)}$ of all $q$-cells and the incidence
number $[\tau : \sigma ]$ is the degree of the attaching map 
from $\partial \sigma$ (the boundary of $\sigma$ in the 
standard topological sense) to $\tau / \partial \tau$. 
The group $G$ acts naturally on this complex by cochain maps so, 
for any $K \subset \Sigma \,,\, \left( C^* (K;\mf L )^G  ,\mr{d} \right)$ 
is a differential complex. We define 
the equivariant cohomology of $K$ with coefficients in $\mf L$ as
\begin{equation}
H^q_G (K;\mf L) := H^q \left( C^* (K;\mf L)^G ,\mr{d} \right)
\end{equation}
More generally for $K' \subset K ,\: C^* (K, K'; \mf L )$ 
is the kernel of the restriction map $C^* (K; \mf L) \to C^*(K'; \mf L)$ and 
\begin{equation}
H^q_G (K,K';\mf L ) = 
H^q \left( C^* (K,K';\mf L )^G ,\mr{d} \right)
\end{equation}
By construction there exists a local coefficient 
system $\mf L^G$ (more or less consisting of the $G$-invariant 
elements of $\mf L$) on the CW complex $\Sigma / G$ such that 
the differential complexes $(C^* (K,K';\mf L )^G ,\mr{d} )$ and
$(C^* (K/G, K'/G; \mf L^G ), \mr{d} )$ are isomorphic. 
Notice that $\mf L^G$ defines a sheaf over $\Sigma / G$ (with
the cells as cover), such that
\begin{equation}\label{eq:2.54}
H^q_G (K,K';\mf L ) \cong \check H^q \left( K/G, K'/G; \mf L^G \right) .
\end{equation}
Let $\Sigma^p$ be the $p$-skeleton of $\Sigma$. We capture all 
the above things in a spectral sequence $(E_r^{p,q})_{r \geq 1}$, 
degenerating already for $r \geq 2$, as follows: 
\begin{align}
E_1^{p,q} = H^{p+q}_G (\Sigma^p ,\Sigma^{p-1} ;\mf L ) &= 
 \left\{ \begin{array}{l@{\;\;\quad \mr{if} \quad}l}
   C^p (\Sigma ; \mf L )^G & q = 0 \\ 
   0                       & q > 0 
 \end{array} \right. \label{eq:2.55} \\
E_2^{p,q} &= \left\{ \begin{array}{l@{\qquad \mr{if} \quad}l}
   H^p_G (\Sigma; \mf L ) & q = 0 \\ 
   0                      & q > 0 
 \end{array} \right. \label{eq:2.56}
\end{align}
The differential $d_1^E$ is the composition
\begin{equation}\label{eq:2.57}
E_1^{p,q} \to C^{p+q} (\Sigma^p ; \mf L )^G \to E_1^{p+1,q}
\end{equation}
of the maps induced by the inclusion $(\Sigma^p ,\emptyset ) \to 
(\Sigma^p ,\Sigma^{p-1})$ and the coboundary d.

We are mostly interested in this cohomology theory for a particular
coefficient system, which we now define. Consider the Fr\'echet algebra
\begin{equation}\label{eq:2.15}
B = C(\Sigma ; M_N (\mh C )) = M_N (C (\Sigma )) .
\end{equation}
(It is a $C^*$-algebra if $\Sigma$ is compact.) 
We assume that we have $u_g \in B^\times$ such that 
\begin{equation}\label{eq:2.16}
g b (x) = u_g (x) b (g^{-1} x) u_g^{-1}(x)
\end{equation}
defines an action of $G$ on $B$. Then the invariants $B^G$ constitute
a Fr\'echet subalgebra of $B$. Notice that
by \eqref{eq:2.1} and \eqref{eq:2.4} the $C^*$-completion of an 
affine Hecke algebra is a direct sum of algebras of this form.

To associate a local coefficient system to $B^G$. 
we first assume that $K$ is connected. In that case we let
\begin{equation}\label{eq:2.62}
G_K := \{ g \in G : g x = x \quad \forall x \in K \}
\end{equation}
be the isotropy group of $K$ and we define $\mf L_u (K)$ 
to be the free abelian group on the (equivalence classes of) 
irreducible projective $G_K$-representations contained in 
$(\pi_x, \mh C^N)$, where $\pi_x (g) = u_g (x)$ for $g \in G_K , 
x \in K$. By the continuity of the $u_g$ we get the same group 
for any $x \in K$. If $K$ is not connected, then we let 
$\{K_i \}_i$ be its connected components, and we define
\begin{equation}\label{eq:2.63}
\mf L_u (K) = \prod\nolimits_i \mf L_u (K_i)
\end{equation}
Suppose that $g K \subset K'$ and that $\rho$ is a projective $G_K$-representation. 
Then we define a projective $G_{K'}$-representation by
\begin{equation}\label{eq:2.64}
\mf L_u (g: K \to K') \rho (g') = \rho (g^{-1} g' g) \quad g' \in G_{K'}
\end{equation}
If $h \in G$ gives the same map from $K$ to $K'$ as $g$ then 
$h^{-1} g \in G_K$ and
\begin{equation}
\mf L_u (h: K \to K') \rho (g') = \rho (h^{-1} g' h) =
\rho (h^{-1} g) \rho  (g^{-1} g' g) \rho (g^{-1} h)
\end{equation}
so $\mf L_u (h: K \to K') \rho$ is isomorphic to 
$\mf L_u (g: K \to K') \rho$ as a projective representation. 
This makes $\mf L_u$ into a functor.
We can regard $\mf{L}_u$ as a sheaf on $\Sigma$, where a section $s$ is continuous
on $U$ if and only if $s(K) |_{G_{K'}} = s(K')$ for every inclusion
$K \subset K' \subset U$.

\begin{ex}
Suppose that $u_g (x) = 1$ for all $x \in \Sigma, 
g \in G$. Then $\mf L_u$ and $\mf L_u^G$ are the constant 
sheaves $\mh Z$ over $\Sigma$ and $\Sigma/G$ respectively, and 
\begin{equation}
H^*_G (\Sigma;\mf L_u) \cong \check H^*(\Sigma/G; \mh Z)
\end{equation}
is the ordinary cellular cohomology of $\Sigma/G$. Furthermore
\[
K_* ( B^G ) \cong K_* \big( C(\Sigma /G; M_N (\mh C )) \big) = K_* ( C(\Sigma / G)) ,
\]  
which is isomorphic to $\check H^* (\Sigma /G ;\mh Z )$ modulo torsion.
\end{ex}

It turns out that a relation like \eqref{eq:2.7}, between 
$K_* \left( B^G \right)$ and the \v Cech cohomology 
$H^* \left( \Sigma /G;\mf L_u^G \right)$ is valid in the generality of the algebras
$B^G$ from \eqref{eq:2.15} and \eqref{eq:2.16}. Notice that we do not require $\Sigma$
to be compact, we consider the K-theory of $B^G$ as a Fr\'echet algebra.
The skeleton of the CW complex $\Sigma$ gives rise to the following filtration:
\begin{equation}\label{eq:2.58} 
\begin{split}
&K_* \left( B^G \right) = K_*^0 \left( B^G \right) \supset 
K_*^1 \left( B^G \right) \supset \cdots \supset 
K_*^{\dim \Sigma} \left( B^G \right) \supset 
K_*^{1 + \dim \Sigma} \left( B^G \right) = 0 \\
&K_*^p \left( B^G \right) := \text{im}
\Big( K_* \big( C_0(\Sigma / \Sigma^{p-1} ;M_N (\mh C ) )^G \big) \to 
K_* \big( C(\Sigma; M_N(\mh C ) )^G \big) \Big) .
\end{split}
\end{equation}

\begin{thm}\label{thm:2.15}
The graded group associated with the filtration \eqref{eq:2.58}
is isomorphic to $\check H^* \left( \Sigma/G ; \mf L_u^G \right)$.
In particular there is an (unnatural) isomorphism
\begin{equation}\label{eq:2.85}
K_* \left( B^G \right) \otimes \mh Q \cong \check H^* 
\left( \Sigma /G; \mf L_u^G \otimes \mh Q \right)
\end{equation}
and
\[
K_* \left( B^G \right) \cong 
\check H^* \left( \Sigma /G; \mf L_u^G \right)
\]
if the right hand side is torsion free.
\end{thm}
\begin{proof}
For $p,r \geq 0$ we set $K(p,p+r) = K_* \big(C_0 (\Sigma^{p+r-1} / \Sigma^{p-1}; 
M_N (\C) )^G \big)$. When $p' \geq p$ and $p'+r' \geq p+r$, the map
\[
(\Sigma^{p+r-1},\Sigma^{p-1}) \to (\Sigma^{p'+r'-1},\Sigma^{p'-1})
\]
induces a group homomorphism $K(p',p'+r') \to K(p,p+r)$.
For any $s \geq 0$ the sequence 
\begin{equation}\label{eq:2.18}
(\Sigma^{p+r-1} , \Sigma^{p-1}) \to (\Sigma^{p+r+s-1},\Sigma^{p-1}) 
\to (\Sigma^{p+r+s-1},\Sigma^{p+r-1}) 
\end{equation}
gives rise to a connecting homomorphism $K(p,p+r) \to K(p+r,p+r+s)$.
Using \cite[Section XV.7]{CaEi} we construct a spectral sequence
$(F_r^p )_{r \geq 1}$ with terms:
\begin{equation}\label{eq:2.17}
\begin{array}{lllll}
F_1^p & = & K(p,p+1) / K(p,p) & = & K_* \left( C_0 (\Sigma^p / \Sigma^{p-1}; M_N
(\mh C ) )^G \right) , \\
F_\infty^p & = & K(p,\infty) / K(p+1,\infty) & = &
K_*^p \left( B^G \right) / K_*^{p+1} \left(B^G \right) .
\end{array}
\end{equation}
The entire setting is $\Z / 2\Z$-graded by the K-degree. We put 
\[
K^q(p,p+r) = 
K_{p+q} \big(C_0 (\Sigma^{p+r-1} / \Sigma^{p-1}; M_N (\C) )^G \big)
\]
and we refine \eqref{eq:2.17} to
\begin{equation}
\begin{aligned}
F_1^{p,q} &= K_{p+q} \left( C_0 (\Sigma^p / \Sigma^{p-1}; M_N
(\mh C ) )^G \right) , \\
F_\infty^{p,q} &= K_{p+q}^p \left( B^G \right) / 
K_{p+q}^{p+1} \left(B^G \right) .
\end{aligned}
\end{equation}
By the definition of a $G$-CW complex, the pointwise stabilizer of a $p$-cell 
$\sigma$ is equal to its setwise stabilizer in $G$. Consequently
\[
C_0 (\Sigma^p / \Sigma^{p-1}; M_N (\mh C ) )^G \cong
\prod\nolimits_{\sigma \in \Sigma^{(p)} / G} C_0 (\R^p) \otimes M_N (\C)^{G_\sigma}
\]
and $F_1^{p,1} = 0$. From Bott periodicity and the definition of $\mf{L}_u$ in 
\eqref{eq:2.63} we see that
\[
F_1^{p,0} \cong \prod_{\sigma \in \Sigma^{(p)} / G} \mf{L}_u (\sigma) \cong
\big( \prod_{\sigma \in \Sigma^{(p)}} \mf{L}_u (\sigma) \big)^G .
\]
Now replace $\mf L$ in \eqref{eq:2.55} by $\mf L_u$ and sum over
all $q$ to obtain $E_r^p$. If we compare the result with $F_1^p = F_1^{p,0} \oplus
F_1^{p,1}$ we see that $E_1^p \cong F_1^p$. So we get a diagram
\begin{equation}\label{eq:2.61}
\begin{array}{ccc}
F_1^{p,q} & \xrightarrow{\quad d_1^F \quad} & F_1^{p+1,q} \\[1mm]
\cong & & \cong \\
\prod_{n \in \mh Z} E_1^{p,q+2n} & \xrightarrow{\quad 
  d_1^E \quad} & \prod_{n \in \mh Z} E_1^{p+1,q+2n}
\end{array}
\end{equation}
The differential $d_1^F$ for $F_1^*$ is induced from the 
construction of a mapping cone of a Puppe sequence in the 
category of $C^*$-algebras, coming from \eqref{eq:2.18}. 
This is the noncommutative counterpart of the construction of 
the differential in cellular cohomology, so by naturality $d_1^F$ 
corresponds to $d_1^E$ under the above isomorphism. Therefore 
the spectral sequences $E_r^p$ and $F_r^p$ are isomorphic, and
in particular $F_r^p$ degenerates for $r \geq 2$. Now the 
isomorphism \eqref{eq:2.85} follows from \eqref{eq:2.54}.

If $\check H^* \left( \Sigma /G; \mf L_u^G \right)$ is torsion free, then every term 
$E_\infty^p \cong F_\infty^p$ must be torsion free. 
Hence in this case both $K_* \left( B^G \right)$ and 
$\check H^* \left( \Sigma /G; \mf L_u^G \right)$ are free 
abelian groups, of the same rank. 
\end{proof}

Theorem \ref{thm:2.15} allows us to reduce the computations of $K_* (C_r^* (\mc R,q))$
to \v Cech cohomology, where a lot of tools are available. For several root data it
is easiest to look at the case $q=1$, for which we will develop more machinery in the
next paragraph.
For some other root data (in particular of type $PGL_n$) it is more convenient to study 
$K_* (C_r^* (\mc R,q))$ with $q \neq 1$, for then there are fewer possibilities for 
torsion elements, compared to $q=1$. In those cases we need the full force of 
Theorem \ref{thm:2.15}.

\subsection{Crossed products} \
\label{par:crossed}

In the special case of crossed products the technique from Theorem \ref{thm:2.15}
can be improved. A crucial role will be played by the extended quotient, whose
definition we recall now. Let $G$ be a finite group $G$ acting on a topological space 
$\Sigma$. We define
\[
\widetilde{\Sigma} = \{ (g,t) \in G \times T_\un : g (t) = t \} ,
\]
a closed subset of the topological space $G \times \Sigma$.
The group $G$ acts on $\widetilde{\Sigma}$ by
\[
g (g',t) = (g g' g^{-1},g(t)) . 
\]
The (geometric) extended quotient of $\Sigma$ by $G$ is defined as
\begin{equation}\label{eq:2.11}
\Sigma /\!/ G = \widetilde{\Sigma} / G .
\end{equation}
It decomposes as
\begin{equation}\label{eq:2.13}
\Sigma /\!/ G = \bigsqcup\nolimits_{g \in cc (G)} \Sigma^w / Z_G (g) ,
\end{equation}
where $cc(G)$ denotes a set of representatives for the conjugacy classes in $G$. 

We will develop a method that allows one to pass from the $G$-equivariant
K-theory of $\Sigma$ to the integral cohomology of $\Sigma /\!/ G$. However,
it does not work automatically, we require that the cohomology is torsion-free
and that all $G$-isotropy groups of points of $\Sigma$ are Weyl groups (and it
uses some of our earlier results on the representation rings of Weyl groups).

From now on we assume that $\Sigma$ is a smooth manifold (possibly
with boundary) on which $G$ acts smoothly. According to \cite{Ill} $\Sigma$ 
also admits the structure of a countable, locally finite, finite dimensional 
$G$-simplicial complex.
The crossed product $C(\Sigma) \rtimes G$ fits in the framework of
\eqref{eq:2.15} and \eqref{eq:2.16} by the isomorphisms
\begin{equation}\label{eq:2.14}
C(\Sigma ) \rtimes G \cong C \big( \Sigma ; \End_\C (\C[G]) \big)^G = B^G .
\end{equation}
In this case $u_g (x)$ is right multiplication by $g^{-1}$ and $\pi_x$ is the 
direct sum of $[G:G_x]$ copies
of the regular representation of $G_x$. It is not hard to see that
$\mf L_u^G \otimes_\Z \C$ is isomorphic to the direct image of the 
constant sheaf $\C$ on $\widetilde \Sigma$, under the 
canonical map $pr: \widetilde \Sigma / G \to \Sigma / G$. 
Since $pr$ is finite to one there are no topological complications, 
and we get an isomorphism
\begin{equation}\label{eq:2.65}
H^*_G (\Sigma;\mf L_u \otimes \C) \cong \check H^* (\Sigma / G; \mf L_u^G \otimes_\Z \C)
\cong \check H^* \big( \widetilde \Sigma /G; \C \big)
\end{equation}
From this one can recover \eqref{eq:2.7}. Unfortunately this approach does not
automatically lead to an isomorphism between $\check H^* (\Sigma /G;\mf L_u^G)$ and
$\check H^* (\widetilde \Sigma / G; \Z)$, for $\mf L_u^G$ need not be isomorphic 
to the direct image of the constant sheaf $\Z$ under $pr$. 

Sometimes this can be approached better via a dual homology theory.
Let $C_q (\Sigma ; \mf{L}_u)$ be the subgroup of $C^q (\Sigma;\mf{L}_u)$ consisting
of functions supported on finitely many $q$-cells.
The graded $\Z$-module $C_* (\Sigma ; \mf{L}_u)$ admits a $G$-equivariant boundary map,
which in the notation of \eqref{eq:2.53} can be written as
\begin{align*}
& \partial : C^{q+1} (\Sigma ; \mf{L}_u) \to C^q (\Sigma ; \mf{L}_u) ,\\
& (\partial f) (\tau) = \sum\nolimits_{\sigma \in \Sigma^{(q+1)}} 
[\tau:\sigma] \, \ind_{G_\sigma}^{G_\tau} (f(\sigma)) .
\end{align*}
This is a natural perfect pairing on each $\mf{L}_u (\sigma) \cong R_\Z (G_\sigma)$,
since $G_\sigma$ is a finite group. With that one sees that the differential complex 
$(C^* (\Sigma ; \mf{L}_u),\mr{d})$ is isomorphic to
$\Hom_\Z \big( (C^* (\Sigma ; \mf{L}_u), \partial) , \Z \big)$. This persists to the
$G$-invariants:
\begin{equation}\label{eq:2.19}
\big( C^* (\Sigma / G ; \mf{L}_u^G ),\mr{d} \big) \cong 
\Hom_\Z \big( (C^* (\Sigma / G; \mf{L}_u^G ), \partial ) , \Z \big) .
\end{equation}
Suppose now that $\Sigma$ is a manifold on which the finite group $G$ acts smoothly.
For $t \in \Sigma$ the isotropy $G_t$ acts $\R$-linearly on the tangent space 
$T_t (\Sigma)$. We say that $G_t$ is a Weyl group if it is the Weyl group of some
root system in $T_t (\Sigma)$. 

\begin{thm}\label{thm:2.3}
Let $G$ be a finite group acting smoothly on a manifold $\Sigma$. 
\enuma{
\item Suppose that $G_t$ is a Weyl group for all $t \in \Sigma$. Then 
\[
H_i \big( C_* (\Sigma / G; \mf{L}_u^G ), \partial \big) \cong H_i (\Sigma /\!/ G; \Z)
\text{ for all } i \in \Z_{\geq 0} .
\]
\item Suppose that the conclusion of part (a) holds, and that  
$H^* (\Sigma /\!/ G;\Z)$ is torsion-free. Then
\[
K_* (C(\Sigma) \rtimes G) \cong H^* (\Sigma /\!/ G; \Z) . 
\]
}
\end{thm}
\begin{proof}
(a) For every subgroup $H \subset G$ the set of fixpoints $\Sigma^H$ is a submanifold
of $\Sigma$ \cite[Lemma 4.1]{BaCo}. It follows that for every $g \in cc(G)$ 
and every connected component $\Sigma_i^g$ 
of $\Sigma^g$ the map $t \mapsto G_t$ is constant on an open dense subset of $\Sigma_i^g$.
Pick a point $t_i$ in this dense subset of $\Sigma_i^g$ and write $G_{t_i} = W_i$.
By assumption $W_i$ is a Weyl group and $G_t \supset W_i$ for all $t \in \Sigma_i^g$.

For a cell $\tau$ and $t \in \tau \setminus \partial \tau$ we have $G_\tau = G_t$.
Using Proposition \ref{prop:1.3} we define, for $t \in \Sigma^g_i, 
t \in \tau \setminus \partial \tau$,
\begin{equation}\label{eq:2.20}
s(g,t) = s(g,\tau) = \ind_{W_i}^{G_\tau} (H(u_g,\rho_g )) .
\end{equation}
We may and will assume that $s(g,h \tau) = h \cdot s(g,\tau)$ for all $h \in Z_G (g)$.
This extends uniquely to a $G$-equivariant map $\Sigma^g \to
\bigcup_{\tau \subset \Sigma^g} R_\Z (G_\tau)$, and hence defines an element 
$s (g) \in C^* (\Sigma;\mf{L}_u^G)$. Thus $s(g)$ is nonzero at $Gt \in \Sigma / G$ if
and only if $Gt \cap \Sigma^g$ is nonempty. 

The $s(g)$ with $g \in cc(G)$ yield precisely one representation for each element of
the extended quotient 
\[
\Sigma /\!/ G = \bigsqcup\nolimits_{g \in cc(G)} \Sigma^g / Z_G (g) .
\]
So for every $t \in \Sigma$ we get exactly $|cc(G_t)| = |\Irr (G_t)|$ representations
$s(g,t)$. By Proposition \ref{prop:1.8} the $s(g,t)$ with $g \in cc(G)$ and
$t \in G \Sigma^g$ form a $\Z$-basis of the representation ring of the Weyl group $G_t$.
This also shows that for $t \in \tau \setminus \partial \tau$ the set
\begin{equation}\label{eq:2.21}
\{ h \cdot \ind_{W_i}^{G_\sigma} (H(u_g,\rho_g)) : 
\sigma \subset \Sigma^g, h \in G_\tau \backslash G, h \sigma = \tau \}
\end{equation}
is linearly independent in $R_\Z (G_t) = R_\Z (G_\tau)$.

Let $\tau \otimes s(g,\tau)$ with $\tau \subset \Sigma^g$ be the terms of which
$s(g)$ is made. Then \eqref{eq:2.20} entails that the span
of the $\tau \otimes s(g,\tau)$ forms a sub-chain complex $C(g,\Sigma)$ of
$(C_* (\Sigma / G; \mf{L}_u^G), \partial )$ and \eqref{eq:2.21} implies
that $C(g,\Sigma)$ is isomorphic to the cellular homology complex \\
$C_* (\Sigma^g / Z_G (g);\Z)$.
Since the $s(g,t)$ form a basis of $R_\Z (G_t)$ for every $t \in \Sigma$, 
\[
C_* (\Sigma / G; \mf{L}_u^G) = \bigoplus\nolimits_{g \in cc(G)} C(g,\Sigma) .
\]
The claim about the homology of $(C_* (\Sigma / G; \mf{L}_u^G), \partial )$ follows.\\
(b) In the absence of torsion, the Universal Coefficient Theorem says that the dual
of the homology of a different complex is naturally isomorphic to the cohomology 
of the dual complex. This gives the horizontal isomorphisms in the following commutative
diagram:
\begin{equation}\label{eq:2.22}
\begin{array}{ccc}
H^* (\Sigma /\!/ G; \Z) & \isom & \Hom_\Z (H_* (\Sigma /\!/ G; \Z), \Z) \\
\uparrow & & \downarrow \\
H^* (C^* (\Sigma / G; \mf{L}_u^G), \mr{d} ) & \isom & 
\Hom_\Z \big( H_* (C_* (\Sigma / G; \mf{L}_u^G), \partial ), \Z \big) .
\end{array} 
\end{equation}
By assumption the right vertical arrow is an isomorphism.
We define the left vertical arrow to be the isomorphism such that the diagram 
becomes commutative.
The lower left corner of \eqref{eq:2.22} is $H_G^* (\Sigma;\mf L)$, which by
Theorem \ref{thm:2.15} is isomorphic to $K_* (C(\Sigma) \rtimes G)$.
\end{proof}

Let us return to the case of $C(T_\un ) \rtimes W = C_r^* (W^e)$, where
$T_\un, W$ and $W^e$ come from a root datum $\mc R$. Then $W$ acts by algebraic
group automorphisms on the compact torus $T_\un$.  

\begin{cor}\label{cor:2.5}
Let $\mc R$ be the root datum of a reductive algebraic group with simply 
connected derived group, and assume that $H^* (T_\un /\!/ W ;\Z)$ is torsion-free. 
Then, for any positive parameter function $q$, 
\[
K_* (C_r^* (\mc R,q)) \cong H^* (T_\un /\!/ W ;\Z).
\]
\end{cor}
\begin{proof}
Let $\mc R$ be the root datum of $(\mc G (\C),T)$. By Steinberg's connectedness 
theorem \cite{Ste} the group $Z_{\mc G (\C)}(t)$ is connected for every $t \in T$.
Hence $W_t = W (Z_{\mc G (\C)}(t), T)$ is always a Weyl group.
Now Theorem \ref{thm:2.3} says that 
\[
H^* (T_\un /\!/ W ;\Z) \cong K_* (C(T_\un) \rtimes W) = K_* (C_r^* (\mc R,1)).
\]
Apply Theorem \ref{thm:2.2} to the right hand side.
\end{proof}

In fact Corollary \ref{cor:2.5} also applies to some other root data, for example
those of type $SO_{2n+1}$.

\section{Examples} \label{sec:exa}

In this section we will compute the topological K-theory of the $C^*$-Hecke algebras 
$C_r^* (\mc R,q)$ associated to common root data $\mc R$. As discussed after 
Theorem \ref{thm:2.2}, it suffices to do so for $q=1$ or for generic parameter functions.
For $q=1$ we will apply Theorem \ref{thm:2.3}, when that is possible.

Our approach for $q \neq 1$ will involve the following steps.
\begin{enumerate}
\item Explicitly write down the root datum and the associated
Weyl groups.\\

From \eqref{eq:2.4} we get a canonical decomposition
\begin{equation}\label{eq:6.27}
C_r^* (\mc R ,q) \rtimes \Gamma = \bigoplus\nolimits_P C_r^* (\mc R ,q)_P \rtimes \Gamma_P .
\end{equation}
where $P$ runs over a set of representatives for the action of $\mc G$ on the 
power set of $\Delta$ and $\Gamma_P$ is the setwise stabilizer of $P$ in $\Gamma$.
\item List a good set of $P$'s.\\

For every chosen $P$ we do the following:
\item Determine the root datum $\mc R_P$ and the residual points.
\item Determine the discrete series of $\mc H (\mc R_P ,q_P )$,
and all the relevant intertwining operators.
\item Describe $C_r^* (\mc R ,q)_P \rtimes \Gamma_P$ 
and its space of irreducible representations.
\item Calculate $K_* \big( C_r^* (\mc R ,q)_P \rtimes \Gamma_P \big)$.
\end{enumerate}
Often the final step can be reduced to commutative $C^*$-algebras. When this is not 
possible, we will transfer the problem to sheaf cohomology, via Theorem \ref{thm:2.15}.

\subsection{Type $GL_n$} \
\label{par:GLn}

The easiest root data to study are those associated with the reductive group $GL_n$. 
The right way to do this was shown by Plymen. From \cite[Lemma 5.3]{Ply1} we know that 
the topological $K$-groups of these affine Hecke algebras are free 
abelian, of a finite rank which is explicitly given. Strictly speaking, 
we do not really need to study this root datum, as we could just refer to Plymen's results.
Nevertheless, since many other examples rely on this case, we include an analysis.

From now on many things will be parametrized by partitions 
and permutations, so let us agree on some notations. We write 
partitions in decreasing order and abbreviate $(x)^3 = (x,x,x)$. 
A typical partition looks like 
\begin{equation}\label{eq:6.39}
\mu = (\mu_1 ,\mu_2 ,\ldots , \mu_d ) = 
(n)^{m_n} \cdots (2)^{m_2} (1)^{m_1}
\end{equation}
where some of the multiplicities $m_i$ may be 0. By 
$\mu \vdash n$ we mean that the weight of $\mu$ is 
\[
| \mu | = \mu_1 + \cdots + \mu_d = n
\]
The number of different $\mu_i$'s (i.e. the number of blocks in 
the diagram of $\mu$) will be denoted by $b (\mu )$ and the 
dual partition (obtained by reflecting the diagram of $\mu$) 
by $\mu^\vee$. Sometimes we abbreviate
\begin{equation}
\begin{aligned}
& \mr{gcd}(\mu) = \mr{gcd}(\mu_1 ,\ldots ,\mu_d ) \\
& \mu ! = \mu_1 ! \mu_2 ! \cdots \mu_d !
\end{aligned}
\end{equation}
With a such partition $\mu$ of $n$ we associate the permutation 
\[
\sigma (\mu ) = (1 2 \cdots \mu_1 ) (\mu_1 + 1 \cdots \mu_1 + 
\mu_2 ) \cdots (n+1- \mu_d \cdots n) \in S_n
\]
As is well known, this gives a bijection between partitions of 
$n$ and conjugacy classes in the symmetric group $S_n$. The
centralizer $Z_{S_n}(\sigma (\mu ))$ is generated by the cycles
\[
((\mu_1 + \cdots + \mu_i + 1) (\mu_1 + \cdots + \mu_i + 2) 
  \cdots (\mu_1 + \cdots + \mu_i + \mu_{i+1}))
\]
and the ``permutations of cycles of equal length''.
For example, if $\mu_1 = \mu_2$:
\begin{equation}\label{eq:6.40}
(1 \, \mu_1 + 1) (2 \, \mu_1 + 2) \cdots (\mu_1 \, 2\mu_1)
\end{equation}
Using the second presentation of $\mu$ this means that
\[
Z_{S_n}(\sigma(\mu)) \cong \prod\nolimits_{l=1}^n 
  (\mh Z / l \mh Z)^{m_l} \rtimes S_{m_l} .
\]
Let us recall the definition of $\mc R (GL_n )$ and the associated groups. 
Below $Q$ and $Q^\vee$ are the root and coroot lattices.
\[
\begin{aligned}
& X = \mh Z^n \quad Q = \{ x \in X : x_1 + \cdots + x_n = 0 \} \\
& Y = \mh Z^n \quad 
  Q^\vee = \{ y \in Y : y_1 + \cdots + y_n = 0 \} \\
& T = (\mh C^\times )^n \quad  t = (t(e_1),\ldots, t(e_n)) = 
  (t_1, \ldots, t_n) \\
& R = \{ e_i - e_j \in X : i \neq j \} ,\quad \alpha_0 = e_1 - e_n \\
& R^\vee = \{ e_i - e_j \in Y : i \neq j \} , \quad \alpha_0^\vee = e_1 - e_n \\
& s_i = s_{\alpha_i} = s_{e_i - e_{i+1}} \quad s_0 = t_{\alpha_0} s_{\alpha_0} = 
  t_{\alpha_1} s_{\alpha_0} t_{-\alpha_1} : x \to x + \alpha_0 - 
  \inp{\alpha_0^\vee}{x} \alpha_0 \\
& W = \langle s_1 ,\cdots, s_{n-1} | s_i^2 = (s_i s_{i+1})^3 = 
  (s_i s_j)^2 = e : |i - j| > 1 \rangle \cong S_n \\
& S^\af = \{ s_0 ,s_1 ,\ldots s_{n-1} \} \\
& W^\af = \langle s_0, W_0 | s_0^2 = (s_0 s_i)^2 = 
  (s_0 s_1)^3 = (s_0 s_{n-1})^3 = e \;\mr{if}\; 2 \leq i 
  \leq n-2 \rangle \\
& W^e = W^\af \rtimes \Omega \quad \Omega = \langle t_{e_1} 
  (1 \, 2 \cdots n) \rangle \cong \mh Z
\end{aligned}
\]
Because all roots of $R$ are conjugate, $s_0$ is conjugate to any
$s_i \in S^\af$. Hence for any label function we have
\[
q(s_0 ) = q(s_i ) := q
\]
Every point of $T$ is $W$-conjugate to one of the form $t = \big( (t_1)^{\mu_1} 
(t_{\mu_1 +1})^{\mu_2} \cdots (t_n)^{\mu_d} \big) \in T$ and
\begin{equation}\label{eq:6.17}
W_t = S_{\mu_1} \times S_{\mu_2} \times \cdots \times S_{\mu_d}.
\end{equation}

\begin{itemize}
\item {\large \textbf{case} $\mb{q = 1}$}
\end{itemize}

By \eqref{eq:2.12} and \eqref{eq:2.13} we have
\begin{equation}\label{eq:6.1}
K_* \big( C_r^* (W^e) \big) \otimes \mh C \cong \check H^* \big( 
\widetilde{T_\un} \big/ S_n ; \mh C \big) \cong 
\bigoplus\nolimits_{\mu \vdash n} \check H^* \big( T_\un^{\sigma (\mu)} 
\big/ Z_{S_n}(\sigma (\mu )) ; \mh C \big) .
\end{equation}
Therefore we want to determine $T_\un^{\sigma (\mu)} / 
Z_{S_n}(\sigma (\mu ))$. If $\mu$ is as in \eqref{eq:6.39} then
\begin{equation}
\begin{aligned}
& T^{\sigma (\mu )} = \{ (t_1)^{\mu_1} (t_{\mu_1+1})^{\mu_2} 
\cdots (t_n)^{\mu_d} \in T \} , \\
& T^{\sigma (\mu)} \big/ Z_{S_n}(\sigma (\mu )) \cong
  (\mh C^\times)^{m_n} / S_{m_n} \times \cdots \times 
  (\mh C^\times)^{m_1} / S_{m_1} ,
\end{aligned}
\end{equation}
where $S_{m_l}$ acts on $(\mh C^\times)^{m_l}$ by permuting
the coordinates. To handle this space we use the following nice,
elementary result, a proof of which can be found for example in 
\cite[Lemma 5.1]{Ply1}.

\begin{lem}\label{lem:6.1}
For any $m \in \mh N$ there is an isomorphism of algebraic
varieties
\[
(\mh C^\times )^m \big/ S_m \cong \mh C^{m-1} \times \mh C^\times
\]
\end{lem}

Consequently $T_\un^{\sigma (\mu)} \big/ Z_{S_n}(\sigma (\mu ))$
has the homotopy type of $(S^1)^{b (\mu )}$. In particular its integral
cohomology is torsion-free, so Corollary \ref{cor:2.5} is applicable.
It says that \eqref{eq:6.1} can be refined to
\begin{equation}
K_* (C_r^* (W^e)) \cong \bigoplus\nolimits_{\mu \vdash n} 
\check H^* \big( (S^1)^{b (\mu )} ; \Z \big) 
\cong \bigoplus\nolimits_{\mu \vdash n} \Z^{2^{b (\mu )}} .
\end{equation}

\begin{itemize}
\item {\large \textbf{generic, equal parameter case} $\mb q \neq 1$}
\end{itemize}
Inequivalent subsets of $\Delta$ are parametrized by partitions 
$\mu$ of $n$. For the typical partition \eqref{eq:6.39} we have
\[
\begin{aligned}
& P_\mu = \Delta \setminus \{\alpha_{\mu_1}, \alpha_{\mu_1 + 
\mu_2}, \ldots, \alpha_{n - \mu_d} \} \\
& R_{P_\mu} \cong (A_{n-1})^{m_n} \times \cdots \times (A_1)^{m_2} 
  \cong R_{P_\mu}^\vee \\
& X^{P_\mu} \cong \mh Z (e_1 + \cdots + e_{\mu_1} ) / \mu_1 + 
  \cdots + \mh Z (e_{n+1 - \mu_d} + \cdots + e_n) / \mu_d \\
& X_{P_\mu} \cong \big( \mh Z^n / \mh Z (e_1 + \cdots + e_n ) \big)^{m_n} + 
  \cdots +  \big( \mh Z^2 / \mh Z (e_1 + e_2 ) \big)^{m_2} \\
& Y^{P_\mu} = \mh Z (e_1 + \cdots + e_{\mu_1}) + \cdots + \mh Z 
  (e_{n+1 - \mu_d} + \cdots + e_n) \\
& Y_{P_\mu} = \{ y \in \mh Z^n : y_1 + \cdots + y_{\mu_1} = \cdots 
  = y_{n+1- \mu_d} + \cdots + y_n = 0 \} \\
& T^{P_\mu} = \{ (t_1)^{\mu_1} \cdots (t_n)^{\mu_d} \in T \} \\
& T_{P_\mu} = \{ t \in T : t_1 t_2 \cdots t_{\mu_1} = \cdots =
  t_{n+1 -\mu_d} \cdots t_n = 1 \} \\
& K_{P_\mu} = \{ t \in T^{P_\mu} : t_1^{\mu_1} = \cdots =  
  t_n^{\mu_d} = 1 \} \\
& W_{P_\mu} \cong (S_n )^{m_n} \times \cdots \times (S_2 )^{m_2}
  \qquad W(P_\mu ,P_\mu ) \cong S_{m_n} \times \cdots \times 
  S_{m_2} \times S_{m_1} \\
& \mc G_{P_\mu P_\mu} = K_{P_\mu} \rtimes W(P_\mu ,P_\mu ) \qquad
  Z_{S_n}(\sigma (\mu )) = W(P_\mu ,P_\mu ) \ltimes \prod\nolimits_{l=1}^n 
  (\mh Z / l \mh Z)^{m_l}
\end{aligned}
\] 
The $W_{P_\mu}$-orbits of residual points for $\mc H_{P_\mu}$ 
are parametrized by 
\[
K_{P_\mu} \big( ( q^{(\mu_1 - 1)/2} ,q^{(\mu_1 - 3)/2}, \ldots ,
q^{(1 - \mu_1 )/2}) \cdots (q^{(\mu_d - 1)/2} ,q^{(\mu_d - 3)/2},
\ldots ,q^{(1 - \mu_d )/2}) \big)
\]
This set is obviously in bijection with $K_{P_\mu}$, and indeed
the intertwiners $\pi (k), k \in K_{P_\mu}$ act on it by 
multiplication. From the classification of the discrete series we know
that here every residual point carries precisely one discrete series
representation, namely a twist of a Steinberg representation. 
The quickest way to see this is with the Kazhdan--Lusztig classification of 
irreducible representations of affine Hecke algebras with equal parameters, 
in particular \cite[Theorems 7.12 and 8.13]{KaLu}. This implies
\[
\begin{aligned}
& \bigcup\nolimits_\delta \big( P_\mu, \delta ,T^{P_\mu} \big) 
\big/ K_{P_\mu} \cong T^{P_\mu} , \\
& \bigcup\nolimits_\delta \big( P_\mu, \delta ,T^{P_\mu} \big) 
\big/ \mc G_{P_\mu P_\mu} \cong T^{P_\mu} \big/ W(P_\mu ,P_\mu ) =
T^{\sigma (\mu )} \big/ Z_{S_n}(\sigma (\mu )) .
\end{aligned}
\]
If a point $\xi = (P_\mu,\delta,t)$ has a nontrivial stabilizer $\mc G_\xi$,
then by the above this stabilizer is contained in
$W(P_\mu , P_\mu) \cong \prod_{l=1}^n S_{m_l}$. It is easily seen that this 
isotropy group is actually a Weyl group, and that it equals the group 
$W(R_\xi)$ from \eqref{eq:2.8}. In other words, all R-groups are trivial
for this root datum and $q \neq 1$, and all intertwining operators $\pi (g,\xi)$ from
a representation $\pi (\xi)$ to itself are scalar multiples of the identity.
So the action of $\mc W_{P_\mu P_\mu}$ on 
\begin{equation}\label{eq:6.48}
C \Big( \bigsqcup\nolimits_\delta T_\un^{P_\mu} ; M_{n! / \mu !} (\mh C ) \Big)
\end{equation}
is essentially only on $\bigsqcup_\delta T_u^{P_\mu}$ 
and the conjugation part doesn't really matter. In particular we deduce that 
\begin{equation}\label{eq:6.41}
C_r^* (\mc R ,q) \cong \bigoplus_{\mu \vdash n} M_{n! / \mu !} 
\Big( C \Big( \bigsqcup_\delta T_\un^{P_\mu} \Big) \Big)
\cong \bigoplus_{\mu \vdash n} M_{n! / \mu !} \big(
T_\un^{\sigma (\mu )} \big/ Z_{S_n}(\sigma (\mu )) \big) .
\end{equation}
In particular $C_r^* (\mc R ,q)$ is Morita-equivalent with the commutative
$C^*$-algebra of continuous functions on $T_\un /\!/ S_n$.
Similar results were obtained by completely different methods in \cite{Mis}. 

We remark that $\Irr ( C_r^* (\mc R,q) )$ has a clear relation with the elliptic
representation theory of symmetric groups. Every $\delta$ is essentially a
Steinberg representation, so 
\[
\zeta^\vee (\delta \circ \phi_t) \in \Mod \big( \mc O (T) \rtimes Z_{S_n} (\sigma (\mu)) \big)
\]
is given by the $\mc O(T)$-character $t$ and the sign representation of the 
Weyl group $Z_{S_n} (\sigma (\mu))_t$. Moreover the group $Z_{S_n} (\sigma (\mu))_t$
can be identified with $R(\xi)$, where $\xi = (P_\mu,\delta,t)$. Then
$\zeta^\vee (\pi (\xi)) = \ind_{W (R_\xi)}^{(S_n )_t} $(sign)
as $(S_n)_t$-representations, and this is exactly a member of the basis $R_\Z ((S_n)_t)$
exhibited in Proposition \ref{prop:1.8}.b.

Using the analysis from the case $q=1$ it follows that 
\begin{equation}\label{eq:6.3}
K_* \big( C_r^* (\mc R ,q) \big) \cong \bigoplus_{\mu \vdash n}
K^* \big( T_\un^{\sigma (\mu )} \big/ Z_{S_n}(\sigma (\mu )) \big)
\cong \bigoplus_{\mu \vdash n} K^* \big( (S^1)^{b(\mu )} \big) 
\cong \bigoplus_{\mu \vdash n} \mh Z^{2^{b(\mu )}} .
\end{equation}
Recall that the even cohomology of $(S^1)^b$ has the same dimension as its
odd cohomology, unless $b = 0$. The same holds for K-theory, and $b(\mu) = 0$
does not occur because $b(\mu)$ counts the number of different terms in a partition
of $n \geq 1$. So we can refine \eqref{eq:6.3} to
\begin{equation}
K_0 (C_r^* (\mc R,q)) = \bigoplus\nolimits_{\mu \vdash n} \mh Z^{2^{b(\mu ) - 1}} ,\quad
K_1 (C_r^* (\mc R,q)) = \bigoplus\nolimits_{\mu \vdash n} \mh Z^{2^{b(\mu ) - 1}} .
\end{equation}

\subsection{Type $SL_n$} \
\label{par:SLn}

The affine Hecke algebra associated to a root datum of type $SL_n$
describes the category of Iwahori--spherical representations of
$PGL_n (\Q_p)$. Since that is a subcategory of the Iwahori--spherical
representations of $GL_n (\Q_p)$, It can be expected this affine Hecke
algebra behaves very similarly to those in the previous paragraph. Indeed,
we will see that the calculations of the K-theory are essentially the
same as in Paragraph \ref{par:GLn}.

The root datum $\mc R (SL_n)$ is given by:
\[
\begin{aligned}
& X = \mh Z^n / \mh Z (e_1 + \cdots e_n) \cong 
  Q + ((e_1 + \cdots + e_n) /n - e_n) \\
& Q = \{ x \in \mh Z^n : x_1 + \cdots + x_n = 0 \} \\
& Y = Q^\vee = \{ y \in \mh Z^n : y_1 + \cdots + y_n = 0 \} \\ 
& T = \{ t \in (\mh C^\times)^n : t_1 \cdots t_n = 1 \} \quad
  t = (t(e_1),\ldots, t(e_n)) = (t_1, \ldots, t_n) \\
& R = \{ e_i - e_j \in X : i \neq j \} \quad \alpha_0 = e_1 - e_n \\
& R^\vee = \{ e_i - e_j \in Y : i \neq j \} \quad \alpha_0 = e_1 - e_n \\
& s_i = s_{\alpha_i} = s_{e_i - e_{i+1}} \quad s_0 = t_{\alpha_0} s_{\alpha_0} = 
  t_{\alpha_1} s_{\alpha_0} t_{-\alpha_1} : x \to x + \alpha_0 - 
  \inp{\alpha_0^\vee}{x} \alpha_0 \\
& W = \langle s_1 ,\cdots , s_{n-1} | s_i^2 = (s_i s_{i+1})^3 = 
  (s_i s_j)^2 = e \; \mr{if} \; \mr |i - j| > 1 \rangle \cong S_n \\
& S^\af = \{ s_0 ,s_1 ,\ldots ,s_{n-1} \} \\
& W^\af = \langle s_0, W_0 | s_0^2 = (s_0 s_i)^2 = 
  (s_0 s_1)^3 = (s_0 s_{n-1})^3 = e \;\mr{if}\; 2 \leq i 
  \leq n-2 \rangle \\
& W^e = W^\af \rtimes \Omega \quad 
  \Omega = \langle t_{e_1 - (e_1 + \cdots e_n)/n} (1 2 \cdots n)
  \rangle \cong \mh Z / n \mh Z 
\end{aligned}
\]
Because all roots are conjugate, $s_0$ is conjugate to any 
$s_i \in S^\af$, and for any label function
\[
q(s_0 ) = q(s_i ) = q .
\]
The $W$-stabilizer of $\big( (t_1)^{\mu_1} (t_{\mu_1+1})^{\mu_2}
\cdots (t_n)^{\mu_d} \big)$ is isomorphic to
$S_{\mu_1} \times \cdots \times S_{\mu_d}$. Generically there are
$n! \, n$ residual points, and they all satisfy $t(\alpha_i ) = q$
or $t(\alpha_i ) = q^{-1}$ for $1 \leq i < n$. These residual
points form $n$ conjugacy classes, unless $q = 1$.

\begin{itemize}
\item {\large \textbf{group case} $\mb{q = 1}$}
\end{itemize}
In view of \eqref{eq:2.12} and \eqref{eq:2.13} we want to determine 
$T_\un^{\sigma (\mu)} \big/ Z_{S_n}(\sigma (\mu ))$, where $\mu$
is any partition of $n$. Write it as in \eqref{eq:6.39}, then
\[
\begin{array}{lll}
\! T^{\sigma (\mu )} \! & = & \{ (t_1)^{\mu_1} (t_{\mu_1+1})^{\mu_2} 
  \cdots (t_n)^{\mu_d} \in T \} \\
& \cong & \{ (t_1)^{\mu_1} (t_{\mu_1+1})^{\mu_2} \cdots 
  (t_n)^{\mu_d} \in (\mh C^\times )^n \} / \mh C^\times \times 
  \{ (e^{2 \pi i k / n})^n : 0 \leq k < \mr{gcd}(\mu) \} , \\
\multicolumn{3}{l}{\! T^{\sigma (\mu)} \big/ Z_{S_n}(\sigma (\mu )) 
\quad \cong \quad
\left( (\mh C^\times)^{m_n} / S_{m_n} \times \cdots \times (
\mh C^\times)^{m_1} / S_{m_1} \right) \big/ \mh C^\times \times }\\
 & & \hspace{23mm} \{ (e^{2 \pi i k / n})^n : 0 \leq k < \mr{gcd}(\mu) \} .
\end{array}
\]
where $\mh C^\times $ acts diagonally. By Lemma \ref{lem:6.1}
each factor $(\mh C^\times)^{m_i} / S_{m_i}$ is homotopy 
equivalent to a circle. The induced action of $S^1 \subset 
\mh C^\times$ on this direct product of circles identifies with 
a direct product of rotations. Hence $T^{\sigma (\mu)} /
Z_{S_n}(\sigma (\mu ))$ is homotopy equivalent with 
$\mh T^{b(\mu ) - 1} \times \{\mr{gcd}(\mu ) \; \mr{points} \}$,
and the extended quotient $T /\!/ W$ has torsion-free cohomology.
By Corollary \ref{cor:2.5}
\begin{equation}\label{eq:6.11}
K_* (C_r^* (W^e)) \cong \Z^{d(n)}, \quad 
d(n) := \sum\nolimits_{\mu \vdash n} \mr{gcd}(\mu) 2^{b(\mu )-1} .
\end{equation}

\begin{itemize}
\item {\large \textbf{generic, equal parameter case} $\mb q \neq 1$}
\end{itemize}
Inequivalent subsets of $\Delta$ are parametrized by partitions 
$\mu$ of $n$. For the typical partition \eqref{eq:6.39} we put 
\[
\begin{aligned}
& P_\mu = \Delta \setminus \{\alpha_{\mu_1}, \alpha_{\mu_1 + 
\mu_2}, \ldots, \alpha_{n - \mu_d} \} \\
& R_{P_\mu} \cong (A_{n-1})^{m_n} \times \cdots \times (A_1)^{m_2}
  \cong R_{P_\mu}^\vee \\
& X^{P_\mu} \!\cong \left( \mh Z (e_1 + \!\cdots\! + e_{\mu_1}) / 
  \mu_1 + \!\cdots\! + \mh Z (e_{n+1 - \mu_d} + \!\cdots\! + e_n) 
  / \mu_d \right) \!\big/ \mh Z (e_1 + \!\cdots\! + e_n) / g \\
& X_{P_\mu} \cong \big( \mh Z^n / \mh Z (e_1 + \cdots + e_n ) \big)^{m_n} +
  \cdots + \big( \mh Z^2 / \mh Z (e_1 + e_2 ) \big)^{m_2} \\
& Y^{P_\mu} = \{ y \in \mh Z (e_1 + \cdots + e_{\mu_1}) + \cdots 
  + \mh Z (e_{n+1 - \mu_d} + \cdots + e_n) : y_1 + \cdots + y_n = 
  0 \} \\
& Y_{P_\mu} = \{ y \in Y : y_1 + \cdots + y_{\mu_1} = \cdots = 
  y_{n+1- \mu_d} + \cdots + y_n = 0 \} \\
& T^{P_\mu} = \{ (t_1)^{\mu_1} \cdots (t_n)^{\mu_d} \in T :
  t_1^{\mu_1 / g} \cdots t_n^{\mu_d / g} = 1 \} ,\quad g = \mr{gcd}(\mu) \\
& T_{P_\mu} = \{ t \in T : t_1 t_2 \cdots t_{\mu_1} = \cdots =
  t_{n+1 -\mu_d} \cdots t_n = 1 \} \\
& K_{P_\mu} = \{ t \in T^{P_\mu} : t_1^{\mu_1} = \cdots =  
  t_n^{\mu_d} = 1 \} \\
& W_{P_\mu} \cong (S_n )^{m_n} \times \cdots \times (S_2 )^{m_2}
  \qquad W(P_\mu ,P_\mu ) \cong S_{m_n} \times \cdots \times 
  S_{m_2} \times S_{m_1} \\
& \mc G_{P_\mu P_\mu} = K_{P_\mu} \rtimes W(P_\mu ,P_\mu ) \qquad
  Z_{S_n}(\sigma (\mu )) = W(P_\mu ,P_\mu ) \ltimes \prod\nolimits_{l=1}^n 
  (\mh Z / l \mh Z)^{m_l}
\end{aligned}
\]
\begin{thm}\label{thm:6.3}
For $q \neq 1$ the $C^*$-algebra $C_r^* (\mc R(SL_n),q)$ is Morita equivalent 
with the commutative algebra of continuous functions on $T_\un /\!/ W$.

Its K-theory is given by
\begin{align*}
& K_0 (C_r^* (\mc R,q)) = \bigoplus\nolimits_{\mu \vdash n , b(\mu) > 1} 
\mh Z^{\mr{gcd}(\mu) 2^{b(\mu ) - 2}} \oplus 
\bigoplus\nolimits_{\mu \vdash n , b(\mu) = 1} \Z^{\mr{gcd}(\mu)} , \\
& K_1 (C_r^* (\mc R,q)) = \bigoplus\nolimits_{\mu \vdash n, b(\mu) > 1} 
\mh Z^{\mr{gcd}(\mu) 2^{b(\mu ) - 2}} .
\end{align*}
\end{thm}
\begin{proof}
The $W_{P_\mu}$-orbits of residual points for $\mc H_{P_\mu}$ are
represented by the points
\begin{multline}\label{eq:6.50}
\big( ( q^{(\mu_1 - 1)/2} ,q^{(\mu_1 - 3)/2}, \ldots ,
q^{(1 - \mu_1 )/2}) \cdots (q^{(\mu_d - 1)/2} ,q^{(\mu_d - 3)/2},
\ldots ,q^{(1 - \mu_d )/2}) \big) \, \cdot \\ 
\big( (e^{2 \pi i k_1 / \mu_1} )^{\mu_1} \cdots (e^{2 \pi i k_d /
\mu_d} )^{\mu_d} \big) \quad , \quad 0 \leq k_i < \mu_i
\end{multline}
These points are in bijection with $K_{P_\mu} \times \mh Z / 
\mr{gcd}(\mu ) \mh Z$. Also $T^{\sigma (\mu )}$ consists of
exactly gcd$(\mu )$ components, one of which is $T^{P_\mu}$.
Just as in the type $GL_n$ case, this leads to
\[
\begin{aligned}
& \bigcup\nolimits_\delta \big( P_\mu, \delta ,T^{P_\mu} \big) 
\big/ K_{P_\mu} \cong T^{P_\mu} \times \mh Z / \mr{gcd}(\mu ) \mh Z 
\cong T^{\sigma (\mu )} , \\
& \bigcup\nolimits_\delta \big( P_\mu, \delta ,T^{P_\mu} \big) 
\big/ \mc W_{P_\mu P_\mu} \cong T^{\sigma (\mu )} \big/ 
Z_{S_n}(\sigma (\mu )) , \\
& C_r^* (\mc R ,q) \cong \bigoplus_{\mu \vdash n} M_{n! / \mu !} 
\Big( C \Big( \bigsqcup_\delta T_u^{P_\mu} \Big) \Big)
\cong \bigoplus_{\mu \vdash n} M_{n! / \mu !} \big(
T_u^{\sigma (\mu )} \big/ Z_{S_n}(\sigma (\mu )) \big) .
\end{aligned}
\]
The extended quotient $T_\un /\!/ W$ is $\bigsqcup_{\mu \vdash n} 
T_u^{\sigma (\mu )} \big/ Z_{S_n}(\sigma (\mu ))$, which gives the 
desired Morita equivalence. It follows that
\begin{equation}\label{eq:6.4}
K_* \big( C_r^* (\mc R ,q) \big) \cong \bigoplus_{\mu \vdash n}
K^* \big( T_u^{\sigma (\mu )} \big/ Z_{S_n}(\sigma (\mu )) \big)
\cong \bigoplus_{\mu \vdash n} K^* \big( (S^1)^{b(\mu ) -1} 
\big)^{\mr{gcd}(\mu)} .
\end{equation}
This a free abelian group of rank $d(n) = \sum_{\mu \vdash n} \mr{gcd}(\mu) 
2^{b(\mu )-1}$ with $b(\mu)$ as on page \pageref{eq:6.39}.
Since the even K-theory of $(S^1)^b$ has the same rank as the odd K-theory
unless $b=0$, \eqref{eq:6.4} leads to $K_0$ and $K_1$ as claimed.
\end{proof}

\subsection{Type $PGL_n$} \

The root datum for the algebraic group $PGL_n$ gives rise to:
\[
\begin{aligned}
& X = Q = \{ x \in \mh Z^n : x_1 + \cdots + x_n = 0 \} \\
& Q^\vee = \{ y \in \mh Z^n : y_1 + \cdots + y_n = 0 \} \\
& Y = \mh Z^n / \mh Z (e_1 + \cdots +e_n ) \cong 
  Q^\vee + ((e_1 + \cdots + e_n)/n - e_1) \\
& T = (\mh C^\times)^n / \mh C^\times \quad t = (t_1, \ldots, t_n) 
  = (t(e_1), \ldots, t(e_n)) \\
& R = \{ e_i - e_j \in X : i \neq j\}\quad \alpha_0 = e_1 - e_n \\
& R^\vee =  \{ e_i - e_j \in Y : i \neq j\} \quad \alpha_0 = e_1 - e_n \\
& s_i = s_{\alpha_i} = s_{e_i - e_{i+1}} \quad s_0 = t_{\alpha_0} s_{\alpha_0} : 
  x \to x + \alpha_0 - \inp{\alpha_0^\vee}{x} \alpha_0 \\
& W = \langle s_1 ,\cdots , s_{n-1} | s_i^2 = (s_i s_{i+1})^3 = 
  (s_i s_j)^2 = e \; \mr{if} \; \mr |i - j| > 1 \rangle \cong S_n \\
& S^\af = \{ s_0 ,s_1 ,\ldots ,s_{n-1} \} \quad
  \Omega = \{ e \} \\
& W^e = W^\af = \langle s_0, W_0 | s_0^2 = (s_0 s_i)^2 = 
  (s_0 s_1)^3 = (s_0 s_{n-1})^3 = e \;\mr{if}\; 2 \leq i 
  \leq n-2 \rangle 
\end{aligned}
\]
For $n>2 ,\; s_0$ is conjugate to $s_1$ in $W^\af$, for $n=2$ it is not.
So for $n>2$ there is only one parameter $q = q(s_i) \; 0 \leq i \leq n-1$,
whereas for $n=2 \; q_0$ may differ from $q_1$. In particular, for $n=2$ the
equal parameter function $q(s_0) = q(s_1)$ is not generic. Nevertheless, we will
only consider equal parameter functions in this paragraph, explicit computations
for the other parameter functions on $\mc R (PGL_2)$ can be found in 
\cite[\S 6.1]{SolThesis}.

For $q \neq 1$ there are $n!$ residual points. They form one 
$W$-orbit, and a typical residual point is
\[
\big( q^{(1-n)/2}, q^{(3-n)/2}, \ldots ,q^{(n-1)/2} \big)
\]
To determine the isotropy group of points of $T$ we have to be
careful. In general the $W$-stabilizer of 
\[
\big( (t_1)^{\mu_1} (t_{\mu_1 + 1})^{\mu_2} \cdots (t_n 
)^{\mu_d} \big) \in T
\] 
is isomorphic to 
\[
S_{\mu_1} \times S_{\mu_2} \times \cdots \times S_{\mu_d} \subset W .
\]
However, in some special cases the diagonal action of $\mh 
C^\times$ on $(\mh C^\times )^n$ gives rise to extra stabilizing elements.
Let $r$ be a divisor of $n ,\, k \in (\mh Z / r \mh Z)^\times$
and $\lambda = (\lambda_1 ,\ldots ,\lambda_l )$ a partition of
$n/r$. The isotropy group of
\begin{multline*}
\hspace{-3mm} (t_1)^{\lambda_1} (e^{2 \pi i k / r} t_1 )^{\lambda_1} \cdots
(e^{-2 \pi i k / r} t_1 )^{\lambda_1} (t_{r \lambda_1 + 1}
)^{\lambda_2} \cdots (e^{-2 \pi i k / r} t_{r \lambda_1 + 1} 
)^{\lambda_2} \cdots (e^{-2 \pi i k / r} t_n )^{\lambda_l} 
\end{multline*}
is isomorphic to 
\begin{equation}\label{eq:6.43}
S^r_{\lambda_1} \times S^r_{\lambda_2} \times \cdots \times
S^r_{\lambda_l} \rtimes \mh Z / r \mh Z .
\end{equation}
Explicitly the subgroup $\mh Z / r \mh Z$ is generated by
\begin{multline*}\label{eq:6.44}
(1 \; \lambda_1 + 1 \; 2 \lambda_1 + 1 \cdots (r-1)\lambda_1 + 1)
(2 \; \lambda_1 + 2 \; 2 \lambda_1 + 2 \cdots (r-1)\lambda_1 + 2)
\cdots (\lambda_1 \; 2 \lambda_1 \cdots r \lambda_1 ) \\
\cdots (n+1 - r \lambda_d \; n + 1 + (1-r) \lambda_d \cdots 
n+1+ (r-1) \lambda_d ) ( n+ (1-r) \lambda_d \, n +(2-r) 
\lambda_d \cdots n) ,
\end{multline*}
and it acts on every factor $S^r_{\lambda_j}$ in \eqref{eq:6.43}
by cyclic permutations. 

\begin{itemize}
\item {\large \textbf{case} $\mb{q = 1}$}
\end{itemize}
As we noted before, we have to analyse $T_\un^{\sigma (\mu)} \big/ 
Z_{S_n}(\sigma (\mu ))$. For the typical partition $\mu$ we have
\begin{equation}
T^{\sigma (\mu )} = \{ (t_1)^{\mu_1} (t_{\mu_1+1})^{\mu_2} 
\cdots (t_n)^{\mu_d} \} \big/ \mh C^\times \times 
\{ t : t(e_j) = e^{2 \pi i j k / g} ,\, 0 \leq k < g \} ,
\end{equation}
which is the disjoint union of $g = \mr{gcd}(\mu )$ complex tori
of dimension\\ $m_n + m_{n-1} + \cdots + m_1 - 1$. We obtain
\begin{multline}\label{eq:6.45}
T^{\sigma (\mu )} \big/ Z_{S_n}(\sigma (\mu )) \cong
\left( (\mh C^\times)^{m_n} / S_{m_n} \times \cdots \times 
(\mh C^\times)^{m_1} / S_{m_1} \right) \big/ \mh C^\times \times \\
\{ t : t(e_j) = e^{2 \pi i j k / g} ,\, 0 \leq k < g \} .
\end{multline}
Remarkably enough, these sets are diffeomorphic to the corresponding
sets for $\mc R (SL_n )$. We take advantage of this by reusing our 
deduction that \eqref{eq:6.45} is homotopy equivalent with 
$(S^1)^{b (\mu ) - 1} \times \{\mr{gcd}(\mu ) \; \mr{points} \}$. 
With \eqref{eq:2.12} we conclude that 
$K_* (C_r^* (W^e)) \otimes_\Z \C$ has dimension 
$d(n) = \sum\nolimits_{\mu \vdash n} \mr{gcd}(\mu) 2^{b(\mu )-1}$.

\begin{itemize}
\item {\large \textbf{equal parameter case} $\mb{q \neq 1}$}
\end{itemize}
This is noticeably different from the generic cases for
$\mc R (GL_n )$ and $\mc R (A_{n-1}^\vee )$ because
$C_r^* (\mc R (A_{n-1} ,q) )$ is not Morita equivalent to a
commutative $C^*$-algebra. Of course the inequivalent subsets
of $\Delta$ are still parametrized by partitions $\mu$ of $n$.
\[
\begin{aligned}
& P_\mu = \Delta \setminus \{\alpha_{\mu_1}, \alpha_{\mu_1 + 
\mu_2}, \ldots, \alpha_{n - \mu_d} \} \\
& R_{P_\mu} \cong (A_{n-1})^{m_n} \times \cdots \times 
  (A_1)^{m_2} \cong R_{P_\mu}^\vee \\
& X^{P_\mu} \cong \{ x \in \mh Z (e_1 + \cdots + e_{\mu_1}) / \mu_1 
  + \cdots + \mh Z (e_{n+1 - \mu_d} + \cdots + e_n) / \mu_d : \\
& \quad x_1 + \cdots + x_n = 0 \} \\ 
& X_{P_\mu} \cong \{ x \in \mh Z^{\mu_1} / \mh Z (e_1 + \cdots + 
  e_{\mu_1} ) + \cdots + \mh Z^{\mu_d} / \mh Z (e_{n+1-\mu_d} + 
  \cdots + e_n ) : \\
& \quad x_1 + \cdots + x_n \in g \mh Z / g \mh Z \} \\
& Y^{P_\mu} \cong \mh Z (e_1 + \cdots + e_{\mu_1}) + \cdots + \mh Z 
  (e_{n+1 - \mu_d} + \cdots + e_n) / \mh Z (e_1 + \cdots + e_n) \\
& Y_{P_\mu} \cong \{ y : y_1 + \cdots + y_{\mu_1} = \cdots = y_{n+
  1- \mu_d} + \cdots + y_n = 0 \}  / \mh Z (e_1 + \cdots e_n) \\
& T^{P_\mu} = \{ (t_1)^{\mu_1} \cdots (t_n)^{\mu_d} \} /
  \mh C^\times \\
& T_{P_\mu} = \{ t : t_1 t_2 \cdots t_{\mu_1} = \cdots = t_{n+1
  -\mu_d} \cdots t_n = 1 \} / \{ z \in \mh C : z^g = 1 \} \\
& K_{P_\mu} = \{ (t_1)^{\mu_1} \cdots (t_n)^{\mu_d} : t_1^{\mu_1}
  = \cdots = t_n^{\mu_d} = 1 \} / \{ z \in \mh C : z^g = 1 \} \\
& W_{P_\mu} \cong S_n^{m_n} \times S_{n-1}^{m_{n-1}} \times \cdots
  \times S_2^{m_2} \quad W(P_\mu ,P_\mu ) \cong
  S_{m_n} \times \cdots \times S_{m_2} \times S_{m_1} 
\end{aligned}
\]
We note that 
\[
T^{\sigma (\mu )} = T^{P_\mu} \times \{ t : t(e_j) = 
e^{2 \pi i j k / g} ,\, 0 \leq k < g \} .
\]
The $W_{P_\mu}$-orbits of residual points for $\mc H_{P_\mu}$ are
represented by the points of
\[
K_{P_\mu} \big( q^{(\mu_1 -1)/2} ,q^{(\mu_1 -3)/2}, \ldots
q^{(1-\mu_1 )/2} ,q^{(\mu_2 -1)/2} ,\ldots ,q^{(\mu_d -1)/2} ,
\ldots ,q^{(1-\mu_d)/2} \big) .
\]
Hence the intertwiners $\pi (k)$ with $k \in K_{P_\mu}$ permute
the set of discrete series representations of $\mc H_{P_\mu}$ faithfully, and
\[
\bigsqcup\nolimits_\delta \big( P_\mu ,\delta ,T^{P_\mu} \big) / 
K_{P_\mu} \cong T^{P_\mu} = \big( T^{\sigma(\mu)} \big)^\circ .
\]
Just before \eqref{eq:6.48} we saw that the
intertwiners for $\mc R (GL_n) ,q \neq 1$ have the property
\[
w(t) = t \; \Rightarrow \; \pi (w, P_\mu ,\delta ,t) = 1 .
\]
This implies that in our present setting we can have $w(t) = t$
and $\pi (w, P_\mu ,\delta ,t) \neq 1$ only if $w(t) = t$ does
not hold without taking the action of $\mh C^\times$ into account.

Let us classify such $w \in W(P_\mu ,P_\mu )$ and $t \in
T^{P_\mu}$ up to conjugacy. For a divisor $r$ of $g^\vee := 
\mr{gcd}(\mu^\vee )$ we have the partition \index{mu1r@$\mu^{1/r}$}
\[
\mu^{1/r} := (n r)^{m_n / r} \cdots (2 r)^{m_2 /r} (r)^{m_1 /r} .
\]
Notice that
\[
b (\mu^{1/r} ) = b (\mu ) = b (\mu^\vee ) .
\]
There exists a $\sigma \in S_n$ which is conjugate to $\sigma
(\mu^{1/r})$ and satisfies $\sigma^r = \sigma (\mu )$. We 
construct a particular such $\sigma$ as follows. If $r = g^\vee$
then (starting from the left) replace every block 
\[
(d+1 \; d+2 \cdots d + m )(d+1+m \cdots d + 2 m)
\cdots (d+(g^\vee -1)m \cdots d + g^\vee m )
\]
of $\sigma (\mu )$ by
\[
(d+1 \; d+1+m \cdots d+1+(g^\vee - 1)m \; 2 \; d+2+m \cdots
d+2+(g^\vee -1)m \; d+3 \cdots d+g^\vee m) .
\]
We denote the resulting element by $\sigma (\mu )^{1 / g^\vee}$,
and for general $r | g^\vee$ we define
\[
\sigma (\mu )^{1/r} := 
\big( \sigma (\mu )^{1 / g^\vee} \big)^{g^\vee /r} .
\]
Consider the cosets of subtori
\[
T^{P_\mu}_{r,k} := \big( T^{\sigma (\mu )^{1/r}} \big)^\circ 
\big( (1)^{g^\vee \mu_1 /r} (e^{2 \pi i k /r} )^{g^\vee 
\mu_{1 + g^\vee /r} /r} \cdots (e^{-2 \pi i k /r} )^{g^\vee \mu_d
/r} \big) \qquad k \in \mh Z .
\]
If gcd$(k,r) = 1$, then the generic points of $T^{P_\mu}_{r,k}$
have $W(P_\mu ,P_\mu )$-stabilizer 
\[
\inp{W_{P_\mu}}{\sigma (\mu )^{1/r}} \cap W(P_\mu ,P_\mu ) \cong
\mh Z / r \mh Z .
\]
Note that for $r' | g^\vee$
\begin{equation}\label{eq:6.49}
T^{P_\mu}_{r',k} \subset T^{P_\mu}_{r,k} \qquad \mr{if}\; r | r' .
\end{equation}
If a point $t \in T^{P_\mu}_{r,k}$ does not lie on any 
$T^{P_\mu}_{r',k'}$ with $r' > r$, then its $W(P_\mu ,P_\mu 
)$-stabilizer may still be larger than $\mh Z / r \mh Z$. 
However, it is always of the form
\[
S^r_{\lambda_1} \times \cdots \times S^r_{\lambda_l} \rtimes 
\mh Z / r \mh Z .
\]
Here the product of symmetric groups is $W(R_\xi)$ from \eqref{eq:2.8},
and $\mf R_\xi  = \Z / r \Z$. With \cite{DeOp2} it follows that the 
intertwiners $\pi (w, P_\mu ,\delta ,t)$ are scalar for 
$w \in S^r_{\lambda_1} \times \cdots \times S^r_{\lambda_l}$ 
and nonscalar for $w \in (\mh Z / r \mh Z) \setminus \{ e \}$. 
Because $\mh Z / r \mh Z$ is cyclic this implies that 
$\pi (P_\mu ,\delta ,t)$ is the direct sum of exactly $r$
inequivalent irreducible representations. 

Different choices of $\sigma (\mu )^{1/r}$ or of $k \in (\mh Z / 
r \mh Z )^\times$ lead to conjugate subvarieties of $T^{P_\mu}$, 
so we have a complete description of $\Irr \big(C_r^* (\mc R ,q
)_{P_\mu}\big)$. To calculate the $K$-theory of this algebra we 
use Theorem \ref{thm:2.15}, which says that (at least modulo torsion) 
it is isomorphic to 
\[
H^*_{W(P_\mu ,P_\mu )} \big( T_u^{P_\mu};\mc L_u \big) \cong
\check H^* \big( T^{P_\mu} \big/ W(P_\mu ,P_\mu ) ;\mc L_u^{W(
P_\mu ,P_\mu )} \big) .
\]
We can endow $T_u^{P_\mu}$ with the structure of a finite 
$W(P_\mu ,P_\mu )$-CW-complex, such
that every $T_{u,r,k}^{P_\mu}$ is a subcomplex. The local 
coefficient system $\mc L_u$ is not very complicated: 
$\mc L_u (B) \cong \mh Z^r$ if and only if $B \setminus \partial B$
consists of generic points in a conjugate of $T_{u,r,k}^{P_\mu}$.
In suitable coordinates the maps $\mc L_u (B \to B')$ are all of the form 
\[
\mh Z^r \to \mh Z^{r/d} : (x_1 ,\ldots ,x_r ) \to 
(x_1 + x_2 + \cdots + x_d ,\ldots ,x_{1+r-d} + \cdots + x_r ) .
\]
Hence the associated sheaf is the direct sum of several subsheaves
$\mf F^\mu_r$, one for each divisor $r$ of gcd$(\mu^\vee )$. The
support of $\mf F^\mu_r$ is
\[
W(P_\mu ,P_\mu ) T_{u,r,1}^{P_\mu} \big/ W(P_\mu ,P_\mu ) \cong
T_u^{P_{\mu^{1/r}}} \big/ Z_{S_n} (\sigma (\mu^{1/r})) 
\]
and on that space it has constant stalk $\mh Z^{\phi (r)}$.
Here $\phi$ is the Euler $\phi$-function, i.e.
\[
\phi (r) = \# \{ m \in \mh Z : 0 \leq m < r , \mr{gcd}(m,r) = 1 \}
= \# (\mh Z / r \mh Z )^\times ,
\]
This is the rank of $\mf F^\mu_r$, because in every point of
$T_{u,r,1}$ we have $r$ irreducible representations, but the ones
corresponding to numbers that are not coprime to $r$ are already
accounted for by the sheaves $\mf F^\mu_{r'}$ with $r' | r$.
We calculate
\begin{multline}
\check H^* \big( T^{P_\mu}_\un / W(P_\mu ,P_\mu ) ;\mc L_u^{W(
P_\mu ,P_\mu )} \big) 
\cong \bigoplus_{r | \mr{gcd}(\mu^\vee )} \check H^* \big( 
T^{P_\mu}_\un / W(P_\mu ,P_\mu ) ; \mf F^\mu_r \big) \\
\cong \bigoplus_{r | \mr{gcd}(\mu^\vee )} \check H^* \big(
T_\un^{P_{\mu^{1/r}}} \big/ Z_{S_n} (\sigma (\mu^{1/r})) ; 
\mh Z^{\phi (r)} \big) 
\cong \bigoplus_{r | \mr{gcd}(\mu^\vee )} \check H^* \big( 
(S^1)^{b(\mu^{1/r}) - 1} ; \mh Z^{\phi (r)} \big) \\
\cong \bigoplus_{r | \mr{gcd}(\mu^\vee )} \mh Z^{\phi (r)
2^{b(\mu^{1/r}) - 1}} 
= \bigoplus_{r | \mr{gcd}(\mu^\vee )} \mh Z^{\phi (r) 
2^{b(\mu^\vee ) -1}} \: = \; 
\mh Z^{\mr{gcd}(\mu^\vee ) 2^{b(\mu^\vee ) -1}} .
\end{multline}
Now Theorem \ref{thm:2.15} says that $K_* \big(C_r^* (\mc R ,q)_{P_\mu}\big)$ 
is also a free abelian group of rank $\mr{gcd}(\mu^\vee) 2^{b(\mu^\vee ) -1}$. 
Summing over partitions $\mu$ of $n$ we find that $K_* \big( 
C_r^* (\mc R ,q) \big)$ is a free abelian group of rank
\[
\sum\nolimits_{\mu \vdash n} \mr{gcd}(\mu^\vee ) 2^{b(\mu^\vee ) -1} =
\sum\nolimits_{\mu \vdash n} \mr{gcd}(\mu ) 2^{b(\mu ) -1} .
\]
From Theorem \ref{thm:2.2} and the case $q=1$ we see that these K-groups
can also be obtained as the K-theory of a disjoint union of compact tori, with
gcd$(\mu)$ tori of dimension $b(\mu) - 1$. This allows us to immediately determine 
$K_0$ and $K_1$ separately as well:
\begin{equation}
\begin{aligned}
& K_0 (C_r^* (\mc R,q)) = \bigoplus\nolimits_{\mu \vdash n , b(\mu) > 1} 
\mh Z^{\mr{gcd}(\mu) 2^{b(\mu ) - 2}} \oplus 
\bigoplus\nolimits_{\mu \vdash n , b(\mu) = 1} \Z^{\mr{gcd}(\mu)} , \\
& K_1 (C_r^* (\mc R,q)) = \bigoplus\nolimits_{\mu \vdash n, b(\mu) > 1} 
\mh Z^{\mr{gcd}(\mu) 2^{b(\mu ) - 2}} .
\end{aligned}
\end{equation}

\subsection{Type $SO_{2n+1}$} \
\label{par:SO}

The root systems of type $B_n$ are more complicated than those 
of type $A_n$ because there are roots of different lengths. This 
implies that the associated root data allow label functions 
which have three independent parameters. Detailed information
about the representations of type $B_n$ affine Hecke algebras
is available from \cite{Slo2}.

Consider the root datum for the special orthogonal group $SO_{2n+1}$:
\[
\begin{aligned}
& X = Q = \mh Z^n \\
& Y = \mh Z^n \quad Q^\vee = \{ y \in Y : y_1 + \cdots + y_n 
  \; \mr{even} \} \\
& T = (\mh C^\times)^n \quad t = (t_1, \ldots, t_n) = 
  (t(e_1), \ldots, t(e_n)) \\
& R = \{ x \in X : \norm{x} = 1 \;\mr{or}\; \norm{x} = \sqrt 2 \} ,\quad \alpha_0 = e_1 \\
& R^\vee = \{ x \in X : \norm{x} = 2 \;\mr{or}\; \norm{x} = 
  \sqrt 2 \} ,\quad \alpha_0^\vee = 2 e_1 \\
& \Delta = \{ \alpha_i = e_i - e_{i+1} : i=1, \ldots, n-1 \} \cup
  \{\alpha_n = e_n\} \\
& s_i = s_{\alpha_i} \quad s_0 = t_{\alpha_0} s_{\alpha_0} : x \to
  x + \alpha_0 - \inp{\alpha_0^\vee}{x} \alpha_0  \\
& W = \langle s_1, \ldots, s_n | s_j^2 = (s_i s_j )^2 = (s_i
  s_{i+1})^3 = (s_{n-1} s_n )^4 = e : i \leq n-2, |i-j| > 1 \rangle \\
& S^\af = \{ s_0 ,s_1, \ldots ,s_{n-1} ,s_n \} \quad
  \Omega = \{ e\} \\
& W^e = W^\af = \langle W, s_0 | s_0^2 = 
  (s_0 s_i)^2 = (s_0 s_1)^4 = e : i \geq 2 \rangle 
\end{aligned}
\]
For a generic parameter function we have different parameters
$q_0 = q(s_0 ) , q_1 = q(s_i )$  for $1 \leq i < n$ and $q_2 = q(s_n )$. 

The finite reflection group $W = W (B_n )$ is naturally 
isomorphic to $(\mh Z / 2 \mh Z )^n \rtimes S_n$. Let $\mu \vdash n$
and consider a point
\begin{equation}\label{eq:6.7}
t = \big( (t_1^{\pm})^{\mu_1} \cdots (t_{n- \mu_{d-1} - \mu_d}^\pm )^{\mu_{d-2}} 
(1)^{\mu_{d-1}} (-1)^{\mu_d} \big) \in T ,
\end{equation}
where $(t_1^\pm )^{\mu_1}$ means that $\mu_1$ coordinates are equal to
$t_1$ or $t_1^{-1}$, while the other $n-\mu_1$ coordinates of $t$ are 
different. The stabilizer $W_t$ of $t$ is isomorphic to 
\begin{equation}\label{eq:6.8}
S_{\mu_1} \times \cdots \times S_{\mu_{d-2}} \times
W (B_{\mu_{d-1}}) \times W (B_{\mu_d}) . 
\end{equation}
Notice that this is a Weyl group, generated by the reflections it contains.

\begin{itemize}
\item \textbf{\large case $\mb{q_0 = q_1 = q_2 = 1}$}
\end{itemize}
In view of \eqref{eq:2.12} we want to determine the extended
quotient $\widetilde{T_\un} / W$. Therefore we recall the explicit
classification of conjugacy classes in $W$ in terms of bipartitions, which be 
found (for example) in \cite{Car}. 
We already know that the quotient of $W$ by the normal subgroup 
$(\mh Z / 2 \mh Z )^n$ of sign changes is isomorphic to $S_n$, and that 
conjugacy classes in $S_n$ are parametrized by partitions of $n$. 
So we wonder what the different conjugacy classes in 
$(\mh Z / 2 \mh Z )^n \sigma  (\mu )$ are, for $\mu \vdash n$. 
To handle this we introduce some notation, assuming that 
$|\mu | + |\lambda | = n$ and $|\mu | + |\lambda | + |\rho | = n'$:
\begin{equation}\label{eq:6.52}
\begin{array}{lll}
\epsilon_I & = & 
  \prod_{i \in I} s_{e_i} \quad I \subset \{1, \ldots, n\} \\
I_\lambda & = & 
  \{ 1, 1 + \lambda_1, 1 + \lambda_1 + \lambda_2, \cdots \} 
  \quad \lambda = (\lambda_1, \lambda_2, \lambda_3 ,\ldots ) \\
\sigma'(\lambda) & = & \epsilon_{I_\lambda} \sigma(\lambda) \in
  W (B_{|\lambda |}) \\
\sigma(\mu,\lambda) & = &
  \sigma(\mu) \: (m \to m - |\lambda| \: \mr{mod} \:n) \,\sigma'
  (\lambda)\, (m \to m + |\lambda| \: \mr{mod} \:n) \\
\sigma(\mu, \lambda, \rho) & = &
  \sigma(\mu,\lambda) \: (m \to m - |\rho | \: \mr{mod} \:n') 
  \,\sigma'(\rho )\, (m \to m + |\rho | \: \mr{mod} \:n')  
\end{array}
\end{equation}
Let $I \subset \{1, \ldots, m \}$ and $J \subset \{m+1, \ldots, 2m \}$. 
It is easily verified that $\epsilon_I (1 \, 2 \cdots m)$ is 
conjugate to $\mu_J (m+1 \, m+2 \cdots 2m)$ if and only if $|I| + |J|$ is even. 
Therefore the conjugacy classes in $W (B_n)$ are
parametrized by ordered pairs of partitions of total weight $n$.
Explicitly $(\mu ,\lambda )$ corresponds to $\sigma (\mu ,\lambda )$
as in \eqref{eq:6.52}. The set $T^{\sigma (\mu ,\lambda )}$ and the
group $Z_{W_0 (B_n )}(\sigma (\mu ,\lambda ))$ are both the direct
product of the corresponding objects for the blocks of $\mu$ and
$\lambda$, i.e. for the parts $(m,m,\ldots ,m)$. The centralizer of
$\sigma ((m)^k )$ in $W (B_{km})$ is generated by
$(1 \; 2 \cdots m) ,\, \epsilon_{\{1,2,\ldots ,m \}}$ and the 
transpositions of cycles
\begin{equation}\label{eq:6.53}
(am+1 \; am+m+1) (am+2 \; am+m+2) \cdots (am+m \; am+2m) ,
\end{equation}
where $0 \leq a \leq k-2$. It follows that
\begin{equation}\label{eq:6.54}
\begin{array}{l}
Z_{W (B_{km})} (\sigma ((m)^k )) \quad \cong \quad W (B_k ) \ltimes (\Z / m \Z )^k , \\ 
\big( (\mh C^\times)^{km} \big)^{\sigma ((m)^k)} \quad = \quad
\big\{ \big( (t_1)^m (t_{m+1})^m \cdots (t_{km+1-m})^m \big) : 
t_i \in \mh C^\times \big\} , \\
\big( (S^1)^{km} \big)^{\sigma ((m)^k)} \big/ Z_{W_0 (B_{km})} 
(\sigma ((m)^k )) \; \cong \; (S^1)^k / W(B_k) \; \cong \; [-1,1]^k \big/ S_k .
\end{array}
\end{equation}
Now consider the following element of $W (B_{km})$:
\[
\sigma' ((m)^k ) = \epsilon_{\{1,m+1,\ldots,km+1-m\}} \: (1\; 2 \cdots m) 
(m+1 \cdots 2m) \cdots (km+1-m \cdots km) .
\]
It has only $2^k$ fixpoints, namely 
\begin{equation}\label{eq:6.6}
\big( (\pm 1)^m (\pm 1)^m \cdots (\pm 1)^m \big) .
\end{equation}
The centralizer of $\sigma' ((m)^k )$ is generated by $\epsilon_{\{1\}} 
(1 \; 2 \cdots m) ,\, \epsilon_{\{1,2, \ldots, m\}}$ and the
elements \eqref{eq:6.53}. The latter two generate a subgroup isomorphic to $W(B_k)$,
which fits in a short exact sequence
\begin{equation}\label{eq:6.5}
1 \to W(B_k) \to Z_{W (B_{mk})} (\sigma' ((m)^k )) \to (\Z / m \Z)^k \to 1 ,
\end{equation}
where the first factor $\Z / m \Z$ is generated by the image of $\epsilon_{\{1\}} 
(1 \; 2 \cdots m)$. We find
\begin{equation}\label{eq:6.55}
\big( (S^1)^{km} \big)^{\sigma' ((m)^k)} \big/ Z_{W (B_{mk})} 
(\sigma' ((m)^k )) \cong \{ (1)^{am} (-1)^{(k-a)m} : 
0 \leq a \leq k \} .
\end{equation}
Now we can see what $T_\un^{\sigma (\mu ,\lambda)} \big/ Z_W 
(\sigma (\mu, \lambda))$ looks like. Its number of components 
$N(\lambda )$ depends only on $\lambda$, and all these components 
are mutually homeomorphic contractible orbifolds, the shape and
dimension being determined by $\mu$. More precisely, for every 
block of $\mu$ of width $k$ we get a factor $[-1,1]^k / S_k$, and 
for every block of $\lambda$ of width $l$ we must multiply the 
number of components by $l+1$. Alternatively, we can obtain the same
space (modulo the action of $W$) as
\begin{equation}\label{eq:6.56}
\begin{split}
T_\un^{\sigma(\mu,\lambda)} & \big/ Z_{W(B_n)} (\sigma(\mu,\lambda))\; 
 =\; \bigsqcup\nolimits_{\lambda_1 \cup \lambda_2 = \lambda} 
  T_{\un,c}^{\sigma(\mu,\lambda_1,\lambda_2)} \big/ Z_{W (B_n)} 
  (\sigma(\mu,\lambda_1,\lambda_2))  \\
 & =\; \bigsqcup\nolimits_{\lambda_1 \cup \lambda_2 = \lambda} \big( 
  (S^1)^{|\mu|} \big)^{\sigma(\mu)} \!\!\big/ Z_{W (B_{|\mu|})}
  (\sigma (\mu)) \: (-1)^{|\lambda_1|} \: (1)^{|\lambda_2|} \\
 & =\;\big( [-1,1]^{|\mu|} \big)^{\sigma(\mu)} \!\!\big/ Z_{S_{|\mu|}}
  (\sigma (\mu)) \times \bigsqcup\nolimits_{\lambda_1 \cup \lambda_2 = \lambda}  
  (-1)^{|\lambda_1|} \: (1)^{|\lambda_2|} ,
\end{split}
\end{equation}
where the subscript $c$ means that we take only
the connected component containing the point
$\big( (1)^{|\mu|} (-1)^{|\lambda_1|} (1)^{|\lambda_2|} \big)$.

In effect we parametrized the components of the extended quotient
$\widetilde{T_\un} / W$ by ordered triples of partitions 
$(\mu,\lambda_1,\lambda_2)$ of total weight $n$, and every such
component is contractible. In combination with \eqref{eq:6.8} this
shows that the conditions of Theorem \ref{thm:2.3} are fulfilled.

Denote the number of ordered $k$-tuples of partitions of total weight 
$n$ by $\mc P (k,n)$. Now Theorem \ref{thm:2.3} says that
\begin{equation}\label{eq:6.14}
K_* \big( C_r^* (W^e) \big) =
\check H^* \big( \widetilde{T_\un} / W ;\Z \big) =
\check H^0 \big( \widetilde{T_\un} / W ;\Z \big) \cong \Z^{\mc P (3,n)} .
\end{equation}

\begin{itemize}
\item \textbf{\large generic case}
\end{itemize}
The inequivalent subsets of $\Delta$ are parametrized by partitions
$\mu$ of weight at most $n$.
\[
\begin{aligned}
& P_\mu = \Delta \setminus \{\alpha_{\mu_1}, \alpha_{\mu_1 +
 \mu_2}, \ldots, \alpha_{|\mu|} \} \\
& R_{P_\mu} \cong (A_{n-1})^{m_n} \times \cdots \times (A_1)^{m_2} 
\times B_{n - |\mu|} \\
& R_{P_\mu}^\vee \cong (A_{n-1})^{m_n} \times \cdots \times 
(A_1)^{m_2} \times C_{n - |\mu|} \\
& X^{P_\mu} \cong \mh Z (e_1 + \cdots + e_{\mu_1}) /\mu_1 + \cdots 
  + \mh Z (e_{|\mu|+1- \mu_d} + \cdots + e_{|\mu|}) /\mu_d \\
& X_{P_\mu} \cong ( \mh Z^n / \mh Z (e_1 + \cdots + e_n ) )^{m_n}
  + \cdots + (\mh Z^2 / \mh Z (e_1 + e_2 ) )^{m_2} + \mh Z^{n - |\mu |} \\
& Y^{P_\mu} = \mh Z(e_1 + \cdots + e_{\mu_1}) + \cdots + \mh Z 
  (e_{|\mu|+1-\mu_d} + \cdots +e_{|\mu|}) \\
& Y_{P_\mu} = \{ y \in \mh Z^n : y_1 + \cdots + y_{\mu_1} = \cdots 
  = y_{|\mu|+1-\mu_d} + \cdots + y_{|\mu|} = 0 \} \\
& T^{P_\mu} = \{ (t_1 )^{\mu_1} (t_{\mu +1})^{\mu_2} \cdots
  (t_{|\mu |})^{\mu_d} (1)^{n - |\mu |} : t_i \in \mh C^\times \} \\
& T_{P_\mu} = \{ t \in (\mh C^\times )^n : t_1 \cdots t_{\mu_1} =
  t_{\mu_1 + 1} \cdots t_{\mu_1 + \mu_2} = \cdots = t_{|\mu | + 1 - 
  \mu_d} \cdots t_{|\mu |} = 1 \} \\
& K_{P_\mu} = \{ t \in T^{P_\mu} : t_1^{\mu_1} = \cdots = 
  t_{|\mu|}^{\mu_d} = 1 \} \\
& W_{P_\mu} \cong S_n^{m_n} \times \cdots \times S_2^{m_2} \times
  W (B_{n - |\mu |}) \\
& W(P_\mu ,P_\mu ) \cong W (B_{m_n}) \times \cdots \times W (B_{m_2}) \times
  W (B_{m_1}) 
\end{aligned}
\]
We see that $\mc R_{P_\mu}$ is the product of various root data
of type $SL_m$ and one factor $\mc R (SO_{2(n - |\mu |) + 1})$. 
Hence $\mc H_{P_\mu}$ is the tensor product of a type
$A$ part and a type $B$ part. From our study of $\mc R (SL_m)$ 
we recall that the discrete series representations of the type 
$A$ part of $\mc H_{P_\mu}$ are in bijection with $K_{P_\mu}$. 
From \cite[Proposition 4.3]{HeOp} and \cite[Appendix A.2]{Opd-Sp}
we know that the
residual points for $\mc R (SO_{2(n - |\mu |)+1},q)$ are parametrized 
by ordered pairs $(\lambda_1 ,\lambda_2 )$ of total weight $n - |\mu |$.
The unitary part of such a residual point is in the component we
indicated in \eqref{eq:6.56}. Let $RP (\mc R ,q)$ denote the 
collection of residual points for the pair $(\mc R ,q)$. The above
gives canonical bijections
\begin{multline}\label{eq:6.57}
\bigsqcup_{t \in RP (\mc R_{P_\mu}, q_{P_\mu})} t T_\un^{P_\mu} \big/ 
\mc W_{P_\mu P_\mu} \cong \; \bigsqcup_{t \in RP (\mc R 
(SO_{2(n - |\mu |)+1}, q))} t T_\un^{P_\mu} \big/ W (P_\mu ,P_\mu ) \\
\cong \; T_\un^{P_\mu} \big/ Z_{W_0 (B_{|\mu |})} (\sigma (\mu )) 
  \times \bigsqcup_{(\lambda_1 ,\lambda_2) : |\lambda_1 | + 
  |\lambda_2 | = n} (-1)^{|\lambda_1 |} (1)^{|\lambda_2 |} .
\end{multline}

\begin{thm}\label{thm:6.2}
\enuma{
\item For generic $q ,\; C_r^* (\mc R(SO_{2n+1}),q)$ is Morita equivalent with the 
commutative $C^*$-algebra of continuous functions on \eqref{eq:6.57}.
\item $K_1 \big( C_r^* (\mc R (SO_{2n+1}),q) \big) = 0$ and $K_0 \big( C_r^* 
(\mc R (SO_{2n+1},q)) \big)$ is a free abelian group of rank $\mc P (3,n)$.
}
\end{thm}
\begin{proof}
(a) First we note that \eqref{eq:6.57} can be identified with the extended quotient 
$\widetilde{T_\un} / W$ described in \eqref{eq:6.56} and the subsequent lines.

Fix any $u \in T_\un$. The fibre over $u$ of the projection
\[
p : \widetilde{T_\un} / W \to T_\un / W 
\]
is in bijection with the set of conjugacy classes of $W$. By Clifford theory
$|p^{-1} (Wu)|$ is also the number of inequivalent irreducible representations of
$C(T_\un) \rtimes W$ with central character $Wu$. Equivalently, $|p^{-1}(Wu)|$
is the number of inequivalent tempered irreducible representations of $\mc O (T) \rtimes W$
with central character $Wu$. By Theorem \ref{thm:1.5} the latter equals the number
of inequivalent irreducible tempered $\mc H (\mc R,q)$-representations with
central character in $W u T_\rs$.

By Theorem \ref{thm:2.1} every point of \eqref{eq:6.57} is the 
$Z(C_r^* (\mc R,q))$-character of at least one irreducible $C_r^* (\mc R,q)$-representation.
The projection $p'$ from \eqref{eq:6.57} to $T / W$ corresponds to restriction
from $Z(C_r^* (\mc R,q)) \cong C(\Xi_\un / \mc G)$ to $Z(\mc H (\mc R,q)) \cong \mc O (T/W)$.

Suppose that a point of $p'^{-1} (W u T_\rs)$ would carry more than one inequivalent
irreducible $C_r^* (\mc R,q)$-representation. Then the inverse image of $W u T_\rs$ under
\[
\Irr (C_r^* (\mc R,q)) = \Irr_{\mr{temp}}(\mc H (\mc R,q)) \to T / W 
\]
would have more than $|p^{-1}(u)|$ elements. This would contradict what we 
concluded above, using Theorem \ref{thm:1.5}. Thus every $\pi (P_\mu,\delta,t)$ with
$(P_\mu,\delta,t) \in \Xi_\un$ is irreducible and \eqref{eq:6.57} is exactly the space
$\Irr (C_r^* (\mc R,q))$.

When we compare this with Theorem \ref{thm:2.1} and \eqref{eq:2.1}, we see that all 
intertwining operators $\pi (g,P_\mu,\delta,t)$ with $g (P_\mu,\delta,t) = (P_\mu,\delta,t)$
must be scalar. 
Recall from \eqref{eq:2.4} that every indecomposable direct summand of
$C_r^* (\mc R,q)$ is of the form 
\begin{equation}\label{eq:6.2}
C \big( T^{P_\mu}_\un ; \End_\C (\pi (P_\mu,\delta,t) ) \big)^{\mc G_{P\mu,\delta}}.
\end{equation}
From \eqref{eq:6.55} we know that the space $T^{P_\mu}_\un / \mc G_{P\mu,\delta}$ is a
direct product of factors \\ 
$(S^1)^k / W(B_k) \cong [-1,1] / S_k$. We note that
\[
\{ (z_1,z_2,\ldots,z_k ) \in (S^1)^k : \Im (z_i) \geq 0, \Re (z_1) \geq \Re(z_2)
\geq \cdots \geq \Re (z_k) \}
\]
is a closed, connected fundamental domain for action of $W(B_k)$ on $(S^1)^k$. With this
it is easy to find a closed fundamental domain $D_{P_\mu,\delta}$ for the action of 
$\mc G_{P_\mu,\delta}$ on $T^{P_\mu}_\delta$, such that $D_{P_\mu,\delta}$ is homeomorphic
to $T^{P_\mu}_\delta / \mc G_{P_\mu,\delta}$. Then restriction from $T^{P_\mu}_\un$ to
$D_{P_\mu,\delta}$ gives a monomorphism of $C^*$-algebras from \eqref{eq:6.2} to
\[
C \big( D_{P_\mu,\delta} ; \End_\C (\pi (P_\mu,\delta,t)) \big) =
C ( D_{P_\mu,\delta} ) \otimes \End_\C (\pi (P_\mu,\delta,t)) . 
\]
It is surjective because the intertwining operators $\pi (g,P_\mu,\delta,t), \; 
g \in \mc G_{P_\mu,\delta}$, from \eqref{eq:2.9}, depend continuously on 
$t \in T^{P_\mu}_\un$ and are scalar multiples of the identity whenever they map a 
representation to itself. Hence $C_r^* (\mc R,q)$ is Morita equivalent with
$\bigoplus_{(P_\mu,\delta) / \mc G} C(D_{P_\mu,\delta})$, as required.

(b) By the Serre--Swan Theorem $K_* (C_r^* (\mc R,q))$ is the topological K-theory
of the underlying space \eqref{eq:6.57}. Since every connected component of this
space is contractible, $K_1 (C_r^* (\mc R,q)) = 0$ and $K_0 (C_r^* (\mc R,q))$ is
a free abelian group whose rank equals the number of connected components of 
\eqref{eq:6.57}. In the lines following \eqref{eq:6.56} we showed that that number
is $\mc P (3,n)$. By Theorem \ref{thm:2.2} these K-groups are independent of
the parameters $q$.
\end{proof}

\subsection{Type $Sp_{2n}$} \

The root datum for the symplectic group $Sp_{2n}$ is dual to that for 
$SO_{2n+1}$. Concretely, $\mc R (Sp_{2n})$ is given by: 
\[
\begin{aligned}
& X = \{ y \in Y : y_1 + \cdots + y_n \; \mr{even} \} ,\quad Q = \Z^n \\
& Y = Q^\vee = \mh Z^n \\
& T = (\mh C^\times)^n \quad t = (t_1, \ldots, t_n) = 
  (t(e_1), \ldots, t(e_n)) \\
& R = \{ x \in X : \norm{x} = 2 \;\mr{or}\; \norm{x} = \sqrt 2 \} ,\quad \alpha_0 = e_1 + e_2 \\
& R^\vee = \{ x \in X : \norm{x} = 1 \;\mr{or}\; \norm{x} = 
  \sqrt 2 \} ,\quad \alpha_0^\vee = e_1 + e_2 \\
\end{aligned}
\]
\[
\begin{aligned}  
& \Delta = \{ \alpha_i = e_i - e_{i+1} : i=1, \ldots, n-1 \} \cup
  \{\alpha_n = 2 e_n\} \\
& s_i = s_{\alpha_i} \quad s_0 = t_{\alpha_0} s_{\alpha_0} = t_{e_1} s_{\alpha_0} t_{-e_1} : 
  x \to x + \inp{\alpha_0^\vee}{x} \alpha_0  \\
& W = \langle s_1, \ldots, s_n | s_j^2 = (s_i s_j )^2 = (s_i
  s_{i+1})^3 = (s_{n-1} s_n )^4 = e : i \leq n-2, |i-j| > 1 \rangle \\
& S^\af = \{ s_0 ,s_1, \ldots ,s_{n-1} ,s_n \} \quad
  \Omega = \{ e, t_{e_1} s_{2 e_1} \} \\
& W^\af = \langle W, s_0 | s_0^2 = 
  (s_0 s_i)^2 = (s_0 s_2)^3 = e : i \neq 2 \rangle  \quad W^e = W^\af \rtimes \Omega
\end{aligned}
\]
For a generic parameter function we have two independent parameters $q_1 = q(s_1)$
and $q_2 = q(s_n)$.

The groups $X ,W$ and $W^e$ are exactly the same as for $\mc R(SO_{2n+1})$. Everything
that we said in Paragraph \ref{par:SO} about the stabilizers in $W$ of points of $T$
obviously is valid here as well. In particular, for $q=1$ the algebra $\mc H (\mc R (Sp_{2n}),1)$
is identical to $\mc H (\mc R (SO_{2n+1}),1)$, and the entire analysis of the K-theory 
of its $C^*$-completion can be found in the previous paragraph.

For all other $q$ we can use Theorem \ref{thm:2.2}. Thus we get
\begin{multline*}
K_* \big( C_r^* (\mc R(Sp_{2n}),q) \big) \cong K_* \big( C_r^* (\mc R(Sp_{2n}),1) \big) \\
= K_* \big( C_r^* (\mc R(SO_{2n+1}),1) \big) \cong K_* \big( C_r^* (\mc R(SO_{2n+1}),q) \big) .  
\end{multline*}
The last group is the one we actually computed, for generic parameters.
Let us phrase the results explicitly:
\begin{equation}
K_0 \big( C_r^* (\mc R(Sp_{2n}),q) \big) \cong \Z^{\mc P (3,n)} ,\qquad
K_1 \big( C_r^* (\mc R(Sp_{2n}),q) \big) = 0 .
\end{equation}

\subsection{Type $SO_{2n}$} \
\label{par:SOeven}

The root datum for the even special orthogonal group $SO_{2n}$  
has groups contained in those for the root datum of type $SO_{2n+1}$.
\[
\begin{aligned}
& X = \mh Z^n \quad Q = \{ y \in Y : y_1 + \cdots + y_n 
  \; \mr{even} \} \\ 
& Y = \mh Z^n \quad Q^\vee = \{ y \in Y : y_1 + \cdots + y_n 
  \; \mr{even} \} \\
& T = (\mh C^\times)^n \quad t = (t_1, \ldots, t_n) = 
  (t(e_1), \ldots, t(e_n)) \\
& R = \{ x \in X : \norm{x} = \sqrt 2 \} ,\quad \alpha_0 = e_1 + e_2 \\
& R^\vee = \{ x \in X : \norm{x} = 
  \sqrt 2 \} ,\quad \alpha_0^\vee = e_1 + e_2 \\
& \Delta = \{ \alpha_i = e_i - e_{i+1} : i=1, \ldots, n-1 \} \cup
  \{\alpha_n = e_{n-1} + e_n\} \\
& s_i = s_{\alpha_i} \quad s_0 = t_{\alpha_0} s_{\alpha_0} = t_{e_1} s_{\alpha_0} t_{-e_1} 
: x \to  x + \alpha_0 - \inp{\alpha_0^\vee}{x} \alpha_0  \\
& W = \langle s_1, \ldots, s_n | s_j^2 = (s_i s_j )^2 = (s_i
  s_{i+1})^3 = (s_{n-2} s_n )^4 = e : i \leq n-2, |i-j| > 1 \rangle \\
& S^\af = \{ s_0 ,s_1, \ldots ,s_{n-1} ,s_n \} \quad
  \Omega = \{ e, t_{e_1} s_{e_1} s_{e_n} \} \\
& W^\af = \langle W, s_0 | s_0^2 = 
  (s_0 s_i)^2 = (s_0 s_2)^3 = e : i \neq 2 \rangle \subsetneq W^e
\end{aligned}
\]
When $n > 2$ all the simple affine reflections are conjugate in $W^e$, and
\[
q(s_i) = q \quad i = 0,1,\ldots,n 
\]
for every parameter function. For $n=2$ the root system $R \cong A_1 \times A_1$ is
reducible, there is an additional simple affine reflection and there are more possible 
parameter functions. For $n=1 ,\; \mc R (SO_2)$ is the root datum of a one-dimensional
torus, in particular $W = 1$.

The based root datum $\mc R (SO_{2n})$ has one nontrivial automorphism, which exchanges
the roots $\alpha_{n-1}$ and $\alpha_n$. It is easily seen that
\[
W^e (SO_{2n}) \rtimes \mr{Aut}(\mc R (SO_{2n})) \cong W^e (SO_{2n+1}) . 
\]
With Theorem \ref{thm:2.2} we conclude that, for every equal parameter function $q$,
\begin{equation}
\begin{split}
K_* \big( C_r^* (\mc R (SO_{2n}),q) \rtimes \mr{Aut}(\mc R (SO_{2n})) \big) \cong
K_* \big( W^e (SO_{2n}) \rtimes \mr{Aut}(\mc R (SO_{2n})) \big) \\
= K_* \big( C_r^* (W^e (SO_{2n+1})) \big) \cong
K_* \big( C_r^* (\mc R (SO_{2n+1}), q) \big) .
\end{split}
\end{equation}
Unfortunately no such shortcut is available for $K_* \big( C_r^* (\mc R (SO_{2n}),q) \big)$.
Therefore we will just compute $K_* \big( W^e (SO_{2n}) \big)$ by hand, in several steps:
\begin{itemize}
\item We determine the extended quotient $T_\un /\!/ W(D_n)$ and its cohomology.
\item We analyse the (elliptic) representations of the $W(D_n)$-isotropy groups of
points of $T$.
\item We relate the second bullet to the sheaf $\mf{L}_u^{W(D_n)}$ on $T_\un / W(D_n)$.
\item Then we are finally in the right position to apply Theorem \ref{thm:2.15}.
\end{itemize}
The finite reflection group $W(D_n)$ is naturally isomorphic to the index two subgroup
of $W(B_n) = W(C_n)$ consisting of those elements that involve an even number of
sign changes. In other words, let $(\Z / 2 \Z)^n_{ev}$ be the kernel of the summation
map $(\Z / 2\Z)^n \to \Z / 2 \Z$, then 
\[
W(D_n) = (\Z / 2 \Z)^n_{ev} \rtimes S_n .
\]
The conjugacy classes in $W(D_n)$ are similar to, but slightly different from those 
in $W(B_n)$. We rephrase Young's parametrization in the notations from \eqref{eq:6.52}.
For every bipartition $(\mu,\lambda)$ of $n$ where $\lambda$ has an even number
of parts, $\sigma (\mu,\lambda)$ represents one class in $W(D_n)$. 
Suppose now that $\mu \vdash n$ has only even terms, and define
\begin{equation}\label{eq:6.20}
\sigma'' (\mu) = \sigma (\mu) \epsilon_{\{n-1,n\}} = 
\epsilon_{\{n\}}^{-1} \sigma (\mu) \epsilon_{\{n\}} .
\end{equation}
Then $\sigma'' (\mu)$ represents a class of $W(D_n)$ different from the above. The 
$\sigma (\mu,\lambda)$ and the $\sigma'' (\mu)$ form a set of representatives for
all conjugacy classes of $W(D_n)$.

In the representation theory of classical groups some almost direct products of root
data of type $D$ arise \cite{Gol,Hei}. Therefore it will be useful to investigate a more 
general situation, as in Paragraph \ref{par:almost}. Fix $n_1, \ldots, n_d$ with 
$n_1 + \cdots + n_d = n$ and consider the root datum
\[
\mc R'_{\vec n} = \mc R (SO_{2 n_1}) \times \cdots \times \mc R (SO_{2 n_d}) .
\]
Let $W'_{\vec n} = W(D_{\vec n}) \rtimes \Gamma$ be as in \eqref{eq:1.43}, so
$\Gamma \cong (\Z / 2 \Z)^d_{ev}$. The conjugacy classes for $W'_{\vec n}$ are a mixture 
of those for $W(D_n)$ and for $W(B_{\vec n})$. Let us analyse them and the extended
quotient $T_\un /\!/ W'_{\vec n}$ together. 

Recall that for $w \in W(B_{\vec n})$ the groups $T_\un^w$ and $Z_{W(B_{\vec n})}(w)$ 
were already computed in Paragraph \ref{par:SO}, see in particular \eqref{eq:6.54}, 
\eqref{eq:6.6} and \eqref{eq:6.5}. We say that $\vec \mu \vdash \vec n$ if $\vec \mu$ is a
$d$-tuple of partitions $(\mu^{(1)}, \ldots, \mu^{(d)})$ with $|\mu^{(i)}| = n_i$, and 
that $(\vec \mu, \vec \lambda) \vdash \vec n$ if $\vec \lambda = (\lambda^{(1)}, \ldots,
\lambda^{(n)})$ such that $|\mu^{(i)}| + |\lambda^{(i)}|$. To these we can associate
$\sigma (\vec \mu)$ and $\sigma (\vec \mu, \vec \lambda)$, as products of \eqref{eq:6.52}
over the indices $i$.
\begin{itemize}
\item Consider $\sigma (\vec \mu,\vec \lambda)$, where $\vec \lambda$ is nonempty and 
has an even number of terms. Notice that $Z_{W(B_{\vec n})}(\sigma (\vec \mu,\vec \lambda))$ 
contains an element not in $W(D_n)$ which fixes $T^{\sigma (\vec \mu,\vec \lambda)}$ 
pointwise, namely a single factor $\epsilon_{\{a_1\}} (a_1 \cdots a_m)$ of $\vec \lambda$. 
Hence the $W(B_{\vec n})$-conjugacy class of $\sigma (\vec \mu, \vec \lambda)$ is precisely
the $W'_{\vec n}$-conjugacy class of $\sigma (\vec \mu, \vec \lambda)$. Furthermore
\[
T_\un^{\sigma (\vec \mu,\vec \lambda)} / Z_{W'_{\vec n}}(\sigma (\vec \mu,\vec \lambda)) = 
T_\un^{\sigma (\vec \mu,\vec \lambda)} / Z_{W(B_{\vec n})}(\sigma (\vec \mu,\vec \lambda)) ,
\]
and as described in \eqref{eq:6.56}, this is a disjoint union of contractible spaces.
The number of components is given explicitly in terms of $\vec \lambda$.
\item Suppose that $\vec \mu \vdash \vec n$ and that all terms of $\vec \mu$ are even. 
Then the $W(B_{\vec n})$-conjugacy class of $\sigma (\vec \mu)$ splits into two 
$W'_{\vec n}$-conjugacy classes, the other one represented by
\[
\sigma'' (\vec \mu) = \sigma (\vec \mu) \epsilon_{\{n-1,n\}} .
\]
Both $Z_{W(B_{\vec n})}(\sigma (\vec \mu))$ and 
\[ 
Z_{W(B_{\vec n})}(\sigma'' (\vec \mu)) = 
\epsilon_{\{ n \}}^{-1} Z_{W(B_{\vec n})}(\sigma (\vec \mu)) \epsilon_{\{n\}}
\]
are contained in $W'_{\vec n}$. Let $m_l$ be the multiplicity of $l$ in $\vec \mu$.
By \eqref{eq:6.56}
\[
T_\un^{\sigma'' (\vec \mu)} / Z_{W'_{\vec n}}(\sigma'' (\vec \mu)) \cong
T_\un^{\sigma (\vec \mu)} / Z_{W'_{\vec n}}(\sigma (\vec \mu))
\cong \prod\nolimits_{l=1}^n [-1,1]^{m_l} / S_{m_l} ,
\]
which is a contractible space.
\item Let $\mu \vdash n$ be a partition with at least one odd term. Again the 
$W(B_{\vec n})$-conjugacy class of $\sigma (\vec \mu)$ is precisely
the $W'_{\vec n}$-conjugacy class of $\sigma (\vec \mu)$. Now
\[
Z_{W'_{\vec n}}(\sigma (\vec \mu)) \subsetneq Z_{W(B_{\vec n})}(\sigma (\vec \mu)) 
\]
and this really makes a difference. From \eqref{eq:6.54} we deduce 
\begin{equation}\label{eq:6.9}
T_\un^{\sigma (\vec \mu)} / Z_{W'_{\vec n}}(\sigma (\vec \mu)) \cong 
\prod\nolimits_{l=1}^n (S^1 )^{m_l} 
\Big/ \Big( \prod\nolimits_{l=1}^n W(B_{m_l}) \cap W(D_n) \Big) .
\end{equation}
The group $\prod_{l=1}^n W(B_{m_l}) \cap W(D_n)$ equals 
$\big( \prod_{l=1}^n (\Z / 2 \Z)^{m_l} \big)_+ \rtimes \prod_{l=1}^{m_l} S_{m_l}$,
where the subscript + means that the total number of sign changes for odd $l$
must be even. The quotient map
\begin{equation}\label{eq:6.10}
\prod_{l \text{ odd}} (S^1 )^{m_l} \big/ \big( \prod_{l \text{ odd}} 
(\Z / 2 \Z)^{m_l} \big)_+ \longrightarrow \prod_{l \text{ odd}} (S^1 )^{m_l} \big/ 
(\Z / 2 \Z)^{m_l} \cong \prod_{l \text{ odd}} [-1,1]^{m_l} 
\end{equation}
is a two-fold cover which ramifies precisely at the boundary of the unit cube
$\prod_{l \text{ odd}} [-1,1]^{m_l}$. Therefore the left hand side of \eqref{eq:6.10}
is homeomorphic to the unit sphere of dimension $m_1 + m_3 + m_5 + \cdots$
This entails that \eqref{eq:6.9} is homeomorphic to 
\begin{equation}
\prod_{l \text{ even}} \big( [-1,1]^{m_l} / S_{m_l} \big)  \times S^{m_1 + m_3 + \cdots} \big/
\prod_{l \text{ odd}} S_{m_l} .
\end{equation}
This space is contractible unless $m_l = 1$ for all odd $l$, then it is homotopic to
$S^{m_1 + m_3 + \cdots}$.
\end{itemize}

The extended quotient $T_\un /\!/ W'_{\vec n}$ is the disjoint union of the spaces 
$T^w_\un / Z_{W'_{\vec n}}(w)$, as $w$ runs over representatives for the conjugacy classes 
of $W'_{\vec n}$. Since we covered all conjugacy classes for $W(B_{\vec n})$ intersecting
$W'_{\vec n}$, we have a complete description of conjugacy classes for the latter group.
From the above calculations we immediately get the cohomology of the extended quotient.

\begin{lem}\label{lem:6.4}
The abelian group $\check H^* (T_\un /\!/ W'_{\vec n})$ is torsion-free. 

In the case $\vec n = n , W'_{\vec n} = W(D_n)$, we can describe the cohomology of
$T_\un /\!/ W(D_n)$ explicitly. The rank of the odd cohomology is the number of partitions 
$\mu \vdash n$ such that every odd term appears with multiplicity one, and there is 
an odd number of odd terms. 

The rank of the even cohomology of $T_\un /\!/ W(D_n)$ is the sum of four contributions:
\begin{itemize}
\item $\prod_i (k_i + 1)$, for every bipartition $(\mu,\lambda)$ of $n$ with 
$\lambda = (n)^{k_n} \cdots (1)^{k_1}$, such that $\sum_i k_i$ is positive and even;
\item two times the number of partitions of $n$ with only even terms;
\item the number of partitions of $n$ with at least one odd term;
\item the number of partitions of $n$ such that every odd term appears only once,
and the number of odd terms is positive and even.
\end{itemize}
\end{lem}

Every point of $T \cong (\C^\times )^n$ is $W(B_{\vec n})$-conjugate to one of the form
\[
t = (t^{(1)},\ldots,t^{(d)}) ,\quad
t^{(i)} = \big( (t_1)^{\mu^{(i)}_1} \cdots (t_{n_i- m^{(i)}_1 - m^{(i)}_2} )^{\mu_{d_i}} 
(1)^{m^{(i)}_1} (-1)^{m^{(i)}_2} \big) \in (\C^\times )^{n_i} .
\]
The isotropy group of $t$ in $W'_{\vec n}$ is 
\begin{multline}\label{eq:6.12}
(W'_{\vec n} )_t = \Big( \prod\nolimits_{i=1}^d S_{\mu_1^{(i)}} \times \cdots \times 
S_{\mu_{d_i}^{(i)}} \times W (B_{m^{(i)}_1}) \times W (B_{m^{(i)}_2}) \Big) \cap W(D_n) = \\
\Big( \prod\nolimits_{i=1}^d S_{\mu_1^{(i)}} \times \cdots \times 
S_{\mu_{d_i}^{(i)}} \Big) \times \Big( \prod\nolimits_{i=1}^d W (B_{m^{(i)}_1}) \times 
W (B_{m^{(i)}_2}) \Big) \cap W (D_{m^{(1)}_1 + \cdots + m^{(d)}_2}) . 
\end{multline} 
We note that $(W'_{\vec n})_t$ is generated by the reflections it contains if
$t$ has no coordinates 1 or $-1$. Otherwise the reflection subgroup of $W(D_n)_t$ is
\[
(W'_{\vec n})_t^\circ := \prod\nolimits_{i=1}^d S_{\mu_1^{(i)}} \times \cdots \times 
S_{\mu_{d_i}^{(i)}} \times W (D_{m^{(i)}_1}) \times W (D_{m^{(i)}_2}) ,
\]
where $W(D_0) = W(D_1) = 1$. In that case
\begin{equation}\label{eq:6.13}
(W'_{\vec n})_t = \Big( \prod\nolimits_{i=1}^d S_{\mu_1^{(i)}} \times \cdots \times 
S_{\mu_{d_i}^{(i)}} \Big) \times W'_{\vec m} ,
\end{equation}
where $\vec m$ consists of those terms $m^{(i)}_1, m^{(i)}_2$ which are nonzero.
The group $W'_{\vec m}$ is a particular instance of the almost Weyl groups studied in 
Appendix \ref{par:almost}. Thus $(W'_{\vec n})_t$ is an example of the groups 
considered in Lemma \ref{lem:1.11}, and we may use that result.

\begin{prop}\label{prop:6.6}
For any positive parameter function $q$, $K_* (C_r^* (\mc R'_{\vec n},q))$ is a
free abelian group, isomorphic to $H^* (T_\un /\!/ W'_{\vec n} ;\Z)$.

In particular for $\vec n = n, \mc R'_{\vec n} = \mc R (SO_{2n}), W'_{\vec n} =
W (D_n)$, the free abelian group 
\[
K_* \big( C_r^* (\mc R (SO_{2n}),q) \big) \cong H^* (T_\un /\!/ W(D_n);\Z)
\]
has even and odd ranks as given in Lemma \ref{lem:6.4}.
\end{prop}
\begin{proof}
By Theorem \ref{thm:1.5} it suffices to prove this when $q=1$. 

We adapt the notations from \eqref{eq:6.56} to the present setting. Let 
$(\vec \mu,\vec \lambda_1,\vec \lambda_2)$ be a $d$-tuple of tripartitions,
of respectively $n_1,\ldots, n_d$, and such that $\vec \lambda_1 \cup \vec \lambda_2$ 
has an even number of terms. As in \eqref{eq:6.13} we write
\[
\begin{split}
W_{\vec \mu,\vec \lambda_1,\vec \lambda_2} := \Big( \prod\nolimits_{i=1}^d S_{\mu_1^{(i)}} 
\times \cdots \times S_{\mu_{d_i}^{(i)}} \times
W (B_{|\lambda^{(i)}_1|}) \times W (B_{|\lambda^{(i)}_2|}) \big) \cap W(D_n) = \\
\Big( \prod\nolimits_{i=1}^d S_{\mu_1^{(i)}} \times \cdots \times S_{\mu_{d_i}^{(i)}} \Big)
\times W'_{\vec m} ,
\end{split} 
\]
where $\vec m$ consists of the nonzero terms among the $|\lambda^{(i)}_1|, |\lambda^{(i)}_2|$.
The group $W_{\vec \mu,\vec \lambda_1,\vec \lambda_2}$ is the full stabilizer of some point
of $T_\un$, and of the form considered in Lemma \ref{lem:1.11}. 
We note that $\sigma (\vec \mu,\vec \lambda_1,\vec \lambda_2)$ is an elliptic element of
$W_{\vec \mu,\vec \lambda_1,\vec \lambda_2}$.

For every $t \in T_{\un,c}^{\sigma (\mu,\lambda_1,\lambda_2)}$ we have $(W'_{\vec n} )_t 
\supset W_{\vec \mu,\vec \lambda_1,\vec \lambda_2}$. Using Lemma \ref{lem:1.11} we define
\begin{equation}\label{eq:6.18}
s( \sigma (\vec \mu,\vec \lambda_1,\vec \lambda_2), t) = \ind^{(W'_{\vec n})_t}_{
W_{\vec \mu,\vec \lambda_1,\vec \lambda_2}} H(u_{\sigma (\vec \mu,\vec \lambda_1,\vec \lambda_2)},
\rho_{\sigma (\vec \mu,\vec \lambda_1,\vec \lambda_2)}) .
\end{equation}
Suppose that $\vec \mu \vdash \vec n$ and that $\vec \mu$ has only even terms. 
Then $\sigma'' (\vec \mu) = \epsilon_{\{n-1,n\}} \sigma (\vec \mu)$ is conjugate to 
$\sigma (\vec \mu)$ in $W(B_{\vec n})$ but not in $W'_{\vec n}$. The element 
$\sigma'' (\vec \mu)$ is elliptic in $\epsilon_{\{n\}} \big( \prod_{i=1}^d S_{\mu^{(i)}_1} 
\times \cdots \times S_{\mu^{(i)}_{d_i}} \big) \epsilon_{\{n\}}$
and for every $t \in T^{\sigma'' (\vec \mu)}$ we have
\[
(W'_{\vec n})_t \supset \epsilon_{\{n\}} \big( \prod\nolimits_{i=1}^d S_{\mu^{(i)}_1} 
\times \cdots \times S_{\mu^{(i)}_{d_i}} \big) \epsilon_{\{n\}} .
\]
For such $t$ we define
\begin{equation}\label{eq:6.19}
s (\sigma'' (\vec \mu),t) = \ind^{(W'_{\vec n} )_t}_{\epsilon_{\{n\}} \big( 
\prod\nolimits_{i=1}^d S_{\mu^{(i)}_1} \times \cdots \times S_{\mu^{(i)}_{d_i}} \big) 
\epsilon_{\{n\}}} H(u_{\sigma'' (\vec \mu)}, \rho_{\sigma'' (\vec \mu)}) .
\end{equation}
As discussed before Lemma \ref{lem:6.4}, every conjugacy class of $W'_{\vec n}$ 
appears precisely once in \eqref{eq:6.18} and \eqref{eq:6.19} together. 

With this information and Lemma \ref{lem:1.11} available, the same argument as in the proof of
Theorem \ref{thm:2.3}.a works in the present setting, and shows that the conclusion of
Theorem \ref{thm:2.3}.a is fulfilled. Then we apply Theorem \ref{thm:2.3}.b.
\end{proof}

\subsection{Type $G_2$} \

As basis for the root lattice $X$ of type $G_2$ we will take the two simple roots.
We will coordinatize the dual lattice $Y$ so that the pairing between $X$ and $Y$
becomes the standard pairing on $\Z^2$. Explicitly, $\mc R (G_2 )$ becomes:
\[
\begin{aligned}
& X = Q = \Z^2, \qquad Y = Q^\vee = \Z^2 \\
& T = (\mh C^\times )^2 \quad  t = (t(e_1),t(e_2)) = (t_1,t_2) \\
& R^+ = \{ e_1 ,e_2 , e_1 + e_2, 2 e_1 + e_2, 3 e_1 + e_2, 3 e_1 + 2 e_2 \} ,
\quad R = R^+ \cup -R^+ \\
& R^{\vee,+} = \{ 2e_1 - 3 e_2, 2e_1 - e_2, 3 e_2 - e_1, e_1, e_1 - e_2, e_2 \} ,
\quad R^\vee = R^{\vee,+} \cup -R^{\vee,+} \\ 
& \Delta = \{ e_1, e_2 \} ,\quad \alpha_0^\vee = e_1 ,\quad \alpha_0 = 2 e_1 + e_2 \\
& s_1 = s_{e_1}, \quad s_2 = s_{e_2} \quad s_0 = t_{\alpha_0} s_{\alpha_0} = 
  t_{e_1} s_{\alpha_0} t_{-e_1} : x \to x + \alpha_0 - 
  \inp{\alpha_0^\vee}{x} \alpha_0 \\
& W = \langle s_1 , s_2 | s_1^2 = s_2^2 = (s_1 s_2)^6 = e \rangle \cong D_6 \\
& S^\af = \{ s_0 ,s_1 ,s_2 \} ,\quad \Omega = \{e\} \\
& W^e = W^\af = \langle s_0, W_0 | s_0^2 = (s_0 s_2)^2 = (s_0 s_1)^3 = e \rangle 
\end{aligned}
\]
A generic parameter function $q$ for $\mc R (G_2)$ has two independent parameters
$q_1 = q (s_1)$ and $q_2 = q(s_2)$.

The group $W \cong D_6$ has six conjugacy classes: the identity, reflections 
associated to short roots, reflections associated to long roots, the rotation of
order two, rotations of order three and rotations of order six. Representatives
are $e,s_1,s_2,\rho_\pi = (s_1 s_2)^3, \rho_{2 \pi /3} = (s_1 s_2 )^2$ and 
$\rho_{\pi / 6} = s_1 s_2$. We determine the connected components of the 
extended quotient $T_\un /\!/ W$:
\[
\begin{array}{cccl}
w & T^w & Z_{D_6}(w) & T_\un^w / Z_{D_6}(w) \\
\hline 
e & T & D_6 & (S^1)^2 / D_6 \cong \text{ solid triangle} \vphantom{Q^{Q^Q}} \\
s_1 & \{ (1,t_2) : t_2 \in \C^\times \} & \langle s_1 , s_{3e_1 + 2e_2} \rangle &
S^1 / \langle s_{3 e_1 + 2 e_2} \rangle \cong [-1,1] \\
s_2 & \{ (t_1,1) : t_1 \in \C^\times \} & \langle s_2 ,s_{2 e_1 + e_2} \rangle &
S^1 / \langle s_{2 e_1 + e_2} \rangle \cong [-1,1] \\
\rho_\pi & \{ (a,b) : a,b \in \{\pm 1\} \} & D_6 & 2 \text{ points} \\
\rho_{2 \pi /3} & \{ (1,1), (\zeta_3,1), (\zeta_3^2,1) \} & C_6 = \langle \rho_{\pi / 3} \rangle &
2 \text{ points} \\
\rho_{\pi / 3} & \{ (1,1)\} & C_6 = \langle \rho_{\pi / 3} \rangle & 1 \text{ point}
\end{array} 
\]
Here $\zeta_3$ is a primitive third root of unity.
We see that every connected component of $T_\un /\!/ W$ is contractible, and
that its cohomology is zero in positive degrees and $\Z^8$ in degree zero.

The root datum $\mc R (G_2)$ is simply connected, so $W_t$ is a Weyl group for every
$t \in T$. This can also be checked directly: for $t \in T$ with $W_t = \{e\}$ or
$W_t$ generated by one reflection it is true. For all $t \in T$ not of that form,
$W_t$ contains a nontrivial rotation. All rotations (or their inverses) appear in the
above table, along with their fixpoints. We list the isotropy groups of those points:
\begin{align*}
& W_{(1,1)} = D_6 , \\
& W_{(\zeta_3,1)} = W_{(\zeta_3^2,1)} = \langle s_2, \rho_{2\pi/3} \rangle \cong S_3 ,\\
& W_{(-1,-1)} \cong W_{(-1,1)} \cong W_{(1,-1)} = \langle s_1 , s_{3e_1 + 2e_2} \rangle
\cong S_2 \times S_2 .
\end{align*}
We have checked all the conditions of Theorem \ref{thm:2.3}. 
By Corollary \ref{cor:2.5}, for every positive parameter function $q$:
\begin{equation}
K_0 (C_r^* (\mc R (G_2),q)) \cong \Z^8 ,\quad K_1 (C_r^* (\mc R(G_2),q)) = 0 .
\end{equation}

\appendix

\section{Some almost Weyl groups} 
\label{par:almost}

We study some finite groups which are almost Weyl groups. Such groups can arise as 
the component groups of unipotent elements of classical complex groups, and they 
play a role in the affine Hecke algebras associated to general Bernstein components
for classical $p$-adic groups \cite{Gol,Hei}.
The results from this appendix are only needed in Paragraph \ref{par:SOeven}.

Fix $n_1,n_2 ,\ldots, n_d \in \Z_{\geq 1}$ with $n_1 + \cdots + n_d = n$ and consider
\[
W'_{\vec n} := \big( W(B_{n_1}) \times \cdots \times W(B_{n_d}) \big) \cap W(D_n) .
\]
We use the convention that
$W(D_1)$ is the trivial group. The group $W'_{\vec n}$ acts on the root system
\[
D_{\vec n} := D_{n_1} \times \cdots \times D_{n_d} .  
\]
Let $\Delta_{\vec n}$ be the standard basis of $D_{\vec n}$ and let $\Gamma$ be the
stabilizer of $\Delta_{\vec n}$ in $W'_{\vec n}$. Since $W(D_{\vec n})$ acts
simply transitively on the collection of bases of $D_{\vec n}$, 
\begin{equation}\label{eq:1.43}
W'_{\vec n} = W(D_{\vec n}) \rtimes \Gamma . 
\end{equation}
Explicitly, the group $\Gamma \cong (\Z / 2\Z)^{d-1}$ is generated by the elements
$\epsilon^{(k)} \epsilon^{(k+1)}$ for $k=1,\ldots,d-1$, where $\epsilon^{(k)} = s_{e_{n_k}}$
is the reflection associated to the short simple root of $B_{n_k}$.

The Springer correspondence was extended to groups of this kind in
\cite{Kat,ABPS1}. Let $T$ be the diagonal torus of the connected complex group
\begin{equation}\label{eq:1.44}
G^\circ = SO_{2 n_1}(\C) \times \cdots \times SO_{2 n_d}(\C) . 
\end{equation}
Then $W'_{\vec n}$ acts naturally on $T$ and we recover $W(D_{\vec n})$ as the Weyl 
group of $(G^\circ,T)$. The Lie algebra of $T$ can be identified with the 
defining representation of 
\begin{equation}
W (B_{\vec n}) := W(B_{n_1}) \times \cdots \times W(B_{n_d}) .
\end{equation}
Since $\Gamma$ consists of diagram automorphisms of
$D_{\vec n}$, we can build the reductive group
\begin{equation}\label{eq:1.41}
G = G^\circ \rtimes \Gamma. 
\end{equation}
Then $W'_{\vec n}$ becomes the ``Weyl" group of this disconnected group:
\[
W'_{\vec n} = W(G,T) := N_G (T) / T . 
\]
For $u \in G^\circ$ unipotent let $\mc B^u = \mc B_{G^\circ}^u$ be the variety of Borel
subgroups of $G^\circ$ containing $u$. The group $Z_G (u)$ acts naturally on
$\mc B^u \times \Gamma$, and that induces an action of $A_G (u) = \pi_0 (Z_G (u) / Z(G))$
on $H^i (\mc B^u ;\C) \otimes \C [\Gamma]$. For $\rho' \in \Irr (A_G (u))$ we form
the $W'_{\vec n}$-representations
\begin{align*}
& H (u,\rho') = H_{A_G (u)} \big(\rho, H^* (\mc B^u ;\C) \otimes \C [\Gamma] \big) , \\
& \pi (u,\rho') = H_{A_G (u)} \big(\rho, H^{\mr{top}} (\mc B^u ;\C) \otimes \C [\Gamma] \big) .
\end{align*}
We call $\rho'$ geometric if $\pi (u,\rho) \neq 0$. Then \cite[Theorem 4.4]{ABPS1} says
that $\pi (u,\rho') \in \Irr (W'_{\vec n})$ and that this yields a bijection between 
$\Irr (W'_{\vec n})$ and the $G$-conjugacy classes of pairs $(u,\rho')$ with $u \in G^\circ$ 
unipotent and $\rho' \in \Irr (A_G (u))$ geometric.

The $W'_{\vec n}$-representations $H' (u,\rho')$, with $(u,\rho')$ as above, form another $\Z$-basis
of $R_\Z (W'_{\vec n})$. Indeed, this can be shown in the same way as for Weyl groups in 
\cite[Lemma 3.3.1]{Ree}, the input from \cite{BoMa} holds for $W'$ by \cite[Lemma 4.5]{ABPS1}.

For $P \subset \Delta_{\vec n}$ we define the standard parabolic subgroup
\[
W'_P := \langle s_\alpha : \alpha \in P \rangle \rtimes \mr{Stab}_\Gamma (P) . 
\]
As usual, a parabolic subgroup of $W'_{\vec n}$ is a conjugate of some $W'_P$.
Let $P_A$ be the standard basis of the union of the type $A$ root subsystems of $R_P$ and let
$P_B$ be the standard basis of the union of the type $B$ root subsystems of $\Q R_P \cap B_{\vec n}$.
(So $P_B$ need not be contained in $P$.) It is easily seen that
\begin{equation}\label{eq:1.42}
W'_P = W_{P_A} \times W_{P_B} \cap W(D_n ) = W_{P_A} \times W'_{{\vec n}_P} , 
\end{equation}
where ${\vec n}_P$ consists of the numbers $|P_B \cap B_{n_i}|$ which are nonzero.

All the above notions for $W'_{\vec n}$ have natural analogues for $W'_P$, which we indicate
by an additional subscript $P$. In particular \cite[Proposition 6.2]{Kat} entails that, 
as in \eqref{eq:1.35} and \eqref{eq:1.25}:
\[
\ind_{W'_P}^{W'_{\vec n}} \big( H_{W'_P} (u_P,\rho'_P ) \big) \cong
\Hom_{A_{G_P}(u_P)} \big( \rho_P , H^* (\mc B^{u_P} ;\C) \otimes \C [\Gamma] \big) .
\]

\begin{lem}\label{lem:1.9}
The parabolic subgroups of $W'_{\vec n}$ are precisely the isotropy groups of the
points of Lie$(T)$.
\end{lem}
\begin{proof}
Considering the standard representation of $W(B_{\vec n})$ on Lie$(T)$, we see that for 
any $y \in \mr{Lie}(T)$ the isotropy group $(W'_{\vec n} )_y$ is $W(B_{\vec n})$-conjugate
to $W(B_{\vec n})_Q \cap W(D_{\vec n})$, where $W(B_{\vec n})_Q$ is a standard parabolic
subgroup of $W(B_{\vec n})$. From \eqref{eq:1.42} we see that the group 
$W(B_{\vec n})_Q \cap W(D_{\vec n})$ equals $W'_P$ for $R_P = R_Q \cap D_{\vec n}$. Hence
every isotropy group $(W'_{\vec n} )_y$ is $W(B_{\vec n})$-conjugate to some standard
parabolic subgroup of $W'_{\vec n}$. Since the diagram automorphisms $\epsilon^{(k)}$ stabilize
the collection of parabolic subgroups of $W'_{\vec n}$ and $W(B_{\vec n})$ is generated by
$W(D_{\vec n})$ and the $\epsilon^{(k)}$, we conclude that $(W'_{\vec n} )_y$ is 
$W'_{\vec n}$-conjugate to a parabolic subgroup of $W'_{\vec n}$.
\end{proof}

With Lemma \ref{lem:1.9} we can define ellipticity in two equivalent ways. An element of
$W'_{\vec n}$ is elliptic if it is not contained in a proper parabolic subgroup, or
equivalently if it fixes a nonzero element of Lie$(T)$.
With these notions we can develop the elliptic representation theory of $W'_{\vec n}$, 
exactly as in \cite{Ree} and as in Paragraph \ref{par:Weyl}. In particular 
\eqref{eq:1.30} remains valid.

\begin{lem}\label{lem:1.10}
The group of elliptic representations $\overline{R_\Z}(W'_{\vec n})$ is torsion-free.
\end{lem}
\begin{proof}
We will follow the proof of Theorem \ref{thm:1.6}, with
the group $G^\circ$ from \eqref{eq:1.44}. Every Levi subgroup of $G^\circ$ can be described
by a $d$-tuple of partitions $\vec{\alpha} = \big( \alpha^{(1)},\ldots,\alpha^{(d)} \big)$.
The standard Levi subgroup associated to $\vec{\alpha}$ is
\[
G^\circ_{\vec{\alpha}} = \prod_{i=1}^d SO_{2 n_i}(\C)_{\alpha^{(i)}} =
\prod_{i=1}^d GL_{\alpha^{(i)}_1}(\C) \times \cdots \times GL_{\alpha^{(i)}_{d_i}}(\C)
\times SO_{2 (n_i - |\alpha^{(i)}|)}(\C) .
\]
(We note that sometimes several $P \subset \Delta$ are associated to one $\vec{\alpha}$,
as already for $SO_{2n}(\C)$.) We mimic \eqref{eq:1.41} by putting 
\[
\begin{aligned}
G_{\vec{\alpha}} & = G^\circ_{\vec{\alpha}} \rtimes \big\langle \epsilon^{(i)} \epsilon^{(j)} : 
|\alpha^{(i)}| < n_i \text{ and } |\alpha^{(j)}| < n_j \big\rangle \\
 & = \Big( \prod_{i=1}^d GL_{\alpha^{(i)}_1}(\C) \times \cdots \times GL_{\alpha^{(i)}_{d_i}}(\C)
\Big) \times S \Big( \prod_{i=1}^d O_{2 (n_i - |\alpha^{(i)}|)}(\C) \Big) .
\end{aligned}
\]
Then $W(G_{\vec{\alpha}},T) \cong W'_P$ for $P \subset \Delta$ corresponding to $\vec{\alpha}$.

The Bala--Carter classification says that the unipotent classes in 
$G^\circ$ can be pa\-ra\-me\-tri\-zed by $d$-tuples of bipartitions
$(\vec{\alpha} ,\vec{\beta})$ such that $2 |\alpha^{(i)}| + |\beta^{(i)}| = 2 n_i$, 
$\beta^{(i)}$ has only odd parts and all parts of $\beta^{(i)}$ are distinct. 
A typical $u$ in this conjugacy class is distinguished in the standard Levi subgroup 
$G^\circ_{\vec{\alpha}}$.

Like in \eqref{eq:1.45} and \eqref{eq:1.40}, let $G_{{\vec \alpha}''}$ be a standard Levi 
subgroup containing $u$. Then $u = u'' u'$ with $u'$ in a product of groups $GL_{n_k}(\C)$ and 
\[
u' \in S \big( \prod\nolimits_{i=1}^d O_{2 (n_i - |\alpha^{''(i)}|)}(\C) \big) =: H .
\]
The $GL$-factors and $u''$ do not contribute to $A_{G_{{\vec \alpha}''}}(u)$.

In the upcoming calculations we omit the case that $\vec{\beta}$ is empty, that case is
a bit different but can be handled in the same way.

With \cite[\S 13.1]{Car} we find that the quotient of $Z_H (u')$ by its unipotent radical is
\[
\prod_{i = 1}^d \prod_{j \text{ even}} Sp_{2 m^{'(i)}_j}(\C) \times \prod_{i = 1}^d 
\prod_{j \text{ odd}, \text{ not in } \beta^{(i)}} O_{2m^{'(i)}_j}(\C)
\times S \Big( \prod_{i = 1}^d \prod_{j \text{ odd}, \text{ in } \beta^{(i)}} 
O_{2m^{'(i)}_j + 1}(\C) \Big) .
\]
The component groups become
\[
A_{G_{{\vec \alpha}''}}(u) \cong A_H (u') \cong \Big( \prod_{i = 1}^d 
\prod_{j \text{ odd}, \text{ in } \alpha^{'(i)}, \text{ not in } \beta^{(i)}} \Z / 2 \Z \Big) 
\times S \Big( \prod_{i = 1}^d \prod_{j \text{ odd}, \text{ in } \beta^{(i)}} \Z / 2 \Z \Big) .
\]
In the same way as after \eqref{eq:1.31} we see that $\overline{R_\Z}(A_G (u)) = 0$
unless each $\alpha^{(i)}$ has only distinct odd terms, none of them appearing in $\beta^{(i)}$. 
For such $(\vec{\alpha},\vec{\beta})$ the maximal reductive quotient of $Z_G (u)$ simplifies to
\begin{equation}\label{eq:1.48}
\Big( \prod_{i = 1}^d \prod_{j \text{ odd}, \text{ in } \alpha^{(i)}} O_2 (\C) \Big)
\times S \Big( \prod_{i = 1}^d \prod_{j \text{ odd}, \text{ in } \beta^{(i)}} O_1 (\C) \Big) 
\end{equation}
and the component group becomes
\[
A_G (u) = \prod_{i = 1}^d \prod_{j \text{ odd}, \text{ in } \alpha^{(i)}} 
\Z / 2 \Z \times A \quad \text{with} \quad 
A = S \Big(  \prod_{i = 1}^d \prod_{j \text{ odd}, \text{ in } \beta^{(i)}} \Z / 2 \Z \Big) .
\]
Just as in \eqref{eq:1.36} we can calculate that $\overline{R_\Z}(A_G (u)) \cong R_\Z (A)$.
\end{proof}

With Lemmas \ref{lem:1.9} and \ref{lem:1.10} at hand the proof of Proposition 
\ref{prop:1.8} also becomes valid for $W'_{\vec n}$. Let us formulate
this somewhat more generally. Let $W'$ be a finite group which is a direct product of
a Weyl group and a number of groups of the form $W'_{\vec n}$. Let $G'$ be the corresponding
direct product of the groups called $G$ in \eqref{eq:1.47} and \eqref{eq:1.41}. We denote
the basis of the root system $R'$ underlying $W'$ by $\Delta'$, and the standard parabolic
subgroup associated to $P \subset \Delta$ by $W'_P$.

\begin{lem}\label{lem:1.11}
For every $w \in \mc C_P (W')$ there exists a pair $(u_{P,w},\rho'_{P,w})$ such that:
\begin{itemize}
\item $u_{P,w}$ is quasidistinguished unipotent in $G'_P$,
\item $\rho'_{P,w} \in \Irr (A_{G'_P}(u_{P,w}))$ is geometric,
\item the set
\[
\big\{ \ind_{W'_P}^{W'_{\vec n}} \big( H_P (u_{P,w},\rho'_{P,w}) \big) : 
P \in \mc P (\Delta_{\vec n}) / W'_{\vec n} , w \in \mc C_{P,\mr{ell}}(W'_{\vec n}) \big\} 
\]
forms a $\Z$-basis of $R_\Z (W')$.
\end{itemize}
\end{lem}
\begin{proof}
Let $(W'_i )_i$ be the indecomposable factors of $W'$, with root systems $R'_i$.
For every $P \subset \Delta'$:
\[
W'_P = \prod\nolimits_i W'_{P \cap R'_i} \quad \text{and} \quad 
R_\Z (W'_P) = \bigotimes\nolimits_i R_\Z (W'_{P \cap R'_i}) .
\]
Thus we reduce to the case of a single $W'_i$. If $W'_i$ is an irreducible Weyl group, then 
Proposition \ref{prop:1.8} applies immediately, so we may assume that $W'_i = W'_{\vec n}$.

Let $u \in G$ be unipotent and assume that $\overline{R_\Z}(A_G (u)) \neq 0$. From the 
proof of Lemma \ref{lem:1.10} we see that a maximal reductive subgroup of $Z_G (u)$
is of the form \eqref{eq:1.48}. For each $(i,j)$ with $j$ in $\alpha^{(i)}$ we pick an
element $t_{i,j} \in SO_2 (\C) \setminus \{\pm \matje{1}{0}{0}{1} \}$, all different. 
This gives a semisimple element
\[
t := \prod_{i=1}^d \big( \prod_{j \text{ in } \alpha^{(i)}} t_{i,j} \times 
\prod_{j \text{ in } \beta^{(i)}} 1 \big) \in Z_G (u)^\circ .
\]
Furthermore $t$ does not lie in any proper Levi subgroup of $G^\circ$
containing $u$, so $t u$ does lie in any proper Levi subgroup of $G^\circ$. Thus $u$
is quasidistinguished in $G$. 

Knowing this and Lemma \ref{lem:1.10}, the proof of Proposition \ref{prop:1.8} goes through.
\end{proof}

\end{document}